\documentclass[a4paper,reqno,10pt]{amsart}

\usepackage{graphicx}
\usepackage{psfrag}
\usepackage{perpage}
\usepackage{url}
\usepackage{color}
\usepackage{mathrsfs}
\usepackage{dsfont} % nice indicator function symbol (blackboard 1)
\usepackage{mathdots}
\usepackage{tikz}
\usepackage{young}

\usepackage{amsmath, amsfonts, amssymb, amsthm, amscd}

\usepackage[utf8]{inputenc}
\usepackage[T1]{fontenc}
\usepackage{microtype}

\usepackage[a4paper,scale={0.72,0.74},marginratio={1:1},footskip=7mm,headsep=10mm]{geometry}

\usepackage{hyperref}

\makeatletter
\def\@secnumfont{\bfseries\scshape}

\def\section{\@startsection{section}{1}%
  \z@{.7\linespacing\@plus\linespacing}{.5\linespacing}%
  {\normalfont\large\bfseries\scshape\centering}}

\def\subsection{\@startsection{subsection}{2}%
  \z@{.5\linespacing\@plus.7\linespacing}{-.5em}%
  {\normalfont\bfseries\scshape}}
  
\def\subsubsection{\@startsection{subsubsection}{3}%
  \z@{.5\linespacing\@plus.7\linespacing}{-.5em}%
  {\normalfont\scshape}}

\def\specialsection{\@startsection{section}{1}%
  \z@{\linespacing\@plus\linespacing}{.5\linespacing}%
  {\normalfont\centering\large\bfseries\scshape}}
\makeatother

\makeatletter

\renewenvironment{proof}[1][\proofname]{\par
\pushQED{\qed}%
\normalfont \topsep4\p@\@plus4\p@\relax
\trivlist
\item[\hskip\labelsep
\bfseries
#1\@addpunct{.}]\ignorespaces
}{%
\popQED\endtrivlist\@endpefalse
}
\makeatother

\makeatletter
\newcommand \Dotfill {\leavevmode \leaders \hb@xt@ 6pt{\hss .\hss }\hfill \kern \z@}
\makeatother

\makeatletter
\def\@tocline#1#2#3#4#5#6#7{\relax
  \ifnum #1>\c@tocdepth % then omit
  \else
    \par \addpenalty\@secpenalty\addvspace{#2}%
    \begingroup \hyphenpenalty\@M
    \@ifempty{#4}{%
      \@tempdima\csname r@tocindent\number#1\endcsname\relax
    }{%
      \@tempdima#4\relax
    }%
    \parindent\z@ \leftskip#3\relax \advance\leftskip\@tempdima\relax
    \rightskip\@pnumwidth plus4em \parfillskip-\@pnumwidth
    #5\leavevmode\hskip-\@tempdima
      \ifcase #1
       \or\or \hskip 1.65em \or \hskip 3.3em \else \hskip 4.95em \fi%
      #6\nobreak\relax
    \Dotfill
    \hbox to\@pnumwidth{\@tocpagenum{#7}}\par
    \nobreak
    \endgroup
  \fi}
\makeatother

\makeatletter
\def\l@section{\@tocline{1}{0pt}{1pc}{}{\scshape}}
\renewcommand{\tocsection}[3]{%
\indentlabel{\@ifnotempty{#2}{\ignorespaces#1 #2.\hskip 0.7em}}#3}
\def\l@subsection{\@tocline{2}{0pt}{1pc}{5pc}{}}

\def\l@subsubsection{\@tocline{3}{0pt}{1pc}{7pc}{}}

\makeatother

\numberwithin{equation}{section}

\newtheoremstyle{mytheorem}{.7\linespacing\@plus.3\linespacing}{.7\linespacing\@plus.3\linespacing}%
     {\itshape}%         Body font
     {}%         Indent amount (empty = no indent, \parindent = para indent)
     {\bfseries}% Thm head font (e.g. \bfseries, \scshape, \sffamily)
     {. }%        Punctuation after thm head
     {0.3ex}%     Space after thm head (\newline = linebreak)
     {\thmname{{\bfseries #1}}\thmnumber{ {\bfseries #2}}\thmnote{ (#3)}}  % Thm head spec

\theoremstyle{mytheorem}

\newtheorem{theorem}{Theorem}[section]
\newtheorem{lemma}[theorem]{Lemma}
\newtheorem{proposition}[theorem]{Proposition}

\newtheorem{remark}[theorem]{Remark}
\newtheorem{definition}[theorem]{Definition}

\newtheorem{conjecture}[theorem]{Conjecture}

\newcommand{\bbE}{{\ensuremath{\mathbb E}} }

\newcommand{\bbP}{{\ensuremath{\mathbb P}} }

\newcommand{\bbR}{{\ensuremath{\mathbb R}} }

\newcommand{\bbT}{{\ensuremath{\mathbb T}} }

\newcommand{\cE}{{\ensuremath{\mathcal E}} }

\newcommand{\ga}{\alpha}

\newcommand{\gd}{\delta}

\newcommand{\gl}{\lambda}

\renewcommand{\tilde}{\widetilde}          % wider `tilde'
\DeclareMathSymbol{\leqslant}{\mathalpha}{AMSa}{"36} % nicer `smaller or equal'
\DeclareMathSymbol{\geqslant}{\mathalpha}{AMSa}{"3E} % nicer `larger or equal'
\DeclareMathSymbol{\eset}{\mathalpha}{AMSb}{"3F}     % nicer `emptyset'
\newcommand{\dd}{\text{\rm d}}             % a straight d for differentials

\newcommand{\R}{\mathbb{R}}
\newcommand{\C}{\mathbb{C}}

\newcommand{\PEfont}{\mathrm}

\newcommand{\p}{\ensuremath{\PEfont P}}
\newcommand{\e}{\ensuremath{\PEfont E}}
\newcommand{\E}{\e}
\renewcommand{\P}{\p}

\newcommand{\ind}{\mathds{1}}

\renewcommand{\epsilon}{\varepsilon}
\renewcommand{\theta}{\vartheta}
\renewcommand{\rho}{\varrho}

\newenvironment{myenumerate}{%
\renewcommand{\theenumi}{\arabic{enumi}}%
\renewcommand{\labelenumi}{{\rm(\theenumi)}}%
\begin{list}{\labelenumi}
	{%
	\setlength{\itemsep}{0.4em}%
	\setlength{\topsep}{0.5em}%
	\setlength\leftmargin{2.45em}%
	\setlength\labelwidth{2.05em}%
	\setlength{\labelsep}{0.4em}%
	\usecounter{enumi}%
	}%
	}%
{\end{list}
}

{\end{list}
}

{\end{list}
}

{\end{myenumerate}}

\newenvironment{myitemize}{%
\begin{list}{$\bullet$}%
 	{%
	\setlength{\itemsep}{0.4em}%
	\setlength{\topsep}{0.5em}%
	\setlength\leftmargin{2.45em}%
	\setlength\labelwidth{2.05em}%
	\setlength{\labelsep}{0.4em}%
	}%
	}%
{\end{list}}

\renewenvironment{itemize}{
\begin{myitemize}}%
{\end{myitemize}}

\MakePerPage[2]{footnote} % restarts footnote counter at every new page

\newcommand{\Ai}{\mathcal{A}i}
\newcommand{\grsk}{\text{gRSK}}
\newcommand{\gpng}{\text{gPNG}}
\newcommand{\Whitt}{
\Psi^{ \,
\begin{tikzpicture}[scale=.04]
\draw[thick] (1,1)--(1,6);
\draw[thick] (1,6)--(4,6);
\draw[thick] (4,6)--(6,3);
\draw[thick] (6,3)--(4,3);
\draw[thick] (4,3)--(4,1);
\draw[thick] (4,1)--(1,1);
\end{tikzpicture}
 }}

\newcommand{\polygon}{
\begin{tikzpicture}[scale=.03]
\draw[thick] (0,0)--(0,9);
\draw[thick] (0,9)--(9,9);
\draw[thick] (9,9)--(9,6);
\draw[thick] (9,6)--(6,6);
\draw[thick] (6,6)--(6,3);
\draw[thick] (6,3)--(3,3);
\draw[thick] (3,3)--(3,0);
\draw[thick] (3,0)--(0,0);
\end{tikzpicture}
 }
 
 \newcommand{\Ypolygon}{
\begin{tikzpicture}[scale=.15]
\draw[help lines] (1,0) grid (4,2);
\draw[help lines] (1,2) grid (6,4);
\draw[help lines] (1,4) grid (10,7);

\end{tikzpicture}
}

\newcommand{\Ppattern}{
\begin{tikzpicture}[scale=.15]
\node at (0.2,0) {\small{{\bf $z_{11}$}}};
\node at (2,-2.8) {\small{{\bf $z_{21}$}}};
\node at (-2,-2.8) {\small{{\bf $z_{22}$}}};
\draw[thick][dotted] (-3,-4.5)--(-4,-6);
\draw[thick][dotted] (3,-4.5)--(4,-6);
\node at (4.8,-7) {\small{{\bf $z_{n1}$}}};
\node at (-4.5,-7) {\small{{\bf $z_{nn}$}}};
\end{tikzpicture}
}

\newcommand{\Qpattern}{
\begin{tikzpicture}[scale=.15]
\node at (0.2,0) {\small{{\bf $z'_{11}$}}};
\node at (2,-2.8) {\small{{\bf $z'_{21}$}}};
\node at (-2,-2.8) {\small{{\bf $z'_{22}$}}};
\draw[thick][dotted] (-3,-4.5)--(-4,-6);
\draw[thick][dotted] (3,-4.5)--(4,-6);
\node at (4.8,-7) {\small{{\bf $z'_{n1}$}}};
\node at (-4.5,-7) {\small{{\bf $z'_{nn}$}}};
\end{tikzpicture}
}

\def\dd{\mathrm{d}}

\def\sfA{\mathsf A}
\def\sfC{\mathsf C}
\def\sfT{\mathsf T}

\def\sfv{\mathsf v}
\def\sfw{\mathsf w}
\def\sff{\mathsf f}
\def\sfF{\mathsf F}
\def\sfi{\mathsf i}

\newcommand\btau{\boldsymbol\tau}
\newcommand\bZ{\boldsymbol Z}

\newcommand\bx{\boldsymbol x}

\usepackage{bbold} % allows to use \mathbb{} with greek letters
\renewcommand{\ind}{\mathbb{1}}

\begin{document}

\title
[variants of gRSK, gPNG and log-gamma polymer]
{Variants of geometric RSK, geometric PNG and the multipoint distribution of the log-gamma polymer}

\begin{abstract} We show that the reformulation of the geometric Robinson-Schensted-Knuth (gRSK) correspondence via local moves,  introduced in \cite{OSZ14} can be extended to cases where the input matrix is replaced by more general polygonal, Young-diagram-like, arrays of the form $\polygon$. We also show that a rearrangement of the sequence of the local moves gives rise to a geometric version of the polynuclear growth model (PNG). These reformulations are used to obtain integral formulae for the Laplace transform of the joint distribution of the point-to-point partition functions of the log-gamma polymer at different space-time points. In the case of two points at equal time $N$ and space at distance of order $N^{2/3}$, we show formally that the joint law of the partition functions, scaled by $N^{1/3}$, converges to the two-point function of the Airy process. 
\end{abstract}

\author[V.L. Nguyen]{Vu-Lan Nguyen }
\address{Departement de Mathematique\\Paris 7, Diderot\\
 }
\email{vlnguyen@math.univ-paris-diderot.fr}

\author[N.Zygouras]{Nikos Zygouras}
\address{Department of Statistics\\
University of Warwick\\
Coventry CV4 7AL, UK}
\email{N.Zygouras@warwick.ac.uk}

\date{\today}


\maketitle
\section{Introduction}
In recent years there have been several important developments in the study of random polymer models, interacting particle systems and related statistical mechanics models, which appear to have an exactly solvable underlying structure. The object that has been driving these investigations is the Kardar-Parisi-Zhang (KPZ) equation and the universality class that this determines. We will not embark on an interlude of the ubiquitous nature of the KPZ, but instead we will refer to various reviews including 
\cite{KS92}, \cite{C12}. Even though non gaussian fluctuations, governed by a $1/3$ exponent, were predicted in \cite{KPZ86} for the systems belonging in this class, the more profound structure begun to be recognised much later. Fundamental towards this understanding were the contributions of Baik-Deift-Johansson \cite{BDJ99} and Johansson \cite{J00} where the combinatorial structure of certain objects in the KPZ universality was studied and surprising connections to random matrix theory were revealed. In particular, \cite{BDJ99} studied the distribution of the longest increasing subsequence in a random permutation and \cite{J00} studied the distribution of a directed last passage percolation with geometric disorder and in both cases the exponent $1/3$ for the fluctuations  and their convergence to the Tracy-Widom GUE distribution was established.
 \vskip 2mm
The two models studied in \cite{BDJ99}, \cite{J00} are fundamentally related to each other via the Robinson-Schensted-Knuth (RSK) correspondence. The special (and historically prior) application of this correspondence, the Robsinson-Schensted (RS), is a bijective mapping between a permutation of integers to a pair of Young tableaux of same {\it shape} (one which is {\it standard} and one {\it semi-standard}), in which the length of the first row of each tableau equals the length of the longest increasing subsequence in the permutation and more generally the sum of the first $k$ rows equals the maximal length of the union of $k$ disjoint increasing subsequences.
 The RS correspondence was extended by Knuth as a bijection between matrices with nonnegative
 integer entries and pairs of Young tableaux. If $(w_{ij})_{1\leq i \leq m, \,1\leq j \leq n}$ is an input matrix, then the length of the first row of the Young tableaux equals the passage time in directed last passage percolation 
 given by
 \begin{align}\label{def:LPP}
 \tau_{m,n}:=\max_{\pi \in \Pi_{m,n}}\sum_{(i,j)\in \pi} w_{ij},
 \end{align}
 where $\Pi_{m,n}$ is the set of down-right paths from $(1,1)$ to $(m,n)$.
 \vskip 2mm
 Combined with the works of Okounkov \cite{O01} and Pr\"ahofer-Spohn \cite{PS02}, which made connections to integrable systems and the 
 fermionic formalism,  many statistical models of a random growth nature where analysed and their asymptotic fluctuations where established, see \cite{J05} for an account. However, these developments were restricted to {\it zero temperature} models, even though the universal KPZ/Tracy-Widom fluctuations were also expected for
 such models in positive temperature. 
 In \cite{O12},\cite{COSZ14}, \cite{OSZ14} it was realised that the geometric lifting of the RSK correspondence (gRSK)
 introduced by Kirillov \cite{K01} (who had originally named it ``Tropical RSK''-see also \cite{NY04} for a matrix approach) could be used as the stepping stone to make progress in the analysis of
 fluctuations of statistical models in positive temperature. At the same time the theory of Macdonald Processes of Borodin-Corwin \cite{BC14}, based on the theory of Macdonald functions \cite{Mac99}, provided another route to this task analogous to the theory of Schur processes of Okounkov-Reshetikhin \cite{OR03}.
 \vskip 2mm
In this article we work within and further elaborate on the framework of gRSK with the purpose of establishing extensions of this correspondence and 
draw conclusions on the correlation properties of a directed polymer model in random medium (DPRM). The DPRM is the positive temperature version of last passage percolation \eqref{def:LPP} and its partition function is given by
 \begin{align}\label{def:DPRM}
 Z_{m,n}:=\sum_{\pi \in \Pi_{m,n}}\prod_{(i,j)\in \pi} w_{ij},
 \end{align}
 where again $\Pi_{m,n}$ is the set of down-right paths from $(1,1)$ to $(m,n)$. A distribution of the random variables $(w_{ij})$ that leads to an exactly solvable partition \eqref{def:DPRM}
 is the inverse gamma distribution: the variables $w_{ij}$ are considered to be independent and the distribution of each one 
 is given by $\Gamma(\ga_i+\hat\ga_j)^{-1} w_{ij}^{-\ga_i-\hat\ga_j}\exp(-1/w_{ij}) \,\dd w_{ij}/w_{ij}$. A DPRM with such an underlying randomness is now known as the 
 {\it log-gamma} polymer. It was introduced by Sepp\"al\"ainen in \cite{S12} who, using ideas from particle systems, was able to obtain fluctuation bounds of the 
 correct order $N^{1/3}$ for $\log Z_{N,N}$. In \cite{COSZ14}, \cite{OSZ14} an exact contour integral formula for the Laplace transform of $Z_{n,m}$ was obtained
via the use of gRSK. This formula was later turned into a Fredholm determinant  in \cite{BCR13} from which it was deduced that for certain constants $c_1,c_2$
\[
\frac{\log Z_{N,N}-c_1N}{c_2 N^{1/3}}\xrightarrow[N\to\infty]{(d)} F_{GUE},
\]
where $F_{GUE}$ is the Tracy-Widom GUE distribution characterising the asymptotic fluctuations of the largest eigenvalue of a Gaussian Unitary Ensemble of random matrices. 
\vskip 2mm
Besides the above result, which corresponds to the one-point distribution of a polymer starting at $(1,1)$ and conditioned to end at $(N,N)$, one is also interested to know the joint distribution of two or more partition functions of polymers starting at the same point $(1,1)$ and ending at different locations $(N_1,M_1),(N_2,M_2),$... A particular case is when the ending points lie on the same (one dimensional) plane $\{(m,n)\colon m+n=N\}$. Using the space-time coordinate system $(x,t)=(i-j, i+j)$, which is the more natural system for one-dimensional simple random walks, the point-to-point partition functions $\big(Z_{m,n} \colon m+n=N\,\big)$ correspond to random walks starting at zero and ending up at different points after time $N$. Thus the joint distribution of these partition functions is known as the equal-time-distribution. In the case of last passage percolation the joint law was obtained in \cite{J03} via a discrete version of the polynuclear growth model (PNG) of Pr\"ahofer and Spohn \cite{PS02}. PNG can be actually viewed as a different construction of the RSK correspondence and its geometric lifting will be one of the objectives of the present article. It can also be viewed as an ensemble of non intersecting paths, which 
after appropriate space-time scaling converges to a continuum non intersecting line ensemble called the {\it Airy line ensemble }.
In this paper we will concern ourselves with the topmost line in this ensemble, which is known as the {\it Airy process}. The Airy process, denoted by
$\Ai(\cdot)$
 is a continuous, stationary process whose one dimensional distribution is the Tracy-Widom GUE. In \cite{J03} it was established that
when the weights $(w_{ij})$ have a geometric distribution, i.e. $\bbP(w_{ij}=m)=(1-q)q^m, m\geq 0$ then 
\begin{align}\label{2/3-LPP}
\Bigg\{
\frac{\tau_{N+c_0 N^{2/3}\,x, N-c_0 N^{2/3}\,x}-c_1N}{c_2N^{1/3}}
\Bigg\}_{x\in\bbR} \xrightarrow[N\to\infty]{(d)} \big\{ \Ai(x)-x^2\big\}_{x\in\bbR},
\end{align} 
for appropriate constants $c_0,c_1,c_2$ depending on the parameter $q$.
 \vskip 2mm
The analogue result for the joint distribution of the partition functions $Z_{m,n}$ of the log-gamma polymer had not been achieved, so far, the reason being that no 
formula for the joint Laplace transform had been obtained. In this work we obtain such formulae. We do this by establishing two variants of the geometric RSK correspondence, which we think are of independent interest:
\begin{itemize}
\item[1.] The first one is an extension of gRSK in situations where the input array is not a matrix, but rather a polygonal shape of the form $\polygon$. This is done in Section \ref{sec:gRSKpoly} by extending the reformulation of gRSK in terms of local moves introduced in \cite{OSZ14}. 
\item[2.]
The second one is the geometric lifting of PNG, which we will refer to as gPNG. gPNG appears to put the local moves of \cite{OSZ14} under an intuitive and transparent perspective, which may be useful in other situations, as well. The construction of gPNG and its properties are described in Section \ref{sec:gPNG}.
\end{itemize}
Each construction leads to seemingly different integral formulae, cf. Theorems \ref{thm:k-point} and \ref{Gamma_meas}, for the joint Laplace transform of an arbitrary number of log-gamma partition functions corresponding to different space-time points. These formulae have a structure that is reminiscent of Givental's integral formula for the $GL_n(\bbR)-$Whittaker function, but they are nevertheless different (we do not attempt to directly check that they indeed coincide). We also point that the formulae are valid for arbitrary points, not necessarily at equal time (but some restrictions exist, that exclude the {\it strict two-time} points, 
cf. the conditions on $(m_1,n_1)$ and $(m_2,n_2)$
in Theorem \ref{thm:2-point}; for progress on the two-time distribution in a last percolation model see
\cite{J15} and for a physics approach in the case of the continuum polymer in \cite{D13b}). Using the Plancherel theorem for $GL_n(\bbR)$-Whittaker functions, we are able to rewrite the
 formula for the Laplace transform of two partition functions at equal time as a contour integral, cf. Theorem \ref{thm:2-point}. 
 \vskip 2mm
 To put things into context let us show the formula for the joint Laplace transform of $Z_{(m_1,n_1)}, Z_{(m_2,n_2)}$, which in the case $m_1\leq n_1, m_2\geq n_2$  
 writes as
\begin{align}\label{eq:2p_intro}
 \bbE\Big[e^{-u_1 Z_{(m_1,n_1)}-u_2Z_{(m_2,n_2)}}\Big] 
 = &\int_{(\ell_\delta)^{m_1} } \dd \lambda
    \,\,s_{m_1}(\lambda)  \prod_{1\leq i,i'\leq m_1} \Gamma(-\ga_i+\gl_{i'})
     \prod_{i=1}^{m_1}\frac{u_1^{-\gl_i}\, \prod_{j=n_2+1}^{n_1}\Gamma(\gl_i+\hat\ga_j) }{ u_1^{-\ga_i}\, \prod_{j=n_2+1}^{n_1}\Gamma(\ga_i+\hat\ga_j) } \notag\\
&\,\times   \int_{(\ell_{\delta+\gamma})^{n_2} } \dd \mu  
    \,\,s_{n_2}(\mu)  \prod_{1\leq j,j'\leq n_2} \Gamma(-\hat\ga_j+\mu_{j'}) 
    \prod_{j=1}^{n_2}\frac{u_2^{-\mu_j} \, \prod_{i=m_1+1}^{m_2}\Gamma(\mu_j+\ga_i)  }{u_2^{-{\hat\ga_j}} \,  \prod_{i=m_1+1}^{m_2}\Gamma(\hat\ga_j+\ga_i) } \notag\\
 &\,\,   \quad\times \prod_{\substack{1\leq i\leq m_1 \\ 1\leq j\leq n_2 }} \frac{\Gamma(\gl_i+\mu_j)}{\Gamma(\ga_i+\hat\ga_j)}\,,
  \end{align} 
  where $\ell_\delta, \ell_{\delta+\gamma}$ are the vertical lines $\delta+\iota \bbR$ and $\delta+\gamma+\iota \bbR$ in the complex plane, with $\delta,\gamma>0$ such that $\ga_i<\gd$ for all $i$ (shortly written as $\ga<\gd$) and $\hat\ga_j<\gd+\gamma$ for all $j$ (shortly written as $\hat\ga<\gd+\gamma$), and $s_n(\lambda)=\prod_{1\leq i\neq j \leq n} \Gamma(\gl_i-\gl_j)^{-1}$ is the so-called Sklyanin measure.
  Unless otherwise stated we will keep the convention that integrals along vertical lines in the complex plane are traced upwards.
  \vskip 2mm
 Those familiar with the formula for the Laplace transform obtained for a single partition function in \cite{COSZ14}, \cite{OSZ14} will observe certain similarities to the
 above formula. In fact, if in the last product of gamma functions, $\prod \Gamma(\gl_i+\mu_j)/\Gamma(\ga_i+\hat\ga_j)$, one replaced the $\mu_j$ variables 
 with $\hat\ga_j$ and combined with the $\dd \gl$ integral, he/she would obtain the formula for the Laplace transform of $Z_{(m_1,n_1)}$. Similarly, if one replaced the $\gl_i$
 variables with $\ga_i$ and then combined with the $\dd \mu$ integral, he/she would obtain the Laplace transform for $Z_{(m_2,n_2)}$. 
 Bearing also in mind that two polymers, one from $(1,1)$ to $(m_1,n_1)$ and another from $(1,1)$ to $(m_2,n_2)$
  can only interact (via intersection) within the rectangle $1\leq i\leq m_1, 1\leq j\leq n_2$ (if $m_1\leq m_2$ and $n_2\leq n_1$) ,
   we are lead to see that this cross product reflects in a sense the
 correlation of the two partition functions.  
 \vskip 2mm
  Having this formula at hand, we work in Section \ref{sec:Airy} 
  towards obtaining convergence of the joint law of two partition functions to the two-point function of the Airy process, in the
  same spirit as \eqref{2/3-LPP}. Our first attempt was to write the above formula as a Fredholm determinant. This, however, is a non-trivial task and certainly demands further investigation, in order to understand the (what we believe that exists) underlying determinantal structure. We therefore followed an alternative route, which consisted in expressing
  \eqref{eq:2p_intro} as Fredholm-like series. Further nontrivial manipulation of the latter expression reveals a block determinantal structure that points towards the
  Airy process. Even though, at this stage we can pass in the asymptotic limit on each individual term of the series and identify it with the corresponding terms in the Airy
  expansion, a cross term arising from the product $\prod \Gamma(\gl_i+\mu_j)/\Gamma(\ga_i+\hat\ga_j)$ in \eqref{eq:2p_intro} makes it difficult to justify the uniformity of this limiting procedure. We discuss this point a bit more in Section \ref{sec:road}. For this reason our derivation of the convergence of the joint distribution to the two-point function of the Airy process is formal
   (cf. Conjecture \ref{thm:claim}, Section \ref{sec:Airy}). Nevertheless, we believe that our analysis unravels 
   an interesting structure of the contour integral formulae.
  \vskip 2mm 
Our exact formulae for the joint law of the log-gamma polymer appear to be the first such mathematically rigorous derivation for positive temperature polymers. 
We should, though, mention that in the physics literature formulae have been provided for the joint law of the continuum directed polymer in relation to the delta-Bose gas
 with the use of the coordinate Bethe ansatz. That was first done by Prohlac and Spohn \cite{PS11a}, \cite{PS11b}
 and then, again by the use of Bethe ansatz but via different manipulations
of the formulae, by Dotsenko \cite{D13}, \cite{D14}. Although for the same object, it was not obvious that the formulae provided were coinciding.
 This was later checked by Imamura-Sasamoto-Spohn in \cite{ISS13}. 
 Our method is completely different to these approaches, since we investigated, instead, the combinatorial structure of the log-gamma polymer.
Naturally, the formulae we obtain are different but it is worth noting that cross terms consisting of products of various gamma functions appear in all situations. 
Mathematically rigorous formulae for joint moments for related particle models, q-Push(T)ASEP's, have been obtained, see for example
\cite{CP15}. In principle these can degenerate to formulae for joint moments of polymer partition functions $Z_{(m,n)}$ for various points $(m,n)$ which would,
 however, need to share the same coordinate $n$.
Finally, we close this review of related results by mentioning the
work of Ortmann-Quastel-Remenik \cite{OQR15a}, \cite{OQR15b}. There they obtain formulae for the one point distribution (as opposed to our treatment of two-point distributions) 
of the asymmetric exclusion process with flat and half-flat initial data. From these, passing to a weakly asymmetric limit they deduce the law for the one-point 
function of the point to plane continuum directed polymer. Their method is based again on the coordinate Bethe ansatz. Motivated by related work of Lee 
\cite{L10}, this method allowed them to make a guess for the above mentioned formulae. Cross terms (which this time are not given by products of gamma functions)
appear also in this situation, preventing them from rigorously deriving the convergence to Tracy-Widom GOE.
\vskip 2mm
The article is organised as follows: In Section \ref{sec:onRSK}, we start by making a quick review of the main objects of RSK (Section \ref{sec:RSKrem}) 
and of geometric RSK (Section \ref{section:localmoves})
focusing on the definition and properties of the local moves. We also give a quick reminder of Whittaker functions and the properties that we will be using.
In Section \ref{sec:gRSKpoly} we extend the geometric RSK to the case of polygonal input arrays and we use its properties 
 to provide the first formulae for the joint Laplace transform of several point-to-point partition functions. We also use the Plancherel theory for Whittaker
function to obtain contour integral formulae for the Laplace transform of partition functions at equal time. In Section \ref{sec:gPNG} we construct the geometric PNG and deduce 
its main properties. In Section \ref{sec:Airy} we elaborate on the formulae for the joint Laplace transform for two partition functions at equal time and we show formally that
their joint distribution converges to the two-point function of the Airy process.

\section{Geometric RSK on polygonal patterns}\label{sec:onRSK}
The geometric RSK correspondence was introduced by Kirillov \cite{K01} as a bijective mapping
 $$\grsk:(\R_{>0})^{n\times m}\to (\R_{>0})^{n\times m},$$
between positive matrices, via a geometric lifting (i.e. $(\max,+)$ 
replaced by $(+,\times)$) of the Berenstein-Kirillov  \cite{BK95}
implementation of RSK in terms of piecewise linear transformations.
 Kirillov observed in his paper a number of interesting properties many of which deserve further investigation. A matrix algebra framework was 
set  by Noumi and Yamada~\cite{NY04} to accommodate Kirillov's theory and this framework was instrumental for \cite{COSZ14}.
 In \cite{OSZ14} gRSK was formulated via a sequence of `local moves' on matrix elements and this formulation was important in 
  exhibiting its volume preserving properties. 
  
  \subsection{Reminder of RSK, Young tableaux, Gelfand-Tsetlin patterns}\label{sec:RSKrem}
  The Knuth correspondence takes a matrix $(w_{ij})_{m\times n}$ of nonnegative integers and maps it to a pair of Young tableaux $(P,Q)$. A Young tableau is a collection of boxes arranged in left-justfied rows of lengths $\gl_1\geq \gl_2 \geq \cdots$. The boxes are filled with the numbers $1,2,...,n$, such that they are
   weakly increasing across rows (from left to right) and strictly increasing along columns (from top to bottom) 
   (this is the defining property of a {\it semi-standard} tableau).
  It is useful to
  think of the input matrix $(w_{ij})_{m\times n}$ as a sequence of words $(1^{w_{11}} 2^{w_{12}}\cdots n^{w_{1n}} )
  \cdots (1^{w_{m1}} 2^{w_{m2}}\cdots n^{w_{mn}}) $ where $j^{w_{ij}}$ denotes a sequence of $w_{ij}$ letters $j$ in the $i^{th}$ word. Each letter in this sequence is inserted in the $P$ tableau via the Robinson-Schensted row insertion. That is, a letter $j$ attempts to be inserted in the first row of the $P$ tableau by first bypassing all the boxes that contain numbers less or equal to $j$ and then occupying the
   first box that was occupied by a letter greater than $j$. If all boxes contain letters which are less or equal to $j$, then $j$ is appended at the end of the row. Otherwise, the letter $k>j$, that is replaced by $j$, is removed from the first row and tries to be inserterted in the second row in the same manner. This bumping procedure continues, with a box being eventually added at the end of a row in the $P$ tableau,
   filled with the letter $j$, and at the same time 
  another box is added at the same location in the $Q$ tableau filled with the letter $i$. 
  One of the consequences of this algorithm is that the $P$ and the $Q$ tableaux have the same row lengths 
 $\gl_1\geq \gl_2\geq \cdots$ and the vector of row lengths $(\gl_1,\gl_2,...)$ is called the {\it shape} of the tableau. So $P$ and $Q$ have the same shape.

Let us note that there are extensions of the RSK correspondence to the case where the input matrix $(w_{ij})_{m\times n}$ is replaced by 
general polygonal, Young-diagram-like arrays of the form $\polygon$. Such extensions are related to {\it Fomin growth diagrams} \cite{F86}; see also \cite{Kr06}
for further discussions around these constructions. 

A useful  way to record Young tableaux is via the Gelfand-Tsetlin patterns. This is a triangular array of numbers $(z_{ij})_{1\leq i \leq j \leq n}$ and the correspondence to Young tableaux is via
\[
z_{ij} = \sum_{r=j}^i\sharp \{ \text{of $r$'s in row $j$ of the Young tableau} \}, \qquad \text{for} \,\, 1\leq j\leq i\leq n.
\]
Note that $z_{ni}=\gl_i$, the length of the $i^{th}$ row in the Young tableau.
Given the definition of a Young tableaux, one obtains the interlacing conditions for the Gelfand-Tsetlin patterns
\begin{align}\label{GT_interlace}
z_{i+1,j+1}\leq z_{ij} \leq z_{i+1,j}, \qquad \text{for} \,\, 1\leq j \leq i \leq n.
\end{align}
%It is worth pointing out that these interlacing conditions \blue{underly} the determinantal structure in the zero temperature setting. However, they fail in the positive temperature setting of gRSK making the underlying structure much more difficult to reveal.
 \vskip 2mm
 The mappings we just described are depicted in the following diagram (for simplicity, we depict them in the special case of a square input matrix, i.e. $m=n$, where both Gelfand-Tsetlin (GT)
  patterns have a triangular shape)
 \begin{align}\label{RSKmaps}
&\\
& (w_{ij})_{n\times n} 
 \xrightarrow{\text{RSK}}\,\,
 \Ypolygon _{(P)} \,\,  \Ypolygon _{(Q)} 
% \ydiagram{4, 3, 1}_{(P)}\,\, \,
 %\ydiagram{4, 3, 1}_{(Q)}
 \xrightarrow{\,\,\text{GT}\,\,}\,\,
 \Ppattern \,\, \Qpattern 
 \longrightarrow
 \begin{bmatrix}
    z_{nn} & \dots  & z'_{11} \\
    \vdots & \ddots  & \vdots \\
    z_{11} & \dots   & z_{n1}
\end{bmatrix}
\notag
 \end{align}
 Even though not depicted, 
 the boxes of Young diagrams $(P,Q)$ are understood to be filled with numbers weakly increasing along rows and strictly increasing along columns.
 In the last mapping we glued together the Gelfand-Tsetlin patterns in an $n\times n$ matrix.
 This is possible since the $P$ and $Q$ tableaux have the same shape and hence the corresponding patterns have the same bottom row, which is placed on the diagonal.
 The above sequence of mappings explains the reason why RSK can be viewed as a mapping between matrices.
 \vskip 2mm
 Let us now assume that we have inserted the words 
 $(1^{w_{11}} 2^{w_{12}}\cdots n^{w_{1n}} )
  \cdots (1^{w_{m-1,1}} 2^{w_{m-1,2}}\cdots n^{w_{m-1,n}}) $ and let us denote by $(z_{ij})$ the variables of the Gelfand-Tsetlin pattern after these insertions.
  We next want to insert the word $(1^{w_{m,1}} 2^{w_{m-1,2}}\cdots n^{w_{m,n}}) $ and let us denote by $(\tilde{z}_{ij})$ the variables of the new Gelfand-Tsetlin pattern.
    After this insertion, the number of $1$'s in the first row will be 
 \begin{align}\label{first_row_1} 
 \tilde{z}_{11}=z_{11}+w_{m1}.
 \end{align}
  This insertion will bump a number of $2$'s, which will then be inserted in the second row. The
number of bumped $2$'s will be equal to 
\begin{align}\label{twobump}
\min\big(z_{21}-z_{11}, w_{m1}\big)  = \min\big(z_{21}-z_{11}, \tilde{z}_{11}-z_{11}\big) = - \max \big( -(z_{21}-z_{11}), -(\tilde{z}_{11} -z_{11})\big),
\end{align}
 depending on whether the number of $1$'s inserted in the previous step (which equals $w_{m1}$) bumped all the existing $2$'s in the first row 
 (which equals $z_{21}-z_{11}$) or not. Moreover, a number $w_{m2}$ of 
 $2$'s will be inserted in the first row, at the end of a sequence of $1$'s of length $\max\big(z_{21}, \tilde{z}_{11}\big)$ (bumping at the same time a number of $3$'s) leading to the update
 \begin{align}\label{first_row_2}
 \tilde{z}_{21}=w_{m2}+\max\big(z_{21}, \tilde{z}_{11}\big).
 \end{align}
 The bumped $2$'s will be inserted in the second row, thus increasing the number of $2$'s there to
 \begin{align}\label{second_row_2}
 \tilde{z}_{22}= z_{22}  - \max \big( -(z_{21}-z_{11}), -(\tilde{z}_{11} -z_{11})\big).
 \end{align}
 A similar procedure of inserting and bumping with the remaining letters carries on and can be described by analogous set of piecewise linear transformations (local moves).
 
\subsection{Geometric RSK via local moves and Whittaker functions.}\label{section:localmoves}
The local move formulation of gRSK can be viewed as a geometrization (i.e. replace $(\max,+,-)$ with $(+,\times, /\,)$) 
of the local moves described in the previous section, cf. \eqref{first_row_1}, \eqref{twobump}, \eqref{first_row_2}, \eqref{second_row_2}.  Let us recall how it works:
 \vskip 2mm
 Let an input matrix $X=(x_{ij})\in(\R_{>0})^{n\times m}$.  Define $l_{11}$ to be the identity map. For $2\le i\le n$ and $2\le j\le m$, 
define $l_{i1}$ to be the mapping that replaces
the element $x_{i1}$ by $x_{i-1,1}x_{i1}$ and $l_{1j}$ to be the mapping that replaces the element $x_{1j}$ by 
$x_{1,j-1}x_{1j}$, while leaving all other elements of $X$ unchanged.
For each $2\le i\le n$ and $2\le j\le m$ define 
the mapping $l_{ij}$ which replaces the submatrix
$$ \begin{pmatrix} x_{i-1,j-1}& x_{i-1,j}\\ x_{i,j-1}& x_{ij}\end{pmatrix}$$
of $X$ by its image under the map
\begin{equation}\label{abcd}
\begin{pmatrix} a& b\\ c& d\end{pmatrix} \qquad \mapsto \qquad 
\begin{pmatrix} bc/(ab+ac) & b\\ c& d(b+c) \end{pmatrix},
\end{equation}
while leaving the other elements unchanged. The upper-left corner of the output sub matrix \eqref{abcd} corresponds to the bumped elements from a row cf. \eqref{second_row_2}, 
while its lower-right corner corresponds to the updated number of elements in a row, cf. \eqref{first_row_2}.

For $1\le i\le n$ and $1\le j\le m$, set
$$\pi^j_i:=l_{ij}\circ\cdots\circ l_{i1}.$$
For $1\le i\le n$, define
\begin{equation}
R_i := \begin{cases} \pi_1^{m-i+1}\circ\cdots\circ\pi^m_i & i\le m\\
\pi^1_{i-m+1} \circ\cdots\circ\pi^m_i & i\ge m , \end{cases}  \label{defR}
\end{equation}
which corresponds to the (geometric) insertion of the $i^{th}$ row of the input matrix (or, bearing in mind the analogies from the previous section, the insertion of the $i^{th}$ word).
\vskip 2mm
We can now give the algebraic definition of the gRSK mapping as
\begin{equation}
 \grsk:=R_n\circ\cdots\circ R_1. \label{defT}
 \end{equation}
Let us illustrate the sequence with an example in the case $n=m=2$.  Then 
$$R_1=\pi^2_1=l_{12}\circ l_{11}=l_{12},\qquad R_2=\pi^1_1\circ\pi^2_2=l_{11}\circ l_{22}\circ l_{21}=l_{22}\circ l_{21}.$$
$$\grsk=R_2\circ R_1=l_{22}\circ l_{21}\circ l_{12}.$$
and
\begin{align*}
\grsk: \begin{pmatrix} a&b\\c&d\end{pmatrix} \stackrel{l_{12}}{\longmapsto}
\begin{pmatrix} a&ab\\c&d\end{pmatrix} \stackrel{l_{21}}{\longmapsto}
\begin{pmatrix} a&ab\\ac&d\end{pmatrix} \stackrel{l_{22}}{\longmapsto}
\begin{pmatrix} bc/(b+c)&ab\\ac&ad(b+c)\end{pmatrix} .
\end{align*}
Let us point out that
the mappings $R_i$ defined above can also be written as
$$R_i=\rho^i_m\circ\cdots\circ\rho^i_2\circ\rho^i_1$$
where
\begin{eqnarray}\label{rodef}
\rho^i_j=\begin{cases} l_{1,j-i+1}\circ\cdots\circ l_{i-1,j-1}\circ l_{ij} & i\le j\\
l_{i-j+1,1}\circ\cdots\circ l_{i-1,j-1}\circ l_{ij} & i\ge j.\end{cases}
\end{eqnarray}
Here we used the fact that $l_{ij}\circ l_{i'j'}=l_{i'j'}\circ l_{ij}$ whenever $|i-i'|+|j-j'|>2$, which is an immediate consequence of the fact that, in this case,
the mappings $l_{ij}, l_{i'j'}$ transform $2\times 2$ submatrices which don't overlap. 
The mappings $\rho^i_j$ are also going to be used in the construction of the geometric PNG.

This representation is related to the Bender-Knuth transformations, cf \cite{K01}, \cite{BK95}:
For each $1\le i\le n$ and $1\le j\le m$, denote by $b_{ij}$ the map on $(\R_{>0})^{n\times m}$ which
takes a matrix $X=(x_{qr})$ and replaces the entry $x_{ij}$ by 
\begin{equation}\label{bk1}
x'_{ij} = \frac1{x_{ij}} (x_{i,j-1}+x_{i-1,j}) \left( \frac1{x_{i+1,j}}+\frac1{x_{i,j+1}}\right)^{-1},
\end{equation}
leaving the other entries unchanged, with the conventions that $x_{0j}=x_{i0}=0$, 
$x_{n+1,j}=x_{i,m+1}=\infty$ for $1< i<n$ and $1< j<m$, but 
$x_{10}+x_{01}=x_{n+1,m}^{-1}+x_{n,m+1}^{-1}=1$.
Denote by $r_{j}$ the map which replaces the entry $x_{nj}$ by
$x_{n,j+1}/x_{nj}$ if $j<m$ and $1/x_{nm}$ if $j=m$, leaving the other entries unchanged.  
For $j\le m$, define
\begin{equation}\label{bk}
h_j=\begin{cases} b_{n-j+1,1}\circ\cdots\circ b_{n-1,j-1}\circ b_{nj} & j\le n\\
b_{1,j-n+1}\circ\cdots\circ b_{n-1,j-1}\circ b_{nj} & j\ge n.\end{cases}
\end{equation}
It is straightforward from the definitions to see that the mapping $\rho^n_j:(\bbR_{>0})^{n\times m}\mapsto (\bbR_{>0})^{n\times m}$ satisfies 
 $\rho^n_j=h_j\circ r_{j}$.  
\vskip 2mm
The next theorem summarises the main properties of gRSK. The first point is due to Kirillov \cite{K01}, while points 2., 3. where obtained in \cite{OSZ14}
\begin{theorem}\label{thm:OSZ}
Let  $W=(w_{ij} \,,\, 1\leq j\leq n\,,\, 1\leq i\leq m)$
and $T=(t_{ij} \,,\, 1\leq j\leq n\,,\, 1\leq i\leq m)=\grsk(W)$.  Then
\begin{itemize}
\item[{\bf 1.}]  Let $\Pi^{(r)}_{m,n}$ be the set of r-tuples of non intersecting down-right paths starting from points $(1,1),(1,2),...,(1,r)$ and ending at
points $(m,n-r+1),...,(m,n)$, respectively, with $r=1,...,m\wedge n$. Then 
\begin{align*}
t_{m-r+1,n-r+1}\dotsm t_{m-1,n-1} t_{m,n} 
= \sum_{(\pi_1,\ldots,\pi_r)\in\Pi^{(r)}_{m,n}} \, \prod_{(i,j)\in \pi_1\cup\cdots\cup\pi_r} w_{ij},
\end{align*}
and in particular 
\begin{align*}
 t_{m,n} 
= \sum_{\pi\in\Pi^{(1)}_{m,n}} \, \prod_{(i,j)\in \pi} w_{ij},
\end{align*}
is the polymer partition function $Z_{m,n}$.
\item[{\bf 2.}] Setting $t_{ij}=0$, whenever $(i,j)$ does not belong to the array $((i,j)\colon 1\leq j\leq n\,,\, 1\leq i\leq m)$, we have
\begin{align*}
\sum_{i,j}\frac{1}{w_{ij}}=\mathcal{E}(T):=\frac{1}{t_{11}}+\sum_{i,j}\frac{t_{i-1,j}+t_{i,j-1}}{t_{ij}}
\end{align*}
\item[{\bf 3.}]
The transformation 
\begin{align*}
(\log w_{ij}\, ,\,1\leq j\leq n\,,\, 1\leq i\leq m ) \mapsto (\log t_{ij}\, ,\, 1\leq j\leq n\,,\, 1\leq i\leq m),
\end{align*}
has Jacobian equal to $\pm 1$.
\end{itemize}
\end{theorem}
The previous theorem was used in \cite{OSZ14} to identify the push-forward measure under RSK when the input matrix is assigned the log-gamma distribution. Let us first 
define the latter
\begin{definition}[log-gamma measure]\label{def:log-gamma}
 For a sequence of real numbers $\ga=(\ga_i)_{i\geq 1}, \hat\ga=(\hat\ga_j)_{j\geq 1}$, such that $\ga_i+\hat\ga_j>0$ for all $(i,j)$,
  we define the $(\ga,\hat\ga)$-log-gamma measure on $W$, by
 \begin{align}\label{polygo_meas}
 \bbP(\dd W):=\prod_{i,j}\frac{1}{\Gamma(\ga_i+\hat\ga_j)} w_{ij}^{-\ga_i-\hat\ga_j} e^{-1/w_{ij}} \,\,\frac{\dd w_{ij}}{w_{ij}}.
 \end{align}
 \end{definition}
 \vskip 2mm
 Introducing the notation (we will also be using $\wedge$ for minimum and $\vee$ for maximum)
 \begin{align}\label{type}
 \tau_{mj}:=\frac{ \prod_{r=0}^{m\wedge j-1} t_{m-r,j-r}}{ \prod_{r=0}^{m\wedge (j-1)-1} t_{m-r,j-r-1}}
 \quad
\text{ and} 
 \quad
 \tau_{in}:=\frac{ \prod_{r=0}^{i\wedge n-1} t_{i-r,n-r}}{ \prod_{r=0}^{(i-1)\wedge n-1} t_{i-r-1,n-r}},
\end{align}
for $j=1,...,n$ and $i=1,...,m$, we have that, cf. \cite{OSZ14}, Corollary 3.3.
\begin{align*}
\bbP\circ \grsk^{-1}(\dd T)=
\prod_{i,j}\frac{1}{\Gamma(\ga_i+\hat\ga_j)} 
\Big( \prod_{j=1}^n \tau_{mj}^{-\hat{\ga}_j} \Big) \Big(\prod_{i=1}^m \tau_{in}^{-\ga_i}\Big) \, e^{-\mathcal{E}(T)} \prod_{i,j} \frac{\dd t_{ij}}{t_{ij}},
\end{align*}
where in the above products the indices $i,j$ run over $i=1,...,m$ and $j=1,...,n$.
We call the vectors $(\tau_{mj})_{j=1,...,n}, (\tau_{in})_{i=1,...,m} $ the (geometric) {\em type} of the (geometric) $P$ and $Q$ tableaux, respectively. 
This definition is the geometric version of the type of a Young tableau, which in terms of Gelfand-Tsetlin variables 
is defined to be the vector 
$\Big(\sum_{j=1}^iz_{ij} - \sum_{j=1}^{i-1} z_{i-1,j}\colon i\geq 1\Big)$ and records the number of 
letters $1,2,...$ in the tableau. 
 \vskip 2mm
 Let us now assume, for the sake of exposition, that the input matrix $W$ is a square matrix, i.e. $m=n$. Integrating the measure $\bbP\circ \grsk^{-1}(\dd T)$ over
 the variables $t_{ij}, i\neq j$, we obtain the law of the bottom row of the Gelfand-Tsetlin patterns cf. \eqref{RSKmaps}  (which encode the geometric lifting of the shape of the corresponding tableaux. This is given (in terms of Whittaker functions) as follows 
 \[
  e^{-1/t_{11}} \,\Psi^{(n)}_{\ga_1,...,\ga_n}(t_{nn},...,t_{11}) \,\Psi^{(n)}_{\hat\ga_1,...,\hat\ga_n}(t_{nn},...,t_{11})\,\,\prod_{i=1}^n \frac{\dd t_{ii}}{t_{ii}}.
 \]
 Whittaker functions have an integral formula due to Givental
 \[
 \Psi^{(n)}_{\ga_1,...,\ga_n}(x_1,...,x_n)=\int_{\mathcal{Z}(\bx)}\, \prod_{i=1}^n\left(\frac{z_{i1}\cdots z_{ii}}{z_{i-1,1}\cdots z_{i-1,i-1}}  \right)^{-\ga_i}
 e^{-\sum_{1\leq j\leq i\leq n} \big(\frac{z_{i,j}}{z_{i+1,j}}+\frac{z_{i+1,j+1}}{z_{i,j}}\big)} \prod_{1\leq j\leq i \leq n-1}\frac{\dd z_{ij}}{z_{ij}}
 \]
 where $\mathcal{Z}(\bx)=\big\{ (z_{ij})_{1\leq j\leq i \leq n} \in(\bbR_{>0})^{n(n+1)/2} \colon z_{nj}=x_j , \,\text{for}\, j=1,...,n\big\}$.
 If we change variables and define $\psi^{(n)}_{\ga_1,...,\ga_n}(x)=\Psi^{(n)}_{-\ga_1,...,-\ga_n}(z)$, where $x_i=\log z_i$ for $i=1,\ldots,n$, then 
 the functions $\psi^{(n)}_{\ga_1,...,\ga_n}(x)$ are eigenfunctions of the open quantum Toda chain with Hamiltonian given by
\[
H=-\sum_i \frac{\partial^2}{\partial x_i^2} + 2 \sum_{i=1}^{n-1} e^{x_{i+1}-x_i},
\]
with eigenvalue $-\left(\sum_i\ga_i^2\right)$.
Moreover, they play an important role in the theory of automorphic forms \cite{G06}. In our setting they appear more naturally as generalisation of Schur functions
having the form of generating functions of geometric Gelfand-Tsetlin pattern. The latter are to be thought as triangular arrays in the spirit of the standard Gelfand-Tsetlin pattern, with the difference that the interlacing conditions \eqref{GT_interlace} are violated. Instead, the hard boundary conditions of relations \eqref{GT_interlace}
are `softened' via the potential 
\[
-\sum_{1\leq j\leq i\leq n} \big(\frac{z_{i,j}}{z_{i+1,j}}+\frac{z_{i+1,j+1}}{z_{i,j}}\big),
\]
which penalises exponentially patterns that violate the interlacing conditions, see Figure \ref{GT_Whittaker}.
\begin{figure}[t]
 \begin{center}
\begin{tikzpicture}[scale=1.0]

\draw [fill] (0,0) circle [radius=0.05];
\draw [fill] (-0.7,-0.7) circle [radius=0.05];
\draw [fill] (-1.4,-1.4) circle [radius=0.05];
\draw [fill] (-2.1,-2.1) circle [radius=0.05];

\draw [fill] (0.7,-0.7) circle [radius=0.05];
\draw [fill] (1.4,-1.4) circle [radius=0.05];
\draw [fill] (2.1,-2.1) circle [radius=0.05];

\draw [fill] (0,-1.4) circle [radius=0.05];
\draw [fill] (0.7,-2.1) circle [radius=0.05];
\draw [fill] (-0.7,-2.1) circle [radius=0.05];

\node at (0.5,0.1) {\small{{\bf $z_{11}$}}};
\node at (-0.2,-0.7) {\small{{\bf $z_{22}$}}};
\node at (-0.9,-1.4){\small{{\bf $z_{33}$}}};
\node at (-1.7,-2.1) {\small{{\bf $z_{44}$}}};

\node at (1.1,-0.7) {\small{{\bf $z_{21}$}}};
\node at (1.8,-1.4) {\small{{\bf $z_{31}$}}};
\node at (2.5,-2.1) {\small{{\bf $z_{41}$}}};

\node at (0.5,-1.4) {\small{{\bf $z_{32}$}}};
\node at (1.2,-2.1) {\small{{\bf $z_{42}$}}};
\node at (-0.1,-2.1) {\small{{\bf $z_{43}$}}};

\draw[thick][->] (-0.6,-0.6)--(-0.1,-0.1);
\draw[thick][->] (-1.3,-1.3)--(-0.8,-0.8);
\draw[thick][->] (-2.0,-2.0)--(-1.5,-1.5);

\draw[thick][->] (0.1,-0.1)--(0.6,-0.6);
\draw[thick][->] (0.8,-0.8)--(1.3,-1.3);
\draw[thick][->] (1.5,-1.5)--(2.0 ,-2.0 );

\draw[thick][->] (-0.6,-0.8)--(-0.1,-1.3);
\draw[thick][->] (0.1,-1.5)--(0.6, -2.0);
\draw[thick][->] (0.1,-1.3)--(0.6, -0.8);
\draw[thick][->] (-0.6,-2.0)--(-0.1, -1.5);
\draw[thick][->] (0.8,-2.0)--(1.3, -1.5);
\draw[thick][->] (-1.3, -1.5)--(-0.8,-2.0);
%\draw[thick][->] (-0.5,-1.8)--(-0.1, -1.4);
\end{tikzpicture}

\end{center}  
\caption{ \small  This pictures shows a Gelfand-Tsetlin pattern and the arrows depict diagrammatically the potential in the integral form for the Whittaker function,
which is $-\sum_{a\to b} z_a/z_b $.
}  \label{GT_Whittaker}
\end{figure}
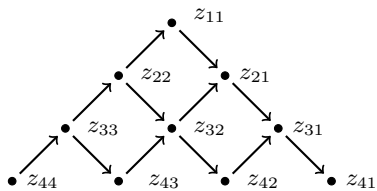

One of the reasons that the log-gamma measure is amenable to analysis is the fact that Whittaker functions have tractable Fourier analysis.
In particular, there exists a Plancherel theorem  \cite{STS94}, \cite{KL01}:
\begin{theorem}\label{thm:Plancerel}
The integral transform
$$\hat f(\lambda) = \int_{(\R_{>0})^n} f(x) \Psi^{(n)}_\lambda(x) \prod_{i=1}^n \frac{\dd x_i}{x_i}$$
defines an isometry from $L_2((\R_{>0})^n, \prod_{i=1}^n \dd x_i/x_i)$ onto 
$L^{sym}_2((\iota\R)^n,s_n(\lambda) \dd\lambda)$, where $L_2^{sym}$ is the space of $L_2$
functions which are symmetric in their variables (the variables $\lambda_1,\lambda_2,...$ are unordered), $\iota=\sqrt{-1}$ and 
$$s_n(\lambda)=\frac1{(2\pi\iota)^n n!} \prod_{i\ne j} \Gamma(\lambda_i-\lambda_j)^{-1},$$
is the {\em Sklyanin} measure.
That is, for any two functions $f,g\in L_2((\R_{>0})^n, \prod_{i=1}^n \dd x_i/x_i)$, it holds that
\[
\int_{(\R_{>0})^n} f(x) g(x) \prod_{i=1}^n\frac{\dd x_i}{x_i} = \int_{(\iota\R)^n} \hat{f}(\lambda) \overline{\hat g (\gl)} s_n(\gl) \dd \gl.
\]
\end{theorem}
Naturally, this Plancherel theorem will play a crucial role in this work, as well. We will also use an integral identity involving Whittaker functions 
with importance in the theory of automorphic forms (see \cite{B84})
\begin{theorem}[Stade] \label{thm:Stade}
 Suppose $r>0$ and $\lambda,\nu\in\C^n$, where $\Re (\lambda_i+\nu_j)>0$ for all $i$ and $j$.  Then
$$\int_{(\R_{>0})^n } e^{-r x_1} \Psi^{(n)}_{-\nu}(x) \Psi^{(n)}_{-\lambda}(x)\prod_{i=1}^n \frac{dx_i}{x_i}
= r^{-\sum_{i=1}^n (\nu_i+\lambda_i)} \prod_{ij}\Gamma(\nu_i+\lambda_j).$$
\end{theorem} 
This identity was conjectured by Bump in \cite{B84} and proved by Stade in \cite{St02}. An alternative proof using the gRSK was provided in \cite{OSZ14}, Corollary 3.7, 
which also showed that Stade's identity is the analogue of the Cauchy identity for Schur functions.
\vskip 2mm
When using the Plancherel-Whittaker theorem we will have to check that certain functions, which are expressed as products of gamma functions, are $L^2$ integrable.
For this will use the following asymptotics of the gamma function, cf \cite{OLBC10}, formula 5.11.9 :
\begin{align}\label{gamma_asympto}
\lim_{b\to\infty}|\Gamma(a+\iota b)| e^{\frac{\pi}{2}|b|}|b|^{\frac{1}{2}-a}=\sqrt{2\pi},
\qquad a>0, \,b \in \mathbb{R} .
\end{align}

\subsection{Extension of $g$RSK to polygonal input arrays}\label{sec:gRSKpoly}
In this section we will extend the gRSK to the case where the input array is a polygonal array, Young-diagram-like
 of the form $\polygon$, rather than a matrix.
 Let us start by considering pairs of positive integers $(m_\ell,n_\ell)_{ \ell=1,2,...,k}$ with $m_{\ell} > m_{\ell-1}$ and $n_{\ell-1}>n_\ell$ and rectangular arrays 
 $W^{(\ell)} = (w_{ij} \,\colon\, m_{\ell-1} < i \leq m_\ell, 1\leq j \leq n_{\ell}  )$. By convention, we have that $m_0=n_0=0$. The general input pattern will consist of an array
 $W=\bigsqcup_{\ell=1}^kW^{(\ell)}$,  created by stacking the rectangular pattern $W^{\ell}$ below $W^{\ell-1}$ (see Figure \ref{poly_pattern} 
 for an illustration). 
 We define the index set of such a polygonal array $W$ as 
 \begin{align}\label{index}
 ind(W):=\bigcup_{\ell=1}^k \big\{(i,j)\colon m_{\ell-1} < i \leq m_\ell, 1\leq j \leq n_{\ell}   \big\}
 \end{align}
 We now come to
 \begin{definition}[gRSK on polygonal arrays]
 We define the geometric RSK on the {\it polygonal array} $W=\bigsqcup_{\ell=1}^k W^{(\ell)}$ inductively as follows:
 \begin{itemize}
 \item[$(1)$] Set $T^{(1)}:=\grsk(W^{(1)})\in \R^{m_1\times n_1}$. $T^{(1)}$ is then the rectangular output provided by gRSK, as this is described in
  Section \ref{section:localmoves}.
 \item[$(2)$] Create the array $W^{(2)}\bigsqcup T^{(1)}$ by appending matrix $W^{(2)}$ at the bottom of the output matrix $T^{(1)}$ obtained in the
 previous step. We then define $T^{(2)}:=
 R_{m_{2}}\circ R_{m_{2}-1}\circ \cdots\circ R_{m_{1}+1} \big(W^{(2)} \, \bigsqcup T^{(1)} \big) $. The crucial point, here, is that the transformations 
 $R_a,\, m_1+1\leq a \leq m_2$, act only on the entries $(i,j)$ of $W^{(2)}\bigsqcup T^{(1)}$ for which $i-j\geq m_1+1-n_2$. That is, they only transform the entries which lie {\it on} or {\it below} the north-west diagonal starting from entry $(m_1+1,n_2)$, while leaving all other entries unchanged, see Figure \ref{fig2}.

 \item[$(3)$] We can, now, define inductively $T^{(\ell)}:=
 R_{m_{\ell}}\circ R_{m_{\ell}-1}\circ \cdots\circ R_{m_{\ell-1}+1} \big(W^{\ell} \, \bigsqcup_{r=1}^{\ell-1} T^{(r)} \big) $ for $\ell\leq k$ and, finally, $\grsk(W):=T^{(k)}$.
 \end{itemize}
 \end{definition}
 
 \begin{figure}[t]
 \begin{center}
\begin{tikzpicture}[scale=.5]
\draw[help lines] (1,1) grid (3,4);
\draw[help lines] (1,4) grid (4,6);
\draw[help lines] (1,6) grid (6,8);
\draw[help lines] (1,8) grid (10,11);

\draw [fill] (3,1) circle [radius=0.1];
\draw [fill] (4,4) circle [radius=0.1];
\draw [fill] (6,6) circle [radius=0.1];
\draw [fill] (10,8) circle [radius=0.1];

\node at (4.2,1) {\small{$(11,3)$}};
\node at (5,4) {\small{$(8,4)$}};
\node at (7,6) {\small{$(6,6)$}};
\node at (11.2,8) {\small{$(4,10)$}};
\end{tikzpicture}
 \end{center}  
\caption{ \small  This figure shows a rectangular array $W=W^{(1)} \bigsqcup  W^{(2)} \bigsqcup W^{(3)} \bigsqcup W^{(4)} $, which 
is a concatenation of four rectangular arrays. Their bottom right corners are marked by points $(m_\ell,n_\ell)_{\ell=1,...,4}=(4,10), (6,6), (8,4), (11,3)$ }
\label{poly_pattern}
\end{figure}
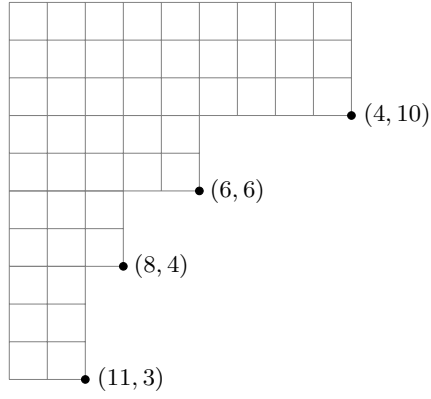

  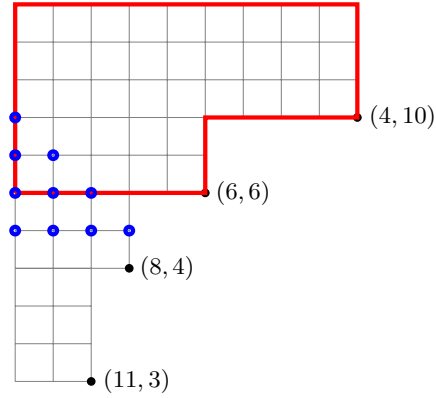
\begin{figure}[t]
 \begin{center}
\begin{tikzpicture}[scale=.5]
\draw[help lines] (1,1) grid (3,4);
\draw[help lines] (1,4) grid (4,6);
\draw[help lines] (1,6) grid (6,8);
\draw[help lines] (1,8) grid (10,11);

\draw [fill] (3,1) circle [radius=0.1];
\draw [fill] (4,4) circle [radius=0.1];
\draw [fill] (6,6) circle [radius=0.1];
\draw [fill] (10,8) circle [radius=0.1];

\node at (4.2,1) {\small{$(11,3)$}};
\node at (5,4) {\small{$(8,4)$}};
\node at (7,6) {\small{$(6,6)$}};
\node at (11.2,8) {\small{$(4,10)$}};

\draw[red, ultra thick] (1,6)--(6,6)--(6,8)--(10,8)--(10,11)--(1,11)--(1,6);

\draw[blue, ultra thick]  (4,5) circle [radius=0.1];
\draw[blue, ultra thick]  (3,6) circle [radius=0.1];
\draw[blue, ultra thick]  (2,7) circle [radius=0.1];
\draw[blue, ultra thick]  (1,8) circle [radius=0.1];
\draw[blue, ultra thick]  (3,5) circle [radius=0.1];
\draw[blue, ultra thick]  (2,6) circle [radius=0.1];
\draw[blue, ultra thick]  (1,7) circle [radius=0.1];
\draw[blue, ultra thick]  (2,5) circle [radius=0.1];
\draw[blue, ultra thick]  (1,6) circle [radius=0.1];
\draw[blue, ultra thick]  (1,5) circle [radius=0.1];
\end{tikzpicture}
 \end{center}  
\caption{ \small  The red contour shows the part of the polygonal  array that will be transformed after application of the gRSK on the two top rectangular arrays, i.e. $\grsk(W^{(2)} \bigsqcup W^{(1})$. The blue dots show the entries that will be modified by the insertion of the seventh row $(w_{(7,1)}, w_{(7,2)}w_{(7,3)}, w_{(7,4)})$ via application of transformation $R_7$. From this it is clear that $R_7$ does not see the full polygonal array and it would lead to the same output, as if only the 
subarray $((i,j)\,,\, 1\leq i\leq 7,\, 1\leq j\leq 4)$ existed.}\label{fig2}
\end{figure}
 
  \begin{figure}[t]
 \begin{center}
\begin{tikzpicture}[scale=.5]
\draw[help lines] (1,1) grid (3,4);
\draw[help lines] (1,4) grid (4,6);
\draw[help lines] (1,6) grid (6,8);
\draw[help lines] (1,8) grid (10,11);

\draw [fill] (3,1) circle [radius=0.1];
\draw [fill] (4,4) circle [radius=0.1];
\draw [fill] (6,6) circle [radius=0.1];
\draw [fill] (10,8) circle [radius=0.1];

\node at (4.2,1) {\small{$(11,3)$}};
\node at (5,4) {\small{$(8,4)$}};
\node at (7,6) {\small{$(6,6)$}};
\node at (11.2,8) {\small{$(4,10)$}};

\draw[very thick] (1,11)--(1,8)--(2,8)--(2,7)--(3,7)--(3,6)--(4,6);
\draw[very thick] (2,11)--(2,9)--(5,9)--(5,6);
\draw[very thick] (3,11)--(4,11)--(4,10)--(6,10)--(6,6);

\end{tikzpicture}
 \end{center}  
\caption{ \small  This figure shows three non intersecting, down-right paths in $\Pi^{(3)}_{6,6}$. The partition function corresponding to this ensemble equals
$t_{6,6}t_{5,5}t_{4,4}$, where $(t_{ij})=\grsk(\bigsqcup_{\ell=1}^4 W^{(\ell)})$
}\label{fig:poly_path}
\end{figure}
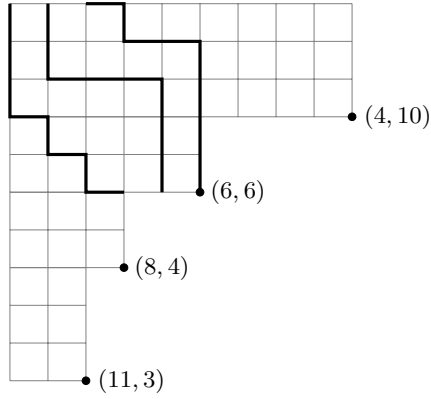

 The next proposition summarises the fundamental bijective properties of gRSK on polygonal arrays and it is the analogue of Theorem \ref{thm:OSZ}. 
\begin{proposition}\label{polygoRSK}
Consider a polygonal array $W=\bigsqcup_{\ell=1}^k W^{(\ell)}$ with $W^{(\ell)}=(w_{ij} \,,\, 1\leq j\leq n_\ell\,,\, m_{\ell-1}< i\leq m_{\ell})$
and $T=(t_{ij},\, (i,j)\in ind(W))=\grsk(W)$, where $ind(W)$ is the index set of the polygonal array $W$ defined in \eqref{index}. Then
\begin{itemize}
\item[{\bf 1.}]  Let $\Pi^{(r)}_{m,n}$ be the set of r-tuples of non intersecting down-right paths starting from points $(1,1),(1,2),...,(1,r)$ and ending at
points $(m, n-r+1),...,(m,n)$, respectively, with $(m,n)\in\{(m_\ell,n_\ell)\colon \ell=1,...,k\}$ and $r=m\wedge n$ (see Figure \ref{fig:poly_path}). Then 
\begin{align*}
t_{m-r+1,n-r+1}\dotsm t_{m-1,n-1} t_{m,n} 
= \sum_{(\pi_1,\ldots,\pi_r)\in\Pi^{(r)}_{m,n}} \, \prod_{(i,j)\in \pi_1\cup\cdots\cup\pi_r} w_{ij},
\end{align*}
and in particular 
\begin{align*}
 t_{m_\ell,n_\ell} 
= \sum_{\pi\in\Pi^{(1)}_{m_\ell,n_\ell}} \, \prod_{(i,j)\in \pi} w_{ij}, \qquad \text{for} \quad \ell=1,...,k,
\end{align*}
are the polymer partition functions $Z_{m_\ell,n_\ell}$.
\item[{\bf 2.}] Setting $t_{ij}=0$, whenever $(i,j)$ does not belong in the index set of the polygonal array, we have
\begin{align*}
\sum_{(i,j)\in ind(W)}\frac{1}{w_{ij}}=\mathcal{E}(T):=\frac{1}{t_{11}}+\sum_{(i,j)\in ind(W)}\frac{t_{i-1,j}+t_{i,j-1}}{t_{ij}}.
\end{align*}
\item[{\bf 3.}]
The transformation 
\begin{align*}
(\log w_{ij}\, ,\, (i,j)\in \text{ind}(W)) \mapsto (\log t_{ij}\, ,\, (i,j)\in \text{ind}(W)),
\end{align*}
has Jacobian equal to $\pm 1$.
\end{itemize}
\end{proposition}
\begin{proof}
Part {\bf 3.} follows immediately from the volume preserving properties of the local moves. The proof of part {\bf 2.} has minor modifications of the
proof of part {\bf 3.} of Theorem \ref{thm:OSZ}. Since a similar property holds for the geometric PNG, whose proof will require slightly more modifications, we prefer
to defer the details till Proposition \ref{prop:PNG_energy}. 

The proof of Part {\bf 1.} proceeds by induction in $k$. For $k=1$ this is the same statement as that of Theorem \ref{thm:OSZ}, Part {\bf 1}. 
Assume, now, that the statement is valid for $k-1$, that is,
 \[
 T^{(k-1)}:=\grsk\Big(\bigsqcup_{\ell=1}^{k-1} W^{(\ell)}\Big) 
\]  
satisfies the statement of Part {\bf 1}. Let
\[T=T^{(k)}:=
 R_{m_{k}}\circ R_{m_{k}-1}\circ \cdots\circ R_{m_{k-1}+1} \big(W^{(k)} \, \bigsqcup \, T^{(k-1)} \big). 
\]
Observe that a transformation 
 $R_a$ acts only on the entries $(i,j)$ with $i-j\geq a-j_*(a)$, where $j_*(a)=\max\{j\colon (a,j) \in ind(W) \}$. This means that when $R_a$ inserts a line 
 $(w_{aj}, \, 1\leq j\leq j_*(a))$, then it will act as if there was only a rectangular pattern $(w_{i,j},\, 1\leq a, \,1\leq j\leq j_*(a) \,)$, ignoring the rest of the polygonal pattern.
 This, in particular, holds for $a=m_{k-1}+1,...,m_k$.
 Moreover, the insertions of $(w_{ij})$ with $j > n_k$ (that have taken place before the insertions $R_a$ for $a=m_{k-1}+1,...,m_k$) 
 will not have an effect on the entries
  $(t_{ij}\colon i-j>m_{k-1}-n_{k})$ of $T$ (see also Figure \ref{fig2}). 
Both of these statements can be seen as a consequence of the Bender-Knuth transformations \eqref{bk1}, \eqref{bk}, which imply that when an entry $w_{ij}$ is inserted, 
it leads only to updates of the
entries of the array along the north-west diagonal starting from $(i,j)$.
   
  The claim now follows by combining the induction hypothesis with the statement of Theorem \ref{thm:OSZ}, Part {\bf 1}, 
  applied to the rectangular array $(w_{ij},\, 1\leq i \leq m_k, \,1 \leq j \leq n_k)$.
\end{proof}

It will be useful, for later purposes, to state separately a proposition, which defines and identifies the type, cf. \eqref{type},
of general polygonal patterns via the gRSK:
\begin{proposition}\label{grsk_type}
Let a polygonal array $W=\bigsqcup_{\ell=1}^k W^{(\ell)}$ with $W^{(\ell)}=(w_{ij} \,,\, 1\leq j\leq n_\ell\,,\, m_{\ell-1}< i\leq m_{\ell})$
and $T=\grsk(W)$. We have that
\begin{itemize}
\item[$\bullet$]
For every $j$, we set $ i_*(j):=\max\{i\colon (i,j)\in ind(W)\}  $. Then
\begin{align}\label{prop_grsk_type_eq1}
\prod_{i=1}^{i_*(j)} w_{ij}=\tau_{i_*(j),j} := 
\frac{ \prod_{r=0}^{i_*(j)\wedge j -1}\, t_{i_*(j)-r,j-r} }{ \prod_{r=0}^{i_*(j)\wedge (j-1) -1} \,t_{i_*(j)-r,j-1-r} }.
\end{align}
\item[$\bullet$]
For every $i$, we set $ j_*(i):=\max\{j\colon (i,j)\in ind(W)\} $. Then
\begin{align}\label{prop_grsk_type_eq2}
\prod_{j=1}^{j_*(i)} w_{ij}=\tau_{i,j_*(i)} := 
\frac{ \prod_{r=0}^{j_*(i)\wedge i -1}\, t_{i-r,j_*(i)-r} }{ \prod_{r=0}^{j_*(i)\wedge (i-1) -1} \,t_{i-1-r,j_*(i)-r} }.
\end{align}
\end{itemize}
\end{proposition}
\begin{proof}
We will prove \eqref{prop_grsk_type_eq1}, the proof of \eqref{prop_grsk_type_eq2} being similar.
When $k=1$ the polygonal array is a rectangular array and the claim follows from the properties of the standard gRSK, cf. point {\bf 1.} of 
Theorem \ref{thm:OSZ}. Assume, now, that $k=2$. Since after the insertion of $W^{(2)}$ the entries 
$(t_{ij})$ of $\grsk(W^{(1)} \bigsqcup W^{(2)} )$ are identical to the corresponding ones of $\grsk(W^{(1)})$, 
  claim \eqref{prop_grsk_type_eq1} hold immediately 
  for $n_2<j\leq n_1$. For $1\leq j \leq n_2$ the entries $(t_{ij})$ involved in the right hand side of
  \eqref{prop_grsk_type_eq1} have indices satisfying $i-j\geq m_2-n_2$ but as already mentioned the values of these entries do not depend on
  $(w_{ij})$ with $j>n_2$, which means that these entries would be identical to the entries of 
  $\grsk\big((w_{ij}\colon 1\leq i \leq m_2, 1\leq j \leq n_2)\big)$. Bearing this in mind, the validity of \eqref{prop_grsk_type_eq1}
  for $1\leq j \leq n_2$ is a consequence of the properties of standard gRSK (Theorem \ref{thm:OSZ} point {\bf 1.}).
  The same argument continues when $k\geq 3$.
\end{proof}
\subsection{Multipoint distribution of the log-gamma polymer}\label{sec:multipoint}
We are now ready to use the formulation of gRSK on polygonal arrays, in order to obtain expressions for the multipoint distribution of the log-gamma polymer. 
The first step in this direction is to determine the push forward measure under gRSK of the log-gamma measure on a polygonal array
 $W=\bigsqcup_{\ell=1}^kW^{(\ell)}$ with $W^{(\ell)} = (w_{ij} \,\colon\, m_{\ell-1} < i \leq m_\ell, 1\leq j \leq n_{\ell}  )$ with $(m_\ell,n_\ell)_{ \ell=1,2,...,k}$ such that $m_{\ell} > m_{\ell-1}$ and $n_{\ell-1}>n_\ell$. By convention $m_0=n_0=0$. To do this, we first recall the definition of an $(\ga,\hat\ga)$-log-gamma measure on such a 
 polygonal array $W$:
 \begin{align*}
 \bbP(\dd W):=\prod_{(i,j)\in ind(W)}\frac{1}{\Gamma(\ga_i+\hat\ga_j)} w_{ij}^{-\ga_i-\hat\ga_j} e^{-1/w_{ij}} \,\,\frac{\dd w_{ij}}{w_{ij}}.
 \end{align*}
  Rewrite the $(\ga,\hat\ga)$-log-gamma measure as
 \begin{align*}
 \bbP(\dd W)&=\prod_{(i,j)\in \,ind(W)}\frac{1}{\Gamma(\ga_i+\hat\ga_j)}
 \prod_{j=1}^{n_1}\Big(\prod_{i=1}^{i_*(j)} w_{ij}\Big)^{-\hat\ga_j}  \prod_{i=1}^{m_k}\Big(\prod_{j=1}^{j_*(i)} w_{ij}\Big)^{-\ga_i}\\
  &\qquad\qquad\times \exp\Big(-\sum_{(i,j)\in \,ind(W)}1/w_{ij}\Big) \,\,\prod_{(i,j)\in\,ind(W)}\frac{\dd w_{ij}}{w_{ij}}.
 \end{align*}
 By Proposition \ref{polygoRSK} and Proposition \ref{grsk_type} and by changing variable 
 $\big(w_{i,j}\colon (i,j)\in ind(W)\big) \mapsto \big(t_{i,j}\colon (i,j)\in ind(W)\big)=\grsk(W)$
  we have that the push forward measure is
  \begin{align}\label{RSK-1}
  \bbP\circ(\grsk)^{-1}(\dd T)&= \prod_{(i,j)\in \,ind(W)}\frac{1}{\Gamma(\ga_i+\hat\ga_j)}\,\,
  \prod_{j=1}^{n_1} \Big( \tau_{i_*(j),j} \Big)^{-\hat\ga_j} 
  \,\prod_{i=1}^{m_k}\Big( \tau_{i,j_*(i)}  \Big)^{-\ga_i}\notag\\
  &\qquad\qquad\times \exp\Big(-\cE(T)\Big) \,\,\prod_{(i,j)\in\,ind(W)}\frac{\dd t_{ij}}{t_{ij}}.
  \end{align}
  We now arrive to our first formula for the joint law of the partition functions:
  \begin{theorem}[k-point distribution]\label{thm:k-point}
  Let $(m_1,n_1),...,(m_k,n_k)$ be $k$ pairs of positive integers such that $m_{\ell}>m_{\ell-1}$ and $n_{\ell-1}>n_{\ell}$ and consider the point-to-point partition functions $Z_{(m_i,n_i)},\,i=1,...,k$ of an $(\ga,\hat\ga)$-log-gamma polymer. The Laplace transform of the joint distribution is given by
  \begin{align*}
  \bbE\Big[e^{-u_1Z_{(m_1,n_1)}-\cdots -u_k Z_{(m_k,n_k)}}\Big] =
  \int_{(\bbR_{>0})^{| ind(W)|}} e^{-u_1 t_{m_1,n_1}-\cdots -u_k t_{m_k,n_k}} \,\,\bbP\circ(\grsk)^{-1}(\dd T),
  \end{align*}
  where $\bbP\circ(\grsk)^{-1}$ is the push forward measure under gRSK of an  $(\ga,\hat\ga)$-log-gamma measure on the
   polygonal array $W:=W_{(m_1,n_1),...,(m_k,n_k)}$, determined by the {\it corners}
  $(m_1,n_1),...,(m_k,n_k)$, see \eqref{RSK-1}.
  \end{theorem}
  \begin{proof}
  The proof follows immediately from the fact that if $T=(t_{ij})=\grsk\big(W_{(m_1,n_1),...,(m_k,n_k)}\big)$, then by Part {\bf 1.} of Proposition \ref{polygoRSK}
  \[
  ( t_{m_1,n_1},..., t_{m_k,n_k} ) = (Z_{(m_1,n_1)},...,Z_{(m_k,n_k)}),
  \]
   and the identification of the push forward measure $\bbP\circ(\grsk)^{-1}(\dd T)$ as in \eqref{RSK-1}.
  \end{proof}
 We will now use the Plancherel Theorem \ref{thm:Plancerel} for Whittaker functions 
 to rewrite the above integral as a contour integral in the case of $k=2$ . We first introduce some notation
  \begin{definition}\label{def:F}
  For a sequence of complex numbers $a=(a_1, a_2,...)$ and integers $p,r$ we will denote the function
\begin{align*}
\sfF_{p,r}^{\,a}(w):=\prod_{p+1\leq i \leq r} \Gamma(w+a_j),\qquad \text{for} \,\,w\in\mathbb{C}.
\end{align*}
When $p=0$ we will simplify the notation by denoting the corresponding function by $\sfF_{r}^{\,a}(w)$.
  \end{definition}
  We will also make use of the following theorem proved in \cite{COSZ14}, \cite{OSZ14}, which expresses the Laplace transform of a single partition function as a
   contour integral
  \begin{theorem}\label{thm:one_point_Laplace}
  	Let $Z_{(m,n)}$ be the partition function from $(1,1)$ to $(m,n)$, with $m\geq n$, of an $(\ga,\hat\ga)$-log-gamma polymer such that $\hat\ga<0$ and $ \ga >0$. Its Laplace transform is given in terms of a contour integral as
  	\begin{align*}
  	\bbE\Big[e^{-u Z_{(m,n)} }\Big] 
  	&=  \int_{(\iota \bbR)^{n} } \dd \mu  
  	\,\,s_{n}(\mu)  \prod_{1\leq j,j'\leq n} \Gamma(-\hat\ga_{j'}+\mu_j) \prod_{j=1}^{n}\frac{u^{-\mu_j} \,\sfF^{\,\ga}_{m}(\mu_j)}{u^{-{\hat\ga_j}} \,\sfF^{\,\ga}_{m}(\hat\ga_j)}.
  	\end{align*}
  \end{theorem}
  
  \begin{remark}\rm{
   If the parameters $\ga,\hat\ga$ are such that $\Re(\ga+\gd)>0$ and $\Re(-\hat\ga+\gd)>0$, for some $\gd \in \bbR$, then the above formula writes as
   \begin{align*}
 	\bbE\Big[e^{-u Z_{(m,n)} }\Big] 
 	&=  \int_{(\ell_{\delta})^{n} } \dd \mu  
 	\,\,s_{n}(\mu)  \prod_{1\leq j,j'\leq n} \Gamma(-\hat\ga_{j'}+\mu_j) \prod_{j=1}^{n}\frac{u^{-\mu_j} \,\sfF^{\,\ga}_{m}(\mu_j)}{u^{-{\ga_j}} \,\sfF^{\,\ga}_{m}(\hat\ga_j)}.
 	\end{align*}
	where $\ell_\gd$ is the contour along which the real part is constant and equal to $\gd$. }
	\end{remark}
 We can now state
  \begin{theorem}[2-point contour integral]\label{thm:2-point}
  Let $Z_{(m_1,n_1)}, Z_{(m_2,n_2)}$ be the partition functions from $(1,1)$ to $(m_1,n_1)$ and $(m_2,n_2)$, respectively, of an $(\ga,\hat\ga)$-log-gamma polymer. Assume that $m_1< m_2$ and $n_1> n_2$. Moreover, without loss of generality, assume that $m_1\leq n_1$.
  Let $\gamma/2>\delta>0$ arbitrary, such that the parameters 
   	$(\ga_i)_{i=1,2...}$ and $(\hat\ga_j)_{j=1,2...}$ satisfy $\vert \ga\vert <\delta$, $\vert \hat \ga-\gamma\vert <\delta$,  
	and let $\ell_{\delta}, \ell_{\delta+\gamma}$ be  vertical lines in the complex plane with real part equal to $\delta$ and $\delta+\gamma$, respectively.
  The joint Laplace transform is given in terms of a contour integral as
  \begin{itemize}
  \item[$\bullet$] in the case that $m_2\geq n_2 $, then
  \begin{align}\label{2p_formula}
 \bbE\Big[e^{-u_1 Z_{m_1,n_1}-u_2Z_{m_2,n_2}}\Big] 
 = &\int_{( \ell_\delta)^{m_1} } \dd \lambda
    \,\,s_{m_1}(\lambda)  \prod_{1\leq i,i'\leq m_1} \Gamma(-\ga_i+\gl_{i'}) \prod_{i=1}^{m_1}\frac{u_1^{-\gl_i}\, \sfF^{\,\hat\ga}_{n_2,n_1}(\gl_i)}{ u_1^{-\ga_i}\, \sfF^{\,\hat\ga}_{n_2,n_1} (\ga_i)} \notag\\
&\quad\times   \int_{(\ell_{\delta+\gamma})^{n_2} } \dd \mu  
    \,\,s_{n_2}(\mu)  \prod_{1\leq j,j'\leq n_2} \Gamma(-\hat\ga_j+\mu_{j'}) \prod_{j=1}^{n_2}\frac{u_2^{-\mu_j} \,\sfF^{\,\ga}_{m_1,m_2}(\mu_j)}{u_2^{-{\hat\ga_j}} \,\sfF^{\,\ga}_{m_1,m_2}(\hat\ga_j)} \notag\\
 &\qquad   \quad\times \prod_{\substack{1\leq i\leq m_1 \\ 1\leq j\leq n_2 }} \frac{\Gamma(\gl_i+\mu_j)}{\Gamma(\ga_i+\hat\ga_j)}\,,
  \end{align}
  \item[$\bullet$] in the case that $m_2 < n_2 $, then
   \begin{align}\label{2p_formulaB}
 \bbE\Big[e^{-u_1 Z_{m_1,n_1}-u_2Z_{m_2,n_2}}\Big] 
 = &\int_{(\ell_{\delta})^{m_1} } \dd \lambda
    \,\,s_{m_1}(\lambda)  \prod_{1\leq i,i'\leq m_1} \Gamma(-\ga_i+\gl_{i'}) \prod_{i=1}^{m_1}\frac{(u_1/u_2)^{-\gl_i}\, \sfF^{\,\hat\ga}_{n_2,n_1}(\gl_i)}{ (u_1/u_2)^{-\ga_i}\, \sfF^{\,\hat\ga}_{n_2,n_1} (\ga_i)} \notag\\
&\quad\times   \int_{(\ell_{\delta'})^{m_2} } \dd \mu  
    \,\,s_{m_2}(\mu)  
      \prod_{\begin{substack}{m_1+1\leq i\leq m_2\\ 1\leq i'\leq m_2}\end{substack}} 
    \Gamma(-\ga_i+\mu_{i'}) 
    \prod_{i=1}^{m_2}\frac{u_2^{-\mu_i} \,\sfF^{\,\hat\ga}_{n_2}(\mu_i)}{u_2^{-{\ga_i}} \,\sfF^{\,\hat\ga}_{n_2}(\ga_i)}   \notag\\
 &\qquad\quad\times
  \,\prod_{\begin{substack}{1\leq i\leq m_1\\ 1\leq i'\leq m_2}\end{substack}} 
    \Gamma(-\gl_i+\mu_{i'}) \,,
  \end{align}
  where the contours $\ell_\gd,\ell_{\gd'}$ have real parts $\gd,\gd'$, respectively, with $\gd'>\gd$.
  \end{itemize}
  \end{theorem}
  \begin{proof}
  We first treat the case $m_2\geq n_2$ and for simplicity we assume that $m_2>n_2$ (the case $m_2=n_2$ requires only minor notational modifications).
  From Theorem \ref{thm:k-point} we can write 
  \begin{align*}
&  \bbE\Big[e^{-u_1Z_{m_1,n_1}-u_2 Z_{m_2,n_2}}\Big] =
  \int e^{-u_1 t_{m_1,n_1}-u_2 t_{m_2,n_2}} \,\,\bbP\circ(\grsk)^{-1}(\dd T)\\
 =&\frac{1}{\prod_{i,j}\Gamma(\ga_i+\hat\ga_j)}\int
     \prod_{j=1}^{n_1} \Big( \tau_{i_*(j),j} \Big)^{-\hat\ga_j} 
  \,\prod_{i=1}^{m_2}\Big( \tau_{i,j_*(i)}  \Big)^{-\ga_i} \,e^{-\frac{1}{t_{11}}-u_1 t_{m_1,n_1}-u_2 t_{m_2,n_2}} \notag\\
  &\qquad\qquad\times \exp\Big(-\sum\limits_{\begin{substack}{ (i,j)\in\,ind(W)}\end{substack}}
   \frac{t_{i-1,j}+t_{i,j-1}}{t_{ij}} \Big) \,\,
  \prod_{\begin{substack}{ (i,j)\in\,ind(W) }
  \end{substack}} \frac{\dd t_{ij}}{t_{ij}}.
   \end{align*}
   We integrate first over all variables $(t_{ij})$ except $(t_{m_1,n_1},t_{m_1-1,n_1-1},...t_{1,n_1-m_1+1})$ and also do the integrations
    over the variables $(t_{ij}\colon i-j < m_1-n_1)$ separately from the integration over the variables $(t_{ij}\colon i-j> m_1-n_1)$. 
   We can then rewrite the above integral as
     \begin{align}\label{eq:2-point-int}
        &\frac{1}{ \prod\limits_{ i,j}\Gamma(\ga_i+\hat\ga_j)}
  \int_{(\bbR_{>0})^{m_1}} e^{-u_1 t_{m_1,n_1}} \,e^{-u_2 t_{m_2,n_2}} \,
 \Whitt_{\ga,\hat\ga}(t_{m_1,n_1}^{\nwarrow}) \, \Psi^{(m_1)}_\ga(t_{m_1,n_1}^{\nwarrow}) \prod_{r=0}^{m_1-1}
  \frac{\dd t_{m_1-r,n_1-r}}{ t_{m_1-r,n_1-r}}, \notag\\
  &
   \end{align}
   where $t_{m_1,n_1}^{\nwarrow}$ denotes the vector $(t_{m_1,n_1},t_{m_1-1,n_1-1},...t_{1,n_1-m_1+1})$, $ \Psi^{(m_1)}_\ga(t_{m_1,n_1}^{\nwarrow})$ is a $GL_{m_1}(\bbR)$-Whittaker function corresponding to the triangle $(t_{ij}\colon i-j\leq m_1-n_1)$, i.e. 
   (recall also notation $\tau_{i,j_*(i)}$ from Proposition \ref{grsk_type})
   \begin{align}\label{int_Whitt_m1}
   &\Psi^{(m_1)}_\ga(t_{m_1,n_1}^{\nwarrow}) \\
   &= \int 
   \prod_{i=1}^{m_1}\Big( \tau_{i,j_*(i)} \Big)^{-\ga_i}
    \exp\Big(-\sum\limits_{\begin{substack}{ (i,j)\in\,ind(W) \\ i-j<m_1-n_1}\end{substack}}
   \frac{t_{i-1,j}+t_{i,j-1}}{t_{ij}} -\sum\limits_{\begin{substack}{ (i,j)\in\,ind(W) \\ i-j=m_1-n_1}\end{substack}}
   \frac{t_{i-1,j}}{t_{ij}} \Big) \,\,
  \prod_{\begin{substack}{ (i,j)\in\,ind(W) \\ i-j<m_1-n_1}
  \end{substack}} \frac{\dd t_{ij}}{t_{ij}},\notag
   \end{align}
   and  $\Whitt_{\ga,\hat\ga}(t_{m_1,n_1}^{\nwarrow})$ is a Whittaker-type function corresponding to the pattern 
   $(t_{ij}\colon i-j\geq m_1-n_1)$, defined by
   \begin{align}\label{int_Whitt_poly}
   \Whitt_{\ga,\hat\ga}(t_{m_1,n_1}^{\nwarrow})&:=\int
     \prod_{j=1}^{n_1} \Big( \tau_{i_*(j),j} \Big)^{-\hat\ga_j} 
  \,\prod_{i=m_1+1}^{m_2}\Big( \tau_{i,j_*(i)}  \Big)^{-\ga_i} \\
  &\qquad\times \exp\Big(-\frac{1}{t_{11}}-\sum\limits_{\begin{substack}{ (i,j)\in\,ind(W) \\ i-j>m_1-n_1}\end{substack}}
   \frac{t_{i-1,j}+t_{i,j-1}}{t_{ij}}
   -\sum\limits_{\begin{substack}{ (i,j)\in\,ind(W) \\ i-j=m_1-n_1}\end{substack}}
   \frac{t_{i,j-1}}{t_{ij}} \Big) \,\,
  \prod_{\begin{substack}{ (i,j)\in\,ind(W) \\ i-j>m_1-n_1}
  \end{substack}} \frac{\dd t_{ij}}{t_{ij}}. \notag
   \end{align}
   We will now use the Plancherel Theorem \ref{thm:Plancerel} for Whittaker functions to rewrite the integral in \eqref{eq:2-point-int} as a contour integral. First,
   using the Plancherel theorem, we write   \eqref{eq:2-point-int}  as
   \begin{align}\label{eq:2-point-int2}
   &\int_{(\iota \bbR)^{m_1}} \dd\gl \,\,s_{m_1}(\gl)\,\,\Bigg(
   \int e^{-u_1 t_{m_1,n_1}} \,\Psi^{(m_1)}_\ga(t_{m_1,n_1}^{\nwarrow}) \Psi^{(m_1)}_{-\gl}(t_{m_1,n_1}^{\nwarrow}) 
      \prod_{r=0}^{m_1-1}\frac{\dd t_{m_1-r,n_1-r}}{ t_{m_1-r,n_1-r}}\Bigg) \notag\\
   &\,\,\qquad\times \Bigg( \int  e^{-u_2 t_{m_2,n_2}}
 \Whitt_{\ga,\hat\ga}(t_{m_1,n_1}^{\nwarrow}) \, 
  \Psi^{(m_1)}_{\gl}(t_{m_1,n_1}^{\nwarrow})
   \prod_{r=0}^{m_1-1}\frac{\dd t_{m_1-r,n_1-r}}{ t_{m_1-r,n_1-r}}\Bigg).
   \end{align} 
   Moreover, we shift the contour of integration from $(\iota\bbR)^{m_1}$ to $(\ell_\delta)^{m_1}$. We do this to avoid the encounter of poles later on. In order to justify both 
   \eqref{eq:2-point-int2} and the shift of contours, we need  
    to check that the two factors inside the parentheses belong to  $L^2((\iota\bbR)^{m_1}; s_{m_1}(\gl) \dd \gl)$. This will be easier to check once we derive some more explicit formulae for these factors. 
   The first integral in \eqref{eq:2-point-int2} can be computed by Stade's identity \ref{thm:Stade} and equals
   \begin{align}\label{eq:2-point-int2.1}
   u_1^{-\sum_{1\leq i\leq m_1} (\gl_i-\ga_i)}\prod_{1\leq i,i'\leq m_1} \Gamma(-\ga_i+\gl_{i'}) . 
   \end{align}
   Notice that the above quantity belongs to $L^2((\iota \bbR)^{m_1};s_{m_1}(\gl)\dd\gl)$ , since by the asymptotics of Gamma function \eqref{gamma_asympto},
    we have that for all values of $\gl_1,...,\gl_{m_1}$ with large imaginary part
      \begin{align}\label{eq:2p_asymptotic1}
  s_{m_1}(\gl) \,\left| u_1^{-\sum_{1\leq i\leq m_1} (\gl_i-\ga_i)}\right|^2 \prod_{1\leq i,i'\leq m_1} |\Gamma(-\ga_i+\gl_{i'})|^2 &\approx 
  \exp\left( \frac{\pi}{2}\sum_{1\leq i\neq i' \leq m_1 }|\gl_i-\gl_{i'}| -\pi m_1\sum_{i=1}^{m_1}|\gl_i| \right)\notag\\
  &\leq  \exp\left( -\pi \sum_{i=1}^{m_1}|\gl_i| \right),
   \end{align}
   which decays exponentially.
   
    If we, now, expand the functions $\Whitt_{\ga,\hat\ga}(t_{m_1,n_1}^{\nwarrow})$ and  $\Psi^{(m_1)}_{\gl}(t_{m_1,n_1}^{\nwarrow})$ in terms of their integral representations \eqref{int_Whitt_m1} and \eqref{int_Whitt_poly}, we recognise that the second integral in \eqref{eq:2-point-int2}, 
    is the Laplace transform of the point-to-point partition function 
    $Z_{(m_2,n_2)}^{(\lambda,\ga;\hat\ga)}$ of a $(\lambda_1,...,\lambda_{m_1},\ga_{m_1+1},...\ga_{m_2}; \hat\ga_1,...,\hat\ga_{n_1})$-log-gamma polymer
     multiplied by the normalising factor
   \begin{align*}
   \prod_{\begin{substack}{1\leq i\leq m_1\\ 1\leq j\leq n_1}\end{substack}} \Gamma(\lambda_i+\hat\ga_j)
    \prod_{\begin{substack}{m_1+1\leq i\leq m_2\\ 1\leq j\leq n_2 }\end{substack}} \Gamma(\ga_i+\hat\ga_j) .
   \end{align*}
We can, therefore, write 
\begin{align}\label{eq:2-point-rec}
&\int  e^{-u_2 t_{m_2,n_2}} \Whitt_{\ga,\hat\ga}(t_{m_1,n_1}^{\nwarrow}) \, 
   \Psi^{(m_1)}_{\gl}(t_{m_1,n_1}^{\nwarrow}) \prod_{r=0}^{m_1-1}\frac{\dd t_{m_1-r,n_1-r}}{ t_{m_1-r,n_1-r}} \notag\\
&\qquad= \prod_{\begin{substack}{1\leq i\leq m_1\\ 1\leq j\leq n_1}\end{substack}} \Gamma(\lambda_i+\hat\ga_j)
    \prod_{\begin{substack}{m_1+1\leq i\leq m_2\\ 1\leq j\leq n_2 }\end{substack}} \Gamma(\ga_i+\hat\ga_j)
    \,\,\bbE\Big[e^{-u_2 Z_{m_2,n_2}^{(\gl,\ga;\hat\ga)}}\Big].
\end{align}
To make this formula more transparent, compare with \eqref{eq:2-point-int}, set $u_1=0$ and recall that $t_{m_2,n_2}$ is the partition function $Z_{m_2,n_2}$.

When $m_2\geq n_2$ we may use the formula provided by Theorem \ref{thm:one_point_Laplace}, applied to the
Laplace transform of $Z_{(m_2,n_2)}^{(\gl,\ga;\hat\ga)}$ with $0<\Re(\delta+\gamma-\hat\ga)$ and $\Re(\delta+\gamma+\ga) >0$, in order to write \eqref{eq:2-point-rec} as
\begin{align}\label{eq:2pointcont}
&     \prod_{\begin{substack}{1\leq i\leq m_1\\ 1\leq j\leq n_1}\end{substack}} \Gamma(\lambda_i+\hat\ga_j)
\prod_{\begin{substack}{m_1+1\leq i\leq m_2\\ 1\leq j\leq n_2 }\end{substack}} \Gamma(\ga_i+\hat\ga_j) 
\int_{(\ell_{\delta+\gamma})^{n_2}} \dd \mu \,s_{n_2}(\mu) \prod_{1\leq j,j'\leq n_2}  \Gamma(-\hat\ga_{j'}+\mu_j)\notag\\
&\hskip 4cm \times
\prod_{j=1}^{n_2}\frac{u_2^{-\mu_j} \prod_{i=1}^{m_1}\Gamma(\mu_j+\gl_i) \prod_{i=m_1+1}^{m_2} \Gamma(\mu_j+\ga_i) }
{u_2^{-\hat\ga_j} \prod_{i=1}^{m_1}\Gamma(\hat\ga_j+\gl_i) \prod_{i=m_1+1}^{m_2} \Gamma(\hat\ga_j+\ga_i)}\notag\\
%&  = \prod_{\begin{substack}{1\leq i\leq m_1\\ n_2+1\leq j\leq n_1}\end{substack}} \Gamma(\lambda_i+\hat\ga_j) 
%	\int_{(\iota\bbR)^{n_2}} \dd \mu \,s_{n_2}(\mu) \prod_{1\leq j,j'\leq n_2}  \Gamma(-\hat\ga_{j'}+\mu_j)\notag\\
%&\hskip 4cm \times
%	\prod_{j=1}^{n_2} u_2^{-(\mu_j-\hat\ga_j)} \prod_{i=1}^{m_1}\Gamma(\mu_j+\gl_i) \prod_{i=m_1+1}^{m_2} \Gamma(\mu_j+\ga_i).
\end{align}
Finally we insert \eqref{eq:2pointcont} and \eqref{eq:2-point-int2.1} into \eqref{eq:2-point-int2} we arrive to formula \eqref{2p_formula}. 
In a similar way we can arrive to \eqref{2p_formulaB}, in the case $m_2<n_2$. The only difference would be that we will need to adapt Theorem \ref{thm:one_point_Laplace} to apply it when $m<n$, but this is straightforward by transposition of the rectangular array.
 \vskip 2mm
It only remains to check that \eqref{eq:2-point-rec} or equivalently \eqref{eq:2pointcont} belongs to $L^2((\iota\bbR)^{m_1}; s_{m_1}(\gl) \dd \gl)$. We prove this separately, in the following lemma, as it requires a particular combinatorial analysis.
\end{proof}
\begin{lemma}\label{lem:L2int}
The quantity \eqref{eq:2-point-rec} (or equivalently \eqref{eq:2pointcont}) belongs to $L^2((\iota\bbR)^{m_1}; s_{m_1}(\gl) \dd \gl)$.
\end{lemma}
\begin{proof}
We work on the $L^2$ integrability of \eqref{eq:2pointcont} and start by
using the Cauchy-Schwarz inequality  (bearing in mind the easily verified fact that the Sklyanin measure $s_{n_2}(\mu)$ is a positive measure):
  \begin{align}
 & \int_{(\iota\bbR)^{m_1}} \dd \gl\,s_{m_1}(\gl) \prod_{\begin{substack}{1\leq i\leq m_1\\ n_2+1\leq j\leq n_1}\end{substack}} \left|\Gamma(\lambda_i+\hat\ga_j) \right|^2 \,
	\Big|\int_{( \ell_{\delta+\gamma})^{n_2}} \dd \mu \,s_{n_2}(\mu) \prod_{1\leq j,j'\leq n_2}  \Gamma(-\hat\ga_{j'}+\mu_j)\notag\\
&\hskip 4cm \times
	\prod_{j=1}^{n_2} u_2^{-(\mu_j-\hat\ga_j)} \prod_{i=1}^{m_1}\Gamma(\mu_j+\gl_i) \prod_{i=m_1+1}^{m_2} \Gamma(\mu_j+\ga_i) \Big|^2 \notag\\
&\leq  \int_{(\iota\bbR)^{m_1}} \dd \gl\, s_{m_1}(\gl) \prod_{\begin{substack}{1\leq i\leq m_1\\ n_2+1\leq j\leq n_1}\end{substack}} \left|\Gamma(\lambda_i+\hat\ga_j) \right|^2
\int_{(\ell_{\delta+\gamma})^{n_2}} \dd \mu \,s_{n_2}(\mu) 
\prod_{\begin{substack}{1\leq i\leq m_1\\ 1\leq j\leq n_2}\end{substack}} \Big|\Gamma(\mu_j+\gl_i) \Big|^2 
\prod_{\begin{substack}{m_1+1\leq i\leq m_2\\ 1\leq j\leq n_2}\end{substack}} \Big|\Gamma(\mu_j+\ga_i) \Big|^2\notag\\
&\hskip 4cm \times   \int_{( \ell_{\delta+\gamma})^{n_2}} \dd \mu \,s_{n_2}(\mu) \,\,\Big|u_2^{-\sum_{j=1}^{n_2}(\mu_j-\hat\ga_j)} \Big|^2
\prod_{1\leq j,j'\leq n_2}   \Big|\Gamma(-\hat\ga_{j'}+\mu_j)\Big|^2 \notag.
\end{align}
The second integral with respect to $\dd\mu$ is independent of the $\lambda$ variables and finite for the same reason as in \eqref{eq:2p_asymptotic1}. To check the finiteness of the rest of the integral  
we use the asymptotics of the gamma function \eqref{gamma_asympto}. Since the variables $(\gl_i), (\mu_j)$ have a constant real part and since we will only be using the 
exponential part of the asymptotics of the gamma function \eqref{gamma_asympto}, which only involves the imaginary part, we may assume that the real part is zero.  
We then have
\begin{align}
&s_{m_1}(\gl)  \,s_{n_2}(\mu)
\,\prod_{\begin{substack}{1\leq i\leq m_1\\ n_2+1\leq j\leq n_1}\end{substack}} \left|\Gamma(\lambda_i+\hat\ga_j) \right|^2
\prod_{\begin{substack}{1\leq i\leq m_1\\ 1\leq j\leq n_2}\end{substack}} \Big|\Gamma(\mu_j+\gl_i) \Big|^2 
\prod_{\begin{substack}{m_1+1\leq i\leq m_2\\ 1\leq j\leq n_2}\end{substack}} \Big|\Gamma(\mu_j+\ga_i) \Big|^2\notag\\
&\approx  \exp\Big( \frac{\pi}{2}\sum_{1\leq i\neq i' \leq m_1 }|\gl_i-\gl_{i'}| +\frac{\pi}{2}\sum_{1\leq j\neq j' \leq n_2 }|\mu_j-\mu_{j'}| 
-\pi \sum_{\begin{substack}{1\leq i\leq m_1\\ n_2+1\leq j\leq n_1}\end{substack}}|\gl_i| \notag\\
&\hskip 4cm
-\pi \sum_{\begin{substack}{1\leq i\leq m_1\\ 1\leq j\leq n_2}\end{substack}}|\gl_i+\mu_j|  
-\pi \sum_{\begin{substack}{m_1+1\leq i\leq m_2\\ 1\leq j\leq n_2}\end{substack}}|\mu_j|
\Big)\notag.
\end{align}
Since the last expression involves only absolute values, we may further assume that $\gl$'s and $\mu$'s are real (the absolute value will absorb the imaginary $\iota$).
Without loss of generality, we may also assume that $\gl_{1}\leq \gl_2\leq \cdots\leq \gl_{m_1}$ and that 
$\mu_1\leq\mu_2\leq\cdots\leq \mu_{n_2}$, so that the above expression writes as
\begin{align}\label{lm}
&\exp\Big( \pi\  \sum_{j=1}^{m_1}(m_1-2j+1) \,\,\gl_{m_1-j+1} +\pi \sum_{j=1}^{n_2} (n_2-2j+1) \, \,\mu_{n_2-j+1}  
 -\pi \sum_{\begin{substack}{1\leq i\leq m_1\\ 1\leq j\leq n_2}\end{substack}}|\gl_i+\mu_j|
 \notag\\
&\hskip 3cm  
-\pi (n_1-n_2)\sum_{1\leq i\leq m_1}|\gl_i|
-\pi (m_2-m_1)\sum_{ 1\leq j\leq n_2 }|\mu_j|
\Big).
\end{align}
Since, in general, $m_1-1\geq n_1-n_2$, it is not obvious that the terms $(m_1-2j+1) \gl_j$ for $j=1,...,\lfloor (n_2+1)/2\rfloor$ and $j=m_1,m_1-1,...,m_1-\lfloor (n_2+1)/2\rfloor +1$ can be dominated by the terms $(n_1-n_2)|\gl_j|$ and 
similarly for the corresponding terms involving the $\mu$ variables, since in general $n_1-1\geq m_2-m_1$.
 We will, therefore, need to combine appropriately the $\lambda_j$ and $\mu_j$ terms in a way that it is clear that the combined terms are dominated by the terms $|\gl_i+\mu_j|$ and thus the factors in front of $\gl_j,\mu_j$
drop below the values $n_1-n_2$ and $m_2-m_1$, respectively (see also Figure \ref{matching}). To do this we first assume that $m_1\geq n_2$. This is no loss of generality, since otherwise we can interchange the roles of $\gl$'s and $\mu$'s in the argument below. However, we still need to distinguish between the case $m_1>n_2$ and $m_1=n_2$.

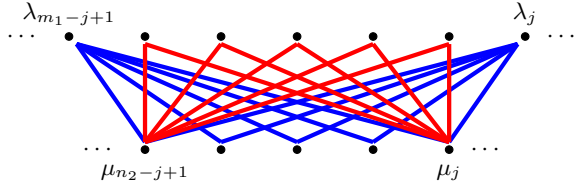
\begin{figure}[t]
 \begin{center}
\begin{tikzpicture}[scale=1.0]

%%%%Diagonall lines%%%%
\draw[blue, ultra thick] (1.1,2.9)--(2,1.6);
\draw[blue, ultra thick] (2,1.6)--(6.9,2.9);

\draw[blue, ultra thick] (1.1,2.9)--(3,1.6);
\draw[blue, ultra thick] (3,1.6)--(6.9,2.9);

\draw[blue, ultra thick] (1.1,2.9)--(4,1.6);
\draw[blue, ultra thick] (4,1.6)--(6.9,2.9);

\draw[blue, ultra thick] (1.1,2.9)--(5,1.6);
\draw[blue, ultra thick] (5,1.6)--(6.9,2.9);

\draw[blue, ultra thick] (1.1,2.9)--(6,1.6);
\draw[blue, ultra thick] (6,1.6)--(6.9,2.9);

\draw[red, ultra thick] (2,1.6)--(2,2.9);
\draw[red, ultra thick] (2,2.9)--(6,1.6);

\draw[red, ultra thick] (2,1.6)--(3,2.9);
\draw[red, ultra thick] (3,2.9)--(6,1.6);

\draw[red, ultra thick] (2,1.6)--(4,2.9);
\draw[red, ultra thick] (4,2.9)--(6,1.6);

\draw[red, ultra thick] (2,1.6)--(5,2.9);
\draw[red, ultra thick] (5,2.9)--(6, 1.6);

\draw[red, ultra thick] (2,1.6)--(6,2.9);
\draw[red, ultra thick] (6,2.9)--(6,1.6);

\draw [fill] (1,3) circle [radius=0.05];
\draw [fill] (2,3) circle [radius=0.05];
\draw [fill] (3,3) circle [radius=0.05];
\draw [fill] (4,3) circle [radius=0.05];
\draw [fill] (5,3) circle [radius=0.05];
\draw [fill] (6,3) circle [radius=0.05];
\draw [fill] (7,3) circle [radius=0.05];

\draw [fill] (2,1.5) circle [radius=0.05];
\draw [fill] (3,1.5) circle [radius=0.05];
\draw [fill] (4,1.5) circle [radius=0.05];
\draw [fill] (5,1.5) circle [radius=0.05];
\draw [fill] (6,1.5) circle [radius=0.05];

\node at (1,3.3) {\small{{\bf $\gl_{m_1-j+1}$}}};
\node at (7,3.3) {\small{{\bf $\gl_{j}$}}};
\node at (2,1.2) {\small{{\bf $\mu_{n_2-j+1}$}}};
\node at (6,1.2) {\small{{\bf $\mu_{j}$}}};
\node at (0.4,3) {\small{{\bf $\cdots$}}};
\node at (7.5,3) {\small{{\bf $\cdots$}}};
\node at (6.5,1.5) {\small{{\bf $\cdots$}}};
\node at (1.4,1.5) {\small{{\bf $\cdots$}}};

\end{tikzpicture}

\end{center}  
\caption{ \small 
This figure corresponds to the case $m_1>n_2$.
Two consecutive blue edges connecting $\gl_{m_1-j+1}$ to $\gl_j$ via a $\mu_i$, for $j\leq i\leq n_2-j+1$, denotes the
rewriting of the difference $\gl_{m_1-j+1} - \gl_j$ as $(\gl_{m_1-j+1} +\mu_i )-(\gl_i+\mu_i)$ The latter is bounded by $|\gl_{m_1-j+1} +\mu_i |+|\gl_i+\mu_i|$ and cancelled 
by the corresponding cross term in \eqref{lm}. We perform $n_2-2j+2$ such matching, thus
reducing the pre factors in front of $\gl_{m_1-j+1}$ and $ \gl_j$, in \eqref{lm}, to $\pm(m_1-n_2-1)$. Similar is the role of the red edges. That is, to write $\mu_{n_2-j+1}-\mu_j$ as
$(\mu_{n_2-j+1}+\gl_i)-(\mu_j+\gl_i)$, for $i=j+1,...,n_2-j+1$ and subsequently bound it by $|\mu_{n_2-j+1}+\gl_i|+|\mu_j+\gl_i|$ and cancel it with the corresponding cross term in \eqref{lm}.
Notice that we have performed the matchings in such a way that no edge connecting a $\gl$ and a $\mu$ is taken twice, i.e. once with a red and once with a blue colour. 
}  
\label{matching}
\end{figure}

In the case $m_1>n_2$ we proceed as follows:
For $j=1,...,\lfloor (n_2+1)/2 \rfloor$ we write, by adding and subtracting the terms $\mu_i$, for $i=j,...,(n_2-j+1)$ (this is what in Figure \ref{matching} we called ``matching''),
\begin{align}\label{eq:aux_lm1}
&(m_1-2j+1)\,\gl_{m_1-j+1}-(m_1-2j+1)\,\gl_j \notag\\
&= (m_1-n_2-1) \,\big( \gl_{m_1-j+1}-\gl_j\big)
+ (n_2-2j+2)\, \gl_{m_1-j+1} - (n_2-2j+2)\, \gl_{j} \notag\\
&=(m_1-n_2-1) \,\big( \gl_{m_1-j+1}-\gl_j\big) 
+\sum_{i=j}^{n_2-j+1} (\gl_{m_1-j+1}+\mu_i) -\sum_{i=j}^{n_2-j+1} (\gl_{j}+\mu_i) \notag\\
&\leq (m_1-n_2-1) \,\big(|\gl_{m_1-j+1}|+ |\gl_j|\big) 
+\sum_{i=j}^{n_2-j+1} |\gl_{m_1-j+1}+\mu_i| +\sum_{i=j}^{n_2-j+1} |\gl_{j}+\mu_i|.
\end{align}
Similarly, by adding and subtracting the terms $\gl_i$, for $i=j+1,...,n_2-j+1$, we write
\begin{align}\label{eq:aux_lm2}
(n_2-2j+1)\mu_{n_2-j+1}-(n_2-2j+1)\mu_j
&= \sum_{i=j+1}^{n_2-j+1} (\mu_{n_2-j+1}+\gl_i) -\sum_{i=j+1}^{n_2-j+1}(\mu_j+\gl_i) \notag\\
&\leq \sum_{i=j+1}^{n_2-j+1} |\mu_{n_2-j+1}+\gl_i| +\sum_{i=j+1}^{n_2-j+1}|\mu_j+\gl_i |.
\end{align}
We now write the exponent in \eqref{lm} as (for convenience we ignore the factors $\pi$)
\begin{align}\label{lm2}
& \sum_{j=1}^{m_1}(m_1-2j+1) \,\,\gl_{m_1-j+1} + \sum_{j=1}^{n_2} (n_2-2j+1) \, \,\mu_{n_2-j+1}  
 - \sum_{\begin{substack}{1\leq i\leq m_1\\ 1\leq j\leq n_2}\end{substack}}|\gl_i+\mu_j|
 \notag\\
&\hskip 3cm  
- (n_1-n_2)\sum_{1\leq i\leq m_1}|\gl_i|
- (m_2-m_1)\sum_{ 1\leq j\leq n_2 }|\mu_j| \notag\\
&= \sum_{j=1}^{\lfloor (n_2+1)/2\rfloor}(m_1-2j+1) \,\big(\gl_{m_1-j+1} -\gl_{j}\big)+
\sum_{j=\lfloor(n_2+1)/2\rfloor +1}^{m_1-\lfloor(n_2+1)/2\rfloor}(m_1-2j+1) \,\gl_{m_1-j+1}\notag \\
&\qquad\qquad+ \sum_{j=1}^{\lfloor (n_2+1)/2 \rfloor} (n_2-2j+1) \, \big(\mu_{n_2-j+1} -\mu_j \big)
 - \sum_{\begin{substack}{1\leq i\leq m_1\\ 1\leq j\leq n_2}\end{substack}}|\gl_i+\mu_j|
 \notag\\
&\hskip 3cm  
- (n_1-n_2)\sum_{1\leq i\leq m_1}|\gl_i|
- (m_2-m_1)\sum_{ 1\leq j\leq n_2 }|\mu_j|.
\end{align}
We insert in this the inequalities \eqref{eq:aux_lm1}, \eqref{eq:aux_lm2} summed up over $j=1,...,\lfloor (n_2+1)/2 \rfloor$, and notice that since for 
$j=\lfloor(n_2+1)/2\rfloor +1,...,m_1-\lfloor(n_2+1)/2\rfloor$ it holds that $|m_1-2j+1| \leq m_1-n_2-1$, it follows that \eqref{lm2} is bounded by
\begin{align*}
- (n_1-m_1-1)\sum_{1\leq i\leq m_1}|\gl_i|
- (m_2-m_1)\sum_{ 1\leq j\leq n_2 }|\mu_j|.
\end{align*}
Since $m_2>m_1$ and $n_1\geq m_1$ the desired integrability follows.

\vskip 2mm
In the case $m_1=n_2$ we do a slightly different matching by writing
\begin{align*}
&(m_1-2j+1) \big(\gl_{m_1-j+1}-\gl_j\big) +(n_2-2j+1)\big(\mu_{n_2-j+1}-\mu_j \big)\\
& = \big( \gl_{m_1-j+1}+ \mu_{n_2-j+1}\big) +\sum_{i=j+1}^{m_1-j}\big( \gl_{m_1-j+1}+\mu_i \big) - 
\sum_{i=j+1}^{m_1-j}\big( \gl_{j}+\mu_i \big)\\
&\qquad +\sum_{i=j+1}^{n_2-j}\big( \mu_{n_2-j+1}+\gl_i \big) - 
\sum_{i=j+1}^{n_2-j}\big( \mu_{j}+\gl_i \big) - \big(\mu_j+\gl_j\big)\\
&\leq \big |\gl_{m_1-j+1}+ \mu_{n_2-j+1}\big| +\sum_{i=j+1}^{m_1-j}\big| \gl_{m_1-j+1}+\mu_i \big| + 
\sum_{i=j+1}^{m_1-j}\big| \gl_{j}+\mu_i \big|\\
&\qquad +\sum_{i=j+1}^{n_2-j}\big| \mu_{n_2-j+1}+\gl_i \big| +
\sum_{i=j+1}^{n_2-j}\big| \mu_{j}+\gl_i \big| +\big|\mu_j+\gl_j\big|
\end{align*}
and finally we proceed as in the previous case.

  \end{proof}
  \begin{remark}\rm
The proof of the above theorem and in particular the derivation of relation \eqref{eq:2-point-rec} shows a way to obtain a contour integral formula
for the joint Laplace transform of $k$ point-to-point partition functions via the following recursion (we do not attempt to write the explicit formulae, which would be
 complicated, or to check the various $L^2$ conditions for the application of the Plancherel theorem):

 Let $Z_{(m_1,n_1)},...,Z_{(m_k,n_k)}$ be the partition functions from $(1,1)$ to $(m_1,n_1),...,(m_k,n_k)$, respectively, of an $(\ga,\hat\ga)$-log-gamma polymer. Assume that $m_1\leq m_2\leq\cdots\leq m_k $ and $n_1\geq n_2\geq \cdots\geq n_k$. Moreover, without loss of generality, assume that $m_1\leq n_1$. The joint Laplace transform satisfies the following recursive relation
 \begin{align*}
 &\bbE\Big[e^{- u_1 Z_{(m_1,n_1)}-\cdots-u_kZ_{(m_k,n_k)}}\Big] \\
 &=\frac{1}{ \prod\limits_{ i,j}\Gamma(\ga_i+\hat\ga_j)}
 \int_{(\iota \bbR)^{m_1} } \dd \lambda
    \,\,s_{m_1}(\lambda)   \,u_1^{-\sum_{1\leq i\leq m_1} (\gl_i-\ga_i)}\, \prod_{1\leq i,i'\leq m_1}\Gamma(-\ga_i+\gl_{i'})      \\
 &\hskip 2cm 
    \times\prod_{\begin{substack}{1\leq i\leq m_1\\ 1\leq j\leq n_1}\end{substack}} \Gamma(\lambda_i+\hat\ga_j)
    \prod_{\begin{substack}{m_1+1\leq i\leq m_2\\ 1\leq j\leq n_2 }\end{substack}} \Gamma(\ga_i+\hat\ga_j)
     \, \bbE\Big[e^{- u_2 Z_{(m_2,n_2)}^{(\gl,\ga;\hat\ga)}-\cdots-u_kZ_{(m_k,n_k)}^{(\gl,\ga;\hat\ga)}}\Big] ,
    \end{align*}
    where $Z_{(m_j,n_j)}^{(\gl,\ga;\hat\ga)}, j=2,...,k$, are the point-to-point partition functions corresponding to a 
    log-gamma polymer with parameters
    $(\gl_1,...,\gl_{m_1},\ga_{m_1+1},...,\ga_{m_k}; \hat\ga_1,...,\hat\ga_{n_1})$.
\end{remark}
\begin{remark}\label{rem:OY}\rm
         Let us show how our formula for the joint Laplace transform of two equal-time log-gamma partition functions provides the corresponding formula for the joint Laplace transform of the
	 semi-discrete directed polymer in a Brownian environment studied in \cite{O12} and often called the {\it O'Connell-Yor polymer}. By taking an appropriate scaling limit (this is a straightforward modification of Section 4.2 of \cite{COSZ14} 
	 in order to include Brownian motions with drift) we can recover the semi-discrete directed polymer
	 partition function from the log-gamma polymer partition function. More precisely, let  $Z_{(m_1,Nt_1)}^{(N)}$  be the 
	 partition function of a log-gamma polymer at point $(m_1,Nt_1)$ with parameters $(\ga_i,\hat\ga_j)=(\ga_i,N)$, for $i,j\geq 1$. 
	 Let also $(B^i(\cdot))_{i\geq 1}$ be independent Brownian motions with drifts $\ga_i, i\geq 1$.
	  We then have the following convergence in law, as $N\to \infty$,
		\begin{equation*}
		\log\Big(N^{Nt_1} Z^{(N)}_{(m_1,Nt_1)}\Big)-\frac{t_1}{2}\to \log \int ...\int_{0\leq s_1\leq...\leq s_{m_1-1}\leq t_1} \exp\Bigg(\sum_{i=1}^{m_1}\big(B^i(s_i)-B^i(s_{i-1})\big)\Bigg) ds_1...ds_{m_1-1}
		\end{equation*}
		which is equivalent to
		\begin{equation*}
		\bbE\Big[e^{-u_1^N Z^{(N)}_{(m_1,Nt_1)}}\Big]\xrightarrow[N\to\infty]{} \bbE\Big[e^{-u_1Z^{OY}(m_1,t_1)}\Big]
		\end{equation*}
		where $u_1^N=\frac{u_1N^{Nt_1}}{(\sqrt{e})^{t_1}}$ and $Z^{OY}(m_1,t_1)$ is the partition function of the O'Connell-Yor
		 semi discrete polymer at point $(m_1,t_1)$. In the case of two points 
		$(m_1,Nt_1)$ and $(m_2,Nt_2)$, such that $m_1<m_2$ and $t_1>t_2$,  we have 
		\begin{equation}\label{OY:conv}
		\lim_{N\to\infty}\bbE\Big[e^{-u_1^N Z^{(N)}_{(m_1,Nt_1)}-u_2^NZ^{(N)}_{(m_2,Nt_2)}}\Big]= \bbE\Big[e^{-u_1Z^{OY}(m_1,t_1)-u_2Z^{OY}(m_2,t_2)}\Big]
		\end{equation}
		where $u_1^N=\frac{u_1N^{Nt_1}}{(\sqrt{e})^{t_1}}$ and $u_2^N=\frac{u_2N^{Nt_2}}{(\sqrt{e})^{t_2}}$. The left hand side corresponds to  the joint Laplace transform for
		 partition functions at points $(m_1,Nt_1)$ and $(m_2,Nt_2)$, hence it is given by the formula (\ref{2p_formulaB}) with 
		 parameters $(\ga_i,\hat{\alpha}_j)=(\ga_i,N), \,i,j\geq 1$. By using the following convergence 
		\begin{equation*}
		\lim_{N\to\infty}\Big(\frac{\Gamma(\mu_j+N)}{\Gamma(\alpha_j+N)}\Big)^{t_2N}\frac{(N^N/\sqrt{e})^{-\mu_jt_2}}{(N^N/\sqrt{e})^{-\alpha_jt_2}}=\exp\big(\,\frac{t_2}{2}(\mu_j^2-\alpha_j^2)\big),
		\end{equation*}
		we can take the limit of the left hand side in (\ref{OY:conv}) and obtain the Laplace transform for two points of semi-discrete polymer
		\begin{align}
		&\bbE\Big[e^{-u_1Z^{OY}(m_1,t_1)-u_2Z^{OY}(m_2,t_2)}\Big] \notag\\
		= &\int_{(\ell_{\delta})^{m_1} } \dd \lambda
		\,\,s_{m_1}(\lambda)  \prod_{1\leq i,i'\leq m_1} \Gamma(-\ga_i+\gl_{i'}) \prod_{i=1}^{m_1}\frac{(u_1/u_2)^{-\gl_i}\, }{ (u_1/u_2)^{-\ga_i}} \exp\Big(\frac{t_1-t_2}{2}(\gl_i^2-\ga_i^2)\Big) \notag\\
		&\quad\times   \int_{(\ell_{\delta'})^{m_2} } \dd \mu  
		\,\,s_{m_2}(\mu)  
		\prod_{\begin{substack}{m_1+1\leq i\leq m_2\\ 1\leq i'\leq m_2}\end{substack}} 
		\Gamma(-\ga_i+\mu_{i'}) 
		\prod_{i=1}^{m_2}\frac{u_2^{-\mu_i}}{ u_2^{-\ga_i}}\exp\Big(\frac{t_2}{2}(\mu_i^2-\ga_i^2)\Big)   \notag\\
		&\qquad\quad\times
		\,\prod_{\begin{substack}{1\leq i\leq m_1\\ 1\leq i'\leq m_2}\end{substack}} 
		\Gamma(-\gl_i+\mu_{i'}) \,.
		\end{align}
	
\end{remark}

\section{PNG and geometric PNG}\label{sec:gPNG}
In this section we define the geometric PNG (gPNG) and use this to obtain alternative integral formulae for the joint Laplace transform of the point-to-point partition functions. Naturally,
 gPNG is a geometric lifting of PNG and we start by reviewing the latter.
 \vskip 2mm
 \subsection{Polynuclear Growth Model and local moves}
PNG can be viewed as a construction of an ensemble of nonintersecting paths. In this process plateaux of certain (random) heights are created and, once created, they grow horizontally at unit speed. When two growing plateaux overlap in the process of growth, the overlapping area drops one level below creating a new plateau. At the same time new plateaux may be created on top of the top-most interface and the process
continues in this fashion. Let us describe this in more detail and let us start with an input matrix $W=(w_{ij})\in(\R_{>0})^{n\times n}$, which for simplicity we take to be square. This matrix encodes the heights of the boxes of unit width that will be created and added on top of the interface.
It will be clear that PNG can be encoded in terms of local transformations $\ell_{ij}$, which we now define:
For each $2\le i\le n$ and $2\le j\le m$ define 
a mapping $\ell_{ij}$ which takes as input a matrix $W=(w_{ij})\in(\R_{>0})^{n\times m}$ and replaces the submatrix
$$ \begin{pmatrix} w_{i-1,j-1}& w_{i-1,j}\\ w_{i,j-1}& w_{ij}\end{pmatrix}$$
of $W$ by its image under the map
\begin{equation}\label{tabcd}
\begin{pmatrix} a& b\\ c& d\end{pmatrix} \qquad \mapsto \qquad 
\begin{pmatrix} b\wedge c - a & b\\ c& d+b\vee c \end{pmatrix},
\end{equation}
and leaves the other elements unchanged.  For $2\le i\le n$ and $2\le j\le m$, 
define $\ell_{i1}$ to be the mapping that replaces
the element $w_{i1}$ by $w_{i-1,1}+w_{i1}$ and $\ell_{1j}$ to be the mapping that replaces the element $w_{1j}$ by 
$w_{1,j-1}+w_{1j}$.  By convention we define $\ell_{11}$ to be the identity map.

\begin{figure}[t]
	\begin{center}
		\begin{tikzpicture}[scale=.7]

		\draw[-,thick] (-3,0)--(3,0);

		\draw[-,thick](-0.5,0)--(-0.5,1.5)--(0.5,1.5)--(0.5,0);
		\draw[<->,red] (1.5,0)--(1.5,1.5);
		\draw[->,green](0.5,0.7)--(1,0.7);
		\draw[->,green](-0.5,0.7)--(-1,0.7);
		\node at (2,1) {\small{{\bf $w_{11}$}}};
		
		%\draw[-](8,1)--(8,4)--(9,4)--(9,2)--(10,2)--(10,3)--(11,3)--(11,1);
		%\draw[-](16,1)--(16,5)--(17,5)--(17,4)--(18,4)--(18,5)--(19,5)--(19,3)--(20,3)--(20,6)--(21,6)--(21,1);
		%\draw[-] (18,1)--(18,2)--(19,2)--(19,1);
		\fill[black]  (0,0.5) circle [radius=0.1];
		\fill[black]  (1,1) circle [radius=0.1];
		\fill[black]  (-1,1) circle [radius=0.1];
		\fill[black]  (0,2) circle [radius=0.1];				
		\fill[black]  (2,2) circle [radius=0.1];
		\fill[black]  (-2,2) circle [radius=0.1];
		
		\draw[-,thick](8,0)--(16,0);
		\draw[-,thick] (10.5,0)--(10.5,1.5)--(13.5,1.5)--(13.5,0);
		\draw[-,thick,red] (10.5,1.5)--(10.5,2.5)--(11.5,2.5)--(11.5,1.5);
		\draw[-,thick,red] (12.5,1.5)--(12.5,2)--(13.5,2)--(13.5,1.5);		
		\fill[black]  (12,0.5) circle [radius=0.1];
		\fill[black]  (13,1) circle [radius=0.1];
		\fill[black]  (11,1) circle [radius=0.1];
		\fill[black]  (12,2) circle [radius=0.1];				
		\fill[black]  (14,2) circle [radius=0.1];
		\fill[black]  (10,2) circle [radius=0.1];
		\draw[<->,thick,red] (9.5,0)--(9.5,2.5);
		\draw[<->,red,thick] (14.5,0)--(14.5,2);
		\draw[<->,red,thick] (10.3,1.5)--(10.3,2.5);
		\draw[<->,red,thick] (13.7,1.5)--(13.7,2);	
	    \draw[->,green](13.5,0.7)--(14,0.7);
	    \draw[->,green](10.5,0.7)--(10,0.7);
		\draw[->,green](13.5,2.2)--(14,2.2);
		\draw[->,green](10.5,2.6)--(10,2.6);
		\draw[->,green](12.5,2.2)--(12,2.2);
	    \draw[->,green](11.5,2.6)--(12,2.6);    
	    \node at (14.1,1.7) {\small{{\bf $w_{12}$}}};
	    \node at (10.5,2.8) {\small{{\bf $w_{21}$}}};
	    \node at (8.3,1.5) {\small{{\bf $w_{11}+w_{21}$}}};
	    \node at (15.7,1) {\small{{\bf $w_{11}+w_{12}$}}};	
%		\draw[blue, ultra thick]  (8.5,4) circle [radius=0.07];
%		\draw[blue, ultra thick]  (9.5,2) circle [radius=0.07];
%		\draw[blue, ultra thick]  (10.5,3) circle [radius=0.07];
%		\draw[blue, ultra thick]  (16.5,5) circle [radius=0.07];
%		\draw[blue, ultra thick]  (17.5,4) circle [radius=0.07];
%		\draw[blue, ultra thick]  (18.5,5) circle [radius=0.07];
%		\draw[blue, ultra thick]  (19.5,3) circle [radius=0.07];
%		\draw[blue, ultra thick]  (20.5,6) circle [radius=0.07];
%		\draw[blue, ultra thick]  (18.5,2) circle [radius=0.07];
%		\node[below] at (3.5,1){at time t=1};
%		\node[below] at (18.5,1){at time t=3};
%		\node[below] at (9.5,1){at time t=2};
		\end{tikzpicture}
	\end{center}  
	
	\caption{ \small  The PNG at time $t=1,2$. The dots represent the entries $(w_{ij})$ of the input array, determining the nucleations and the heights of the plateaux. 
	Once created, the plateaux expand left/right at unit speed.
	}
\end{figure}
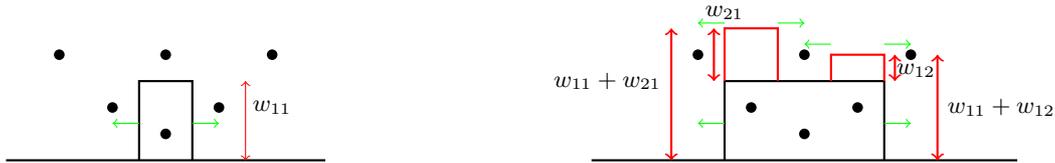
 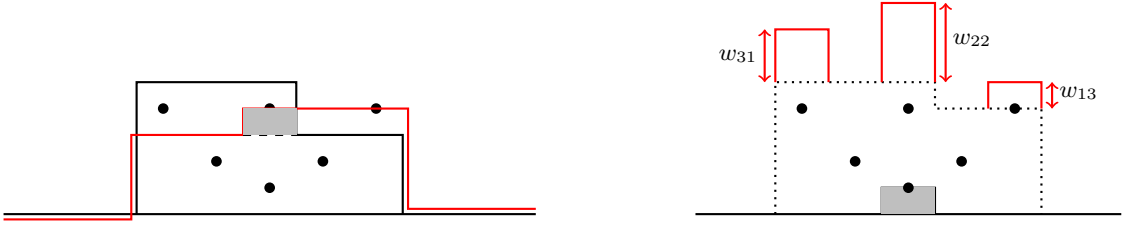
\begin{figure}[t]
 	\begin{center}
 		\begin{tikzpicture}[scale=0.7]
 		\fill[black]  (0,0.5) circle [radius=0.1];
 		\fill[black]  (1,1) circle [radius=0.1];
 		\fill[black]  (-1,1) circle [radius=0.1];
 		\fill[black]  (0,2) circle [radius=0.1];				
 		\fill[black]  (2,2) circle [radius=0.1];
 		\fill[black]  (-2,2) circle [radius=0.1];
 		\draw[black, thick](-5,0)--(5,0);
 		\draw[-,black, thick](-2.5,0)--(-2.5,2.5)--(0.5,2.5)--(0.5,1.5)--(2.5,1.5)--(2.5,0);
 		\draw[-,red,thick] (-5,-0.1)--(-2.6,-0.1)--(-2.6,1.5)--(-0.5,1.5)--(-0.5,2)--(2.6,2)--(2.6,0.1)--(5,0.1);
 		\draw[dashed,thick](-0.5,1.5)--(0.5,1.5);
 		\fill[lightgray](-0.5,1.5)--(0.5,1.5)--(0.5,2)--(-0.5,2)--(-0.5,1.5);

 		\draw[-,thick](8,0)--(16,0);
 		\draw[dotted,thick] (9.5,0)--(9.5,2.5)--(12.5,2.5)--(12.5,2)--(14.5,2)--(14.5,0);
 		\draw[-,thick,red] (9.5,2.5)--(9.5,3.5)--(10.5,3.5)--(10.5,2.5);
 		\draw[-,thick,red] (11.5,2.5)--(11.5,4)--(12.5,4)--(12.5,2.5);	
 		\draw[-,thick,red] (13.5,2)--(13.5,2.5)--(14.5,2.5)--(14.5,2);
 		\draw[-,thick] (11.5,0)--(11.5,0.5)--(12.5,0.5)--(12.5,0);
 		\fill[lightgray](11.5,0)--(11.5,0.5)--(12.5,0.5)--(12.5,0)--(11.5,0);			
 		\fill[black]  (12,0.5) circle [radius=0.1];
 		\fill[black]  (13,1) circle [radius=0.1];
 		\fill[black]  (11,1) circle [radius=0.1];
 		\fill[black]  (12,2) circle [radius=0.1];				
 		\fill[black]  (14,2) circle [radius=0.1];
 		\fill[black]  (10,2) circle [radius=0.1];
 		
 		\draw[<->,thick,red] (9.3,2.5)--(9.3,3.5);
 		\draw[<->,red,thick] (12.7,2.5)--(12.7,4);
 		\draw[<->,red,thick] (14.7,2)--(14.7,2.5);

		 \node at (8.8,3) {\small{{\bf $w_{31}$}}};
		 \node at (13.2,3.3) {\small{{\bf $w_{22}$}}};
		 \node at (15.2,2.3) {\small{{\bf $w_{13}$}}};	
 	
 		\end{tikzpicture}
 	\end{center}  
 	\caption{ \small  At $t=3$, an overlap of height $\min(h_{12},h_{21})-h_{11}$ drops down and new plateaux are created at $x=-2,0,2$.
 	}
 \end{figure}
Keeping these definitions in mind, we now start describing the geometric construction of PNG 
(we clarify that by ``geometric construction''
we do not mean the {\it geometric PNG}).
 It is visually more natural to consider coordinates $(x,t)=(i-j, i+j-1), 1\le i,j\le n$, where $t$ stands for time and $x$ for the spatial variable. We have:
\begin{itemize}
\item[$\bullet$]
Let us start with a single line of height $h=0$ at time $t=0$. 
At time $t=1$ a plateau of height $w_{11}$, of width $1$ and centred at zero, is created. It will then start growing leftwards and rightwards at a unit speed.
\item[$\bullet$]
At time $t=2$ two new plateaux of heights $w_{12}$ and $w_{21}$ will be created on top of the plateau of height $w_{11}$, which by this moment extends from $-3/2$ to $3/2$. The new plateaux are centred 
 at $x=-1$ and $x=1$, respectively (notice that these are the values of the differences $i-j$). The outer shape of this pile of plateaux will be a line with vertical up and down steps, which will have heights
  \begin{align*}
  h^{(1)}_{12}&=w_{12}+w_{11} , \,\,\quad \text{for}\,\,\, x\in[-3/2,-1/2)\\
  h^{(1)}_{11}&=w_{11}              , \quad\quad\quad\quad \text{for}\,\,\, x\in[-1/2,1/2)\\
  h^{(1)}_{21}&=w_{21}+w_{11} , \,\,\quad \text{for}\,\,\, x\in[1/2,3/2),
  \end{align*}
  and has zero height otherwise.
 Subsequently, the new plateaux will start growing at a unit speed  leftwards and rightwards simultaneously with the 
 underlying plateau of height $w_{11}$.
 \item[$\bullet$]
 The third step is a bit more convolved. First, at time $t=3$, two new boxes of heights $w_{13}, w_{31}$ and unit width will be added on the extended plateau created at time $t=2$, above $x=-2$ and $x=2$, respectively, thus creating two new plateaux of heights $ h^{(1)}_{12}+w_{13}$ and $ h^{(1)}_{21}+w_{31}$. At the same time,
 the two plateaux which were created at time $t=2$, will have grown and they extend on top of $[-5/2,1/2]$ and $[-1/2,5/2]$, therefore creating an
 overlap, which extends on top of $[-1/2,1/2]$. The overlapping region of height $h^{(1)}_{12} \wedge h^{(1)}_{21}-h^{(1)}_{11}$ will drop below, on the zero level, creating a new plateau centred at $0$, having width one and height $h^{(2)}_{11}:=h^{(1)}_{12} \wedge h^{(1)}_{21}-h^{(1)}_{11}$. At the same time, on top of the topmost plateau which extends above $[-1/2,1/2]$, a new box of height $w_{22}$ will be added, creating eventually
 a plateau of height $h^{(1)}_{11}=w_{22}+ h^{(1)}_{12} \vee h^{(1)}_{21}$. So at this stage two lines have been created. The top one with heights
 \[
 h^{(1)}(x,3)=
 \begin{cases}
 h^{(1)}_{13}= h^{(1)}_{12}+w_{13},  & \text{for}\,\,\, x\in[-5/2,-3/2)\\
 h^{(1)}_{12}= w_{11}, &\text{for}\,\,\, x\in[-3/2,-1/2)\\
 h^{(1)}_{11}= w_{22}+ h^{(1)}_{12} \vee h^{(1)}_{21},& \text{for}\,\,\, x\in[-1/2,1/2)\\
 h^{(1)}_{21}=w_{11}, & \text{for}\,\,\, x\in[1/2,3/2)\\
 h^{(1)}_{31} = h^{(1)}_{21} +w_{31},  &\text{for}\,\,\, x\in[3/2,5/2)\\
 \end{cases}
 \] 
 and zero elsewhere and the lower one with heights 
  \[
 h^{(2)}(x,3)= h^{(2)}_{11} =h^{(1)}_{12} \wedge h^{(1)}_{21}-h^{(1)}_{11}  , \,\,\quad \text{for}\,\,\, x\in[-1/2,1/2)
 \]
 and  zero elsewhere. 
 \item[$\bullet$]
The same expansion, nucleation and drop down process continues. So at time $t=4$ the top line $h^{(1)}$ will have height $h^{(1)}_{13}+w_{13}$ above interval $[-7/2,5/2)$. Above interval 
$[-5/2,-3/2)$ a drop down, due to overlap of the expanded plateaux, as well an additional nucleation will take place. This action is encoded by the 
 transformation $\ell_{23}$ as
\begin{align*}
&\begin{pmatrix}
 h^{(1)}_{12}& h^{(1)}_{13} \\
h^{(1)}_{22}& w_{23}
\end{pmatrix}
\stackrel{\ell_{23}}{\longmapsto}
\begin{pmatrix}
h^{(1)}_{13}\wedge h^{(1)}_{22}  -h^{(1)}_{12}& h^{(1)}_{13} \\
h^{(1)}_{22} & w_{23} +h^{(1)}_{13}\vee h^{(1)}_{22} 
\end{pmatrix}
\end{align*}
The lower-right entry encodes the new height of the top line $h^{(1)}$ due to the overlap and nucleation and the upper-left entry encodes
 the height of the overlapping area, which will then drop down to add on top of the second line $h^{(2)}$. The latter drop down and addition on $h^{(2)}$ 
 is encoded by the transformation
\begin{align*}
&\begin{pmatrix}
h^{(2)}_{11}\,, & h^{(1)}_{13}\wedge h^{(1)}_{22}  -h^{(1)}_{12} 
\end{pmatrix}
\stackrel{\ell_{12}}{\longmapsto}
\begin{pmatrix}
h^{(2)}_{11}\,, &h^{(1)}_{12}+ (h^{(1)}_{13}\wedge h^{(1)}_{22}  -h^{(1)}_{12}) 
\end{pmatrix}=: 
\begin{pmatrix}
h^{(2)}_{11}\,, &h^{(2)}_{12} 
\end{pmatrix}
\end{align*}
with $\begin{pmatrix}
h^{(2)}_{11}\,, &h^{(2)}_{12} 
\end{pmatrix}$ recording the heights of the second line on top of $[-1/2,1/2)$ and $[-3/2,-1/2)$, respectively.
Similarly, the drop down and nucleation on top of $[1/2,3/2)$ is described via transformations  $\ell_{32}$ and $\ell_{21}$ as (we show both transformations in one step)
\begin{align*}
&\begin{pmatrix}
h^{(2)}_{11}& h^{(2)}_{12} & \cdots\\ 
h^{(1)}_{21}& h^{(1)}_{22} & \cdots\\
h^{(1)}_{31}& w_{32} & \cdots\\
\vdots & \vdots & 
\end{pmatrix}
\stackrel{\ell_{32}}{\longmapsto}
\begin{pmatrix}
h^{(2)}_{11}& h^{(2)}_{12} & \cdots\\ 
h^{(1)}_{22}\wedge h^{(1)}_{31} -h^{(1)}_{21}& h^{(1)}_{22} & \cdots\\
h^{(1)}_{31}& w_{32}+ h^{(1)}_{22} \vee h^{(1)}_{31} & \cdots\\
\vdots & \vdots & 
\end{pmatrix}\\
&\stackrel{\ell_{21}}{\longmapsto}
\begin{pmatrix}
h^{(2)}_{11}& h^{(2)}_{12} & \cdots\\ 
h^{(2)}_{11}+ (h^{(1)}_{22}\wedge h^{(1)}_{31} -h^{(1)}_{21}) & h^{(1)}_{22} & \cdots\\
h^{(1)}_{31}& w_{32}+ h^{(1)}_{22} \vee h^{(1)}_{31} & \cdots\\
\vdots & \vdots & 
\end{pmatrix}
=: 
\begin{pmatrix}
h^{(2)}_{11}& h^{(2)}_{12} & \cdots\\ 
h^{(2)}_{21}& h^{(1)}_{22} & \cdots\\
h^{(1)}_{31}& h^{(1)}_{32} & \cdots\\
\vdots & \vdots & 
\end{pmatrix}.
\end{align*}
Finally, a nucleation of a box of height $w_{41}$ will take place on top of the expanded top line, $h^{(1)}$. This can be encoded via transformation $\ell_{41}$
\begin{align*}
\begin{pmatrix}
\vdots &\\
h^{(1)}_{31} & \cdots\\
w_{41} & \cdots\\
\vdots & 
\end{pmatrix}
&\stackrel{\ell_{41}}{\longmapsto}
\begin{pmatrix}
\vdots &\\
h^{(1)}_{31} & \cdots\\
w_{41} + h^{(1)}_{31} & \cdots\\
\vdots & 
\end{pmatrix} 
=:
\begin{pmatrix}
\vdots &\\
h^{(1)}_{31} & \cdots\\
h^{(1)}_{41}  & \cdots\\
\vdots & 
\end{pmatrix}.
\end{align*}
\item[$\bullet$] The same rules of nucleation and drop down will continue in an inductive manner. This process will result in an ensemble of nonintersecting lines.  
\end{itemize}
The overall evolution for the first four steps can be given as follows
\begin{align*}
&\begin{pmatrix} w_{11}& w_{12} &w_{13} & w_{14} & \cdots\\ 
w_{21}& w_{22} &w_{23} & w_{24} & \cdots\\
w_{31}& w_{32} &w_{33} &  w_{34} &\cdots\\
w_{41}& w_{42} &w_{43} &  w_{44} & \cdots\\
\vdots & \vdots & \vdots &\vdots
\end{pmatrix} 
\xrightarrow[t=1]{h^{(1)}_{11}=w_{11}}
\begin{pmatrix}
h^{(1)}_{11}& w_{12} &w_{13} & w_{14} & \cdots\\ 
w_{21}& w_{22} &w_{23} &  w_{24}  &\cdots\\
w_{31}& w_{32} &w_{33} & w_{34} &\cdots\\
w_{41}& w_{42} &w_{43} &  w_{44} & \cdots\\
\vdots & \vdots & \vdots & \vdots 
\end{pmatrix}\\
&\xrightarrow[t=2]{\ell_{21}\circ \ell_{12}}
\begin{pmatrix}
h^{(1)}_{11}& h^{(1)}_{12} &w_{13} & w_{14} &\cdots\\ 
h^{(1)}_{21}& w_{22} &w_{23} & w_{24} &\cdots\\
w_{31}& w_{32} &w_{33} & w_{34} & \cdots\\
w_{41}& w_{42} &w_{43} &  w_{44} & \cdots\\
\vdots & \vdots & \vdots &\vdots
\end{pmatrix}
\xrightarrow[t=3]{\ell_{31} \circ \ell_{22}\circ \ell_{13}}
\begin{pmatrix}
h^{(2)}_{11}& h^{(1)}_{12} &h^{(1)}_{13} & w_{14} & \cdots\\ 
h^{(1)}_{21}& h^{(1)}_{22} &w_{23} & w_{24} & \cdots\\
h^{(1)}_{31}& w_{32} &w_{33} & w_{34} & \cdots\\
w_{41}& w_{42} &w_{43} &  w_{44} & \cdots\\
\vdots & \vdots & \vdots& \vdots
\end{pmatrix}\\
&\xrightarrow[t=4]{\ell_{41} \,\circ \, (\ell_{21}\circ\ell_{32}) \,\circ \,(\ell_{12}\circ \ell_{23}) \circ \ell_{14}}
\begin{pmatrix}
h^{(2)}_{11}& h^{(2)}_{12} &h^{(1)}_{13} & h^{(1)}_{14} & \cdots\\ 
h^{(2)}_{21}& h^{(1)}_{22} &h^{(1)}_{23} & w_{24} & \cdots\\
h^{(1)}_{31}& h^{(1)}_{32} &w_{33} & w_{34} & \cdots\\
 h^{(1)}_{41}& w_{42} &w_{43} &  w_{44} & \cdots\\
\vdots & \vdots & \vdots& \vdots
\end{pmatrix}
\cdots
\end{align*}  
Introducing the transformations 
\begin{eqnarray}\label{rdef}
r^i_j=\begin{cases} \ell_{1,j-i+1}\circ\cdots\circ \ell_{i-1,j-1}\circ \ell_{ij} ,& i\le j\\
\ell_{i-j+1,1}\circ\cdots\circ \ell_{i-1,j-1}\circ \ell_{ij}, & i\ge j,\end{cases}
\end{eqnarray}
we see that, given a nucleation matrix  $W=(w_{ij})\in(\R_{>0})^{n\times n}$, we can write the PNG evolution in terms of transformations $r^i_j$ as 
\begin{equation}\label{png_matrix}
\mathcal{H}
=     r^n_n  \, \circ \,( r^{n}_{n-1} \circ r^{n-1}_n)\circ \cdots \circ\,(r^n_1\circ \cdots \circ r^1_n) \,\circ \cdots \, \circ \,(r^1_2\circ r^2_1) \, \circ \, r^1_1\, (W).
\end{equation}

The output matrix $\mathcal{H}$ will encode the heights of the line ensemble. Let us, also, remark that the nucleation matrix  $W=(w_{ij})$ need not actually be a matrix, but it can rather have any
shape. An example with which we will be occupied is when $W=(w_{ij}, \,1\le i+j-1\le n)$, corresponding to a triangular array. 
\subsection{Geometric png and local moves.}

Rewriting the transformations $\ell_{ij}$ in definition of \eqref{tabcd} by replacing $(+,\vee)$ with $(\times, +)$, i.e. performing a geometric lifting of the formulae therein, we clearly arrive to the geometric local transformations  $l_{ij}$ defined in \eqref{abcd}. Similarly, writing the transformations $r^i_j$ in \eqref{rdef} in geometric form we arrive to the transformations $\rho^i_j$ of \eqref{rodef}. The geometric version of 
matrix equation \eqref{png_matrix}, i.e.  
\begin{align*}
\mathcal{H}
= \text{gPNG}(W)
\,:=     \rho^n_n  \, \circ \,( \rho^{n}_{n-1} \circ \rho^{n-1}_n)\circ \cdots \circ\,(\rho^n_1\circ \cdots \circ \rho^1_n) \,\circ \cdots \, \circ \,(\rho^1_2\circ \rho^2_1) \, \circ \, \rho^1_1\, (W)
\end{align*}
can be taken to be the definition of the Geometric PNG corresponding to matrices. 
Notice that even though $\gpng$ is a composition of the same local moves $\rho^i_j$ and $l_{ij}$ employed by $\grsk$, the sequence of the compositions is different. However, it turns out that $\gpng(W)=\grsk(W)$, whenever $W$ is a matrix array.
\vskip 2mm
A reason why we are interested in the $\gpng$ formulation is given by the following proposition, which says that when $\gpng$ is applied to a triangular array
$W=(w_{ij}\colon i,j\geq 1,\,\,1\le i+j-1\le n)$, the elements $h_{ij}$, with $i+j=n+1$, of the output array are the point-to-point polymer partition function from $(1,1)$ to $(i,j)$, $i+j=n+1$. 
In the case of a triangular array $W=(w_{ij}\colon i,j\geq 1,\,\,1\le i+j-1\le n)$ the output array via gPNG is a triangular array  $\mathcal{H}=(h_{ij}\colon \,i,j\geq 1,\,1\le i+j-1\le n)$ defined by the 
sequence of local moves as
\begin{align}\label{def:tri}
\mathcal{H}= (\rho^1_n\circ\cdots\circ \rho^n_1)\circ \cdots \circ (\rho^1_2\circ \rho^2_1)\circ \rho^1_1 (W)
\end{align}
The properties of gPNG on triangular arrays as defined in \eqref{def:tri} will be the object of the following propositions and will eventually lead to an alternative formula for the Laplace transform of the
joint distribution of the point-to-point partition functions at a fixed time (cf. Theorem \ref{thm:law_via_PNG}). 
\begin{proposition}\label{p2p-gPNG}
Let  $W=W_n:=(w_{ij}\colon i,j\geq 1,\,\,1\le i+j-1\le n)$ be a triangular array and $\mathcal{H}=(h_{ij}\colon \,i,j\geq 1,\,1\le i+j-1\le n)= \text{gPNG}(W)
=(\rho^1_n\circ\cdots\circ \rho^n_1)\circ \cdots \circ (\rho^1_2\circ \rho^2_1)\circ \rho^1_1 (W)$ be the geometric PNG array. Then 
\[
h_{ij}= Z_{ij}=\sum_{\pi\in \Pi_{ij}} \prod_{(i,j)\in \pi} w_{ij},\qquad \text{for all}\,\,\,\, (i,j) \,\,\, \text{with} \,\,\, i+j=n+1,
\]
where $\Pi_{ij}$ is the set of directed paths (in matrix notation directed means down-right) starting at $(1,1)$ and ending at $(i,j)$.
\end{proposition}
\begin{proof}
We can recursively write gPNG in the form
\begin{align}\label{Hrec}
\mathcal{H}_n&:=\text{gPNG}(W_n)= \rho^n_1\circ\cdots \circ \rho^1_n 
\begin{pmatrix}
&&&& w_{1n}\\
&\mathcal{H}_{n-1}&&w_{2,n-1}\\
&&\iddots\\
&w_{n-1,2}\\
 w_{n1}
\end{pmatrix}\notag\\
&=\rho^n_1\circ\cdots \circ \rho^1_n  \begin{pmatrix}
h'_{11}&h'_{12}&\cdots&h'_{1,n-1}& w_{1n}\\
h'_{21}&h'_{22}&\cdots&w_{2,n-1}\\
\vdots&&\iddots\\
h'_{n-1,1}&w_{n-1,2}\\
 w_{n1}
\end{pmatrix}
\end{align}
Notice that the transformations $\rho^i_j$ only change the entries on the diagonals $(i,j), (i-1,j-1),\dots$. Moreover, if $h'_{k,n-k}=Z_{k,n-k}$, for $k=1,2...$, then transformation
$\rho^k_{n-k+1}$ will produce the entry
$
h_{k,n-k+1}=w_{k,n-k+1}(h'_{k-1,n-k+1}+h'_{k,n-k}),
$
which by the basic recursion of polymer partition functions equals $Z_{k,n-k+1}$.
\end{proof}
\begin{proposition}\label{thm:PNGvolume}
The gPNG, in logarithm variables, is volume preserving, that is, if $\mathcal{H}=(h_{ij}, \,1\le i + j -1\le n)=\text{gPNG}((w_{ij})_{1\leq i+j-1\leq n})$, then the map
$$ (\log w_{ij},\,1 \leq i+j-1\leq n) \mapsto  (\log h_{ij},\,1\le i+j-1\le n),
$$
has Jacobian equal to $\pm1$.
\end{proposition}
\begin{proof}
It follows immediately from the representation of gPNG in terms of transformations $\rho^i_j$ and their volume preserving properties. 
\end{proof}
In the setting of gRSK the {\it energy} of a matrix $X\in (\R_{>0})^{n\times m}$, was defined as
\begin{equation}\label{e}
\mathcal{E}(X)=\frac{1}{x_{11}}+\sum_{i,j} \frac{x_{i-1,j}+x_{i,j-1}}{x_{ij}},
\end{equation}
where 
the summation is over $1\le i\le n,\ 1\le j\le m$ with the conventions $x_{ij}=0$ for $i=0$ or $j=0$. We can extend this definition for general arrays, with the convention that
 $x_{ij}=0$ if $(i,j)$ does not belong in the index set $ind(X)$ of this array and in particular for triangular arrays $X=(x_{ij}, \,1\le i + j -1 \le n)$. Similarly to part {\bf 2.} of Theorem \ref{thm:OSZ} we have
 \begin{proposition}\label{prop:PNG_energy}
 Let $W=(w_{ij} \colon{1\le i + j -1 \le n})$ be a triangular array and $\mathcal{H}=\text{gPNG}(W)$. Then 
 \[
 \mathcal{E}(\mathcal{H}) = \sum_{i,j}\frac{1}{w_{ij}}
 \]
 \end{proposition} 
\begin{proof}
The proof is similar to the one of Theorem 3.2 of \cite{OSZ14} and it is based on the invariance of the energy under transformations $\rho^i_j$. 
For a triangular array $X=(x_{ij}\,,\,  i+ j \le n +1)$ we define
\[
\mathcal{E}^k_n(X) =\frac{1}{x_{11}}+\sum_{\substack{ i+j=n+1,\, i< k \\ \text{or}\,\,i+j\le n }} \frac{x_{i-1,j}+x_{i,j-1}}{x_{ij}} + \sum_{i+j=n+1,\,i\ge k} \frac{1}{w_{ij}}.
\]
Clearly, $\mathcal{E}^{n+1}_n(X)=\mathcal{E}(X)$ and 
$\mathcal{E}^1_n(X)=\mathcal{E}((x_{ij})_{1\le i + j\le n}) + \sum_{i+j=n+1} 1/w_{ij}$.
Let us assume that the claim is valid for $n-1$, i.e. $\mathcal{E}((x_{ij})_{1\le i +j\le n}) = \sum_{1\le i + j\le n} 1/w_{ij}$. Using the recursive structure of
gPNG, cf \eqref{Hrec}, it is enough to show that 
\begin{align}\label{energy-induction}
 \mathcal{E}^{k+1}_n(\rho^{k}_{n-k+1}(X)) =\mathcal{E}^{k}_n(X) , \quad \text{for} \,\, k=1,...,n. 
\end{align}
Let us denote by $X'=(x'_{ij})_{1\le i \le j \le n}= \rho^k_{n-k+1}(X)$. When $k=1$ we have that $x'_{ij}=x_{ij}$ for all $(i,j)\neq(1,n)$ and $x'_{1n}=w_{1n}x_{1,n-1}$. Therefore,
\begin{align*}
\mathcal{E}^1_n(X)&= \frac{1}{x_{11}}+\sum_{ i+j\le n } \frac{x_{i-1,j}+x_{i,j-1}}{x_{ij}} + \sum_{i+j=n+1,\,i\ge 1} \frac{1}{w_{ij}}\\
&=\frac{1}{x_{11}}+\sum_{ i+j\le n } \frac{x_{i-1,j}+x_{i,j-1}}{x_{ij}} +\frac{1}{w_{1n}}+ \sum_{i+j=n+1,\,i\ge 2} \frac{1}{w_{ij}}\\
&=\frac{1}{x_{11}}+\sum_{ i+j\le n } \frac{x_{i-1,j}+x_{i,j-1}}{x_{ij}} + \frac{x_{1,n-1}}{x'_{1n}} +\sum_{i+j=n+1,\,i\ge 1} \frac{1}{w_{ij}}\\
&=\mathcal{E}^2_n(\rho^1_n(X)).
\end{align*}
Let us now check \eqref{energy-induction} for $2\le k$ and $k\neq \lfloor n/2 \rfloor+1$, if $n$ is odd. It suffices to show that for 
\begin{align*}
&\sum{}^{'} \Big( \,\frac{x'_{i-1,j}+x'_{i,j-1}}{x'_{ij}} +
\frac{x'_{i-1,j-1}}{x'_{i,j-1}} + \frac{x'_{i-1,j-1}}{x'_{i-1,j}}\Big)\\
&=\frac{1}{w_{k,n-k+1}} + 
\sum{}^{''} \Big( \,\frac{x_{i-1,j}+x_{i,j-1}}{x_{ij}} +
\frac{x_{i-1,j-1}}{x_{i,j-1}} + \frac{x_{i-1,j-1}}{x_{i-1,j}}\Big)
\end{align*}
where summation $\sum'$ is over $ (i,j)=(k-r,n-k+1-r)$, with $0\le r \le k\wedge (n-k+1)+1$ and $\sum{}^{''}$ is similar except that $1\le r \le k\wedge (n-k+1)+1$.
This is confirmed by the relations, valid for $(i,j)=(k-r,n-k+1-r)$, $r=1,...,k\wedge (n-k+1)+1$
\begin{align*}
\frac{x'_{i-1,j}+x'_{i,j-1}}{x'_{i,j}} &=\frac{x_{i-1,j}+x_{i,j-1}}{x'_{i,j}} = x_{ij}\Big(\frac{1}{x_{i+1,j}} + \frac{1}{x_{i,j+1}}\Big),\\
x'_{ij}\Big(\frac{1}{x'_{i+1,j}} + \frac{1}{x'_{i,j+1}}\Big) &= x'_{ij}\Big(\frac{1}{x_{i+1,j}} + \frac{1}{x_{i,j+1}}\Big)=  \frac{x_{i-1,j}+x_{i,j-1}}{x_{i,j}}. \\
\frac{x'_{k-1,n-k+1}+x'_{k,n-k}}{x'_{k,n-k+1}} &=\frac{1}{w_{k,n-k+1}},
\end{align*}
which follow from \eqref{bk1} and the fact that $x'_{i\pm1,j}=x_{i\pm1,j}$ and  $x'_{i,j\pm 1}=x_{i,j\pm 1}$. In the case that $n$ is odd and $k= \lfloor n/2 \rfloor+1$, we use the same set of relations and in addition
that 
$$x'_{11}=\frac{1}{x_{11}}\Big(\frac{1}{x_{12}}+\frac{1}{x_{21}}\Big)^{-1}.
$$
\end{proof}
The following proposition is the analogue of the fact that geometric RSK preserves the (geometric) type of an input matrix cf. Proposition\ref{grsk_type}.
\begin{proposition}\label{gPNG_type}
Let $W=(w_{ij})_{1\le i + j -1\le n}$ be a nucleation array and $\mathcal{H}=(h_{ij})_{1\le i + j -1\le n}=\text{gPNG}(W)$.  Then for every
$p,q$ such that $p+q=n+1$ or $p+q=n$ it holds that
\begin{align}\label{prodPNG}
\prod_{r =0}^{p\wedge q -1} h_{p-r,q-r}= \prod_{i\leq p, \,j\leq q} w_{ij},
\end{align}
where we notice that the rightmost product is over all the input elements $w_{ij}$ inside the rectangle with downright corner $(p,q)$.
\end{proposition}
\begin{proof}
The proof goes by induction, so let us assume it is valid for $n-1$. Let 
\[
\mathcal{H}' :=(h'_{ij})_{1\le i + j-1\le n-1}=\text{gPNG}((w_{ij})_{1\le i + j-1 \le n-1})
\]
then
\begin{align*}
\mathcal{H}&= \rho^n_1\circ \rho^{n-1}_2\circ \cdots \circ \rho^2_{n-1}\circ \rho^1_n\, 
\begin{pmatrix}
&&&& w_{1n}\\
&\mathcal{H}'&&w_{2,n-1}\\
&&\iddots\\
&w_{n-1,2}\\
 w_{n1}
\end{pmatrix}\notag\\
&=\rho^n_1\circ \rho^{n-1}_2\cdots \circ \rho^2_{n-1} \circ \rho^1_n  \begin{pmatrix}
h'_{11}&h'_{12}&\cdots&h'_{1,n-1}& w_{1n}\\
h'_{21}&h'_{22}&\cdots&w_{2,n-1}\\
\vdots&&\iddots\\
h'_{n-1,1}&w_{n-1,2}\\
 w_{n1}
\end{pmatrix}
\end{align*}
Notice that the 
transformations $\rho_k^{n-k+1}, k=1,...,n,$ only change the entries on the diagonals $\{ (i,j) \colon j-i= n-2k+1 \}$ with $k=1,...,n,$ while leaving invariant the in between diagonals $\{ (i,j) \colon j-i= n-2k \}$ with $k=1,...,n-1$. For this reason we have that $h_{p-r,q-r}=h'_{p-r,q-r}$ for 
all $p,q$ such that $p+q=n$ and $r\geq 0$. Therefore, \eqref{prodPNG} is immediately satisfied since  by the 
inductive hypothesis 
\begin{align*}
\prod_{r=0}^{p\wedge  q-1}h'_{p-r,q-r} = \prod_{i\leq p, j\leq q} w_{ij},\qquad \text{for} \quad p+q=n.
\end{align*}
Let us now check \eqref{prodPNG} when $p,q$ are such that $p+q=n+1$. We consider the product
\begin{align*}
\prod_{r =0}^{p\wedge q -1} h_{p-r,q-r}= h_{p,q} \prod_{r =1}^{p\wedge q -1} h_{p-r,q-r}
\end{align*}
Using the definition of $\rho^i_j$, cf. \eqref{rodef} and \eqref{bk1}, we can write the above as
\begin{align*}
  w_{p,q}( h'_{p-1,q}+h'_{p,q-1}) \prod_{r=1}^{p\wedge q -1} \frac{1}{h'_{p-r,q-r}} \, \frac{h'_{p-r+1,q-r}\,h'_{p-r,q-r+1}}{h'_{p-r+1,q-r}+h'_{p-r,q-r+1}} \, (h'_{p-r-1,q-r}+h'_{p-r,q-r-1}).
\end{align*}
The terms $h'_{i+1,j}+h'_{i,j+1}$ and $h'_{i-1,j}+h'_{i,j-1}$ will telescope and so the above equals
\begin{align*}
w_{p,q} \, \frac{\prod_{r=1}^{p\wedge q-1} h'_{p-r+1,q-r} \,\prod_{r=1}^{p\wedge q-1} h'_{p-r,q-r+1} }{ \prod_{r=1}^{p\wedge q-1} h'_{p-r,q-r}}.
\end{align*}
Using the inductive hypothesis this equals
\begin{align*}
w_{p.q} \frac{\prod_{i\leq p, j\leq q-1} w_{i,j}\, \prod_{i\leq p-1, j\leq q} w_{i,j}}{ \prod_{i\leq p-1,j\leq q-1} w_{ij}}
= \prod_{i\leq p, j\leq q} w_{i,j}
\end{align*}
and this proves the claim.
\end{proof}
Combining the above propositions we arrive easily at the following theorem, which gives an alternative formula for the joint law 
of all the point-to-point partition functions at a fixed time $n$ of a directed polymer with log-gamma disorder:
\begin{theorem}\label{thm:law_via_PNG}
We consider the log-gamma measure on arrays $W=(w_{ij},\,1\le i +j -1\le n)$ defined by
\begin{align}\label{Gamma_meas}
\bbP(\dd W):= \prod_{1\le i + j -1 \le n} \Gamma(\ga_i+\hat\ga_j)^{-1}\,w_{ij}^{-\ga_i+\hat\ga_j} e^{-1/w_{ij}} \,\,\frac{\dd w_{ij}}{w_{ij}}, \qquad \gamma>0,
\end{align}
The joint distribution of the log-gamma polymer partition functions at time $n$
\[
\bZ_{n}:=\Big( Z_{p,q},\, p+q=n+1\Big),
\]
is given by
\begin{align}\label{Phi-function}
&\bbP(\bZ_n\in \dd \bx)=\frac{1}{\prod_{i,j} \Gamma(\ga_i+\hat\ga_j)}
\prod_{i=1}^n\frac{\dd x_{i}}{x_{i}}   \notag\\
&  \times \int_{\bbT(\bx)} \prod_{k=1}^{n} \left(\frac{\prod_{ j-i=n-2k+1} \,\,t_{ij}}{\prod_{ j-i=n-2k}\,\, t_{ij}} \right)^{-\hat\ga_{n-k+1}}
\prod_{k=1}^{n} \left(\frac{\prod_{ j-i=n-2k+1} \,\,t_{ij}}{\prod_{ j-i=n-2k+2}\,\, t_{ij}} \right)^{-\ga_{k}} 
 e^{-\mathcal{E}(\bbT(\bx))} \prod_{ i+j\le n} \frac{\dd t_{ij}}{t_{ij}} \notag\\
 &=: \frac{1}{\prod_{i,j} \Gamma(\ga_i+\hat\ga_j)} \, \Phi^{(n)}_{\ga,\hat\ga}(x_1,...,x_n) \prod_{i=1}^n \frac{\dd x_{i}}{x_{i}}
\end{align}
where $\bbT(\bx):=\big((t_{ij}),\,1\le i\le j \le n\,\,\text{with}\,\, t_{ij}=x_i \,\, \text{when}\,\,i+j=n+1 \big)$ and we also follow the convention 
in the above formula that if $(i,j)$ do not satisfy $1\leq i+j-1\leq n$, then $t_{ij}=1$.
\end{theorem}
\begin{proof}
Proposition \ref{prop:PNG_energy} shows that $\sum_{i,j} 1/w_{ij} = \mathcal{E}(\bbT)$, where $\bbT=\bbT(\bx)$ when $t_{ij}$ are set equal
to $x_i$ for $i+j=n+1$. Proposition \ref{thm:PNGvolume} shows that $\prod_{i,j} \dd w_{ij}/w_{ij} = \prod_{ij} \dd t_{ij}/t_{ij}$. Finally, from
Proposition \ref{gPNG_type} shows that
\begin{align*}
\prod_{\ell=1}^k w_{\ell,n-k+1} = \frac{\prod_{ j-i=n-2k+1} \,\,t_{ij}}{\prod_{ j-i=n-2k}\,\, t_{ij}}
\quad\text{and} \quad
\prod_{\ell=1}^{n-k+1} w_{k,\ell} = \frac{\prod_{ j-i=n-2k+1} \,\,t_{ij}}{\prod_{ j-i=n-2k+2}\,\, t_{ij}},
\end{align*}
which raising to the powers $-\hat\ga_{n-k+1}$ and $-\ga_{k}$, respectively, gives the corresponding factors in \eqref{Phi-function}.
Putting the pieces together provides the push forward of the log-gamma measure on the gPNG array, which when integrated, fixing the
entries $t_{ij}=x_i$ for $i+j=n+1$ leads to distribution \eqref{Phi-function}. The proof is now completed by recalling Proposition \ref{p2p-gPNG}
which identifies the entries $\{t_{ij}\colon i+j=n+1\}$ of the gPNG array with the point-to-point partition functions of the directed polymer.
\end{proof}
\begin{remark}\rm\rm
It is worth noting that the function  $\Phi^{(n)}_{\ga,\hat\ga}(x_1,...,x_n)$ appears to have an integral structure of the same flavour as that of Whittaker functions, however the structure of the potential $\mathcal{E}(\bbT(\bx))$ is different. This is indicated in Figure \ref{int-structure} where the convolutions in $\mathcal{E}(\bbT(\bx))$ 
are depicted (compare also with the analogous diagram for Whittaker functions in Figure \ref{GT_Whittaker}.)
\end{remark}
\begin{figure}[t]
 \begin{center}
\begin{tikzpicture}[scale=1.0]
\draw[thick][->] (0.5,3.5)--(0.9,3.1);
\draw[thick][->] (3.1,3)--(3.8,3);
\draw[thick][->] (2.1,3)--(2.9,3);
\draw[thick][->] (1.1,3)--(1.9,3);
\draw[thick][->] (2.1,2)--(2.9,2);
\draw[thick][->] (1.1,2)--(1.9,2);
\draw[thick][->] (1.1,1)--(1.9,1);
%%%%Vertical lines%%%%
\draw[thick][->] (3,2.9)--(3,2.1);
\draw[thick][->] (2,2.9)--(2,2.1);
\draw[thick][->] (1,2.9)--(1,2.1);
\draw[thick][->] (1,1.9)--(1,1.1);
\draw[thick][->] (2,1.9)--(2,1.1);
\draw[thick][->] (1,0.9)--(1,0.1);

\draw [fill] (1,1) circle [radius=0.05];
\draw [fill] (2,3) circle [radius=0.05];
\draw [fill] (3,3) circle [radius=0.05];
\draw [fill] (4,3) circle [radius=0.1];
\draw [fill] (1,2) circle [radius=0.05];
\draw [fill] (2,2) circle [radius=0.05];
\draw [fill] (3,2) circle [radius=0.1];
\draw [fill] (1,3) circle [radius=0.05];
\draw [fill] (2,1) circle [radius=0.1];
\draw [fill] (1,0) circle [radius=0.1];

\node at (3.3,3.2) {\small{{\bf $t_{13}$}}};
\node at (2.3,3.2) {\small{{\bf $t_{12}$}}};
\node at (1.3,3.2) {\small{{\bf $t_{11}$}}};
\node at (1.3,2.2) {\small{{\bf $t_{21}$}}};
\node at (2.3,2.2) {\small{{\bf $t_{22}$}}};
\node at (1.3,1.2) {\small{{\bf $t_{31}$}}};
\node at (4.3,3.2) {\small{{\bf $x_1$}}};
\node at (3.3,2.2) {\small{{\bf $x_2$}}};
\node at (2.3,1.2) {\small{{\bf $x_3$}}};
\node at (1.3,0.2) {\small{{\bf $x_4$}}};
\end{tikzpicture}

\end{center}  
\caption{ \small This figure shows the structure of the potential $\mathcal{E}(\bbT(\bx))$ appearing in the function $ \Phi^{(n)}_\ga(x_1,...,x_n)$ in \eqref{Phi-function}.
 The arrows in the picture show that the bonded variables are convoluted in the integrals, i.e. there are exponential weights $\exp(-t_{ij}/t_{i+1,j})$ or 
$\exp(-t_{ij}/t_{i,j+1})$.  The pointer of each arrow indicates the variable that lies in the denominator in the previous fractions. The arrow pointing towards $t_{11}$ corresponds to the term $\exp(-1/t_{11})$.  }  \label{int-structure}
\end{figure}
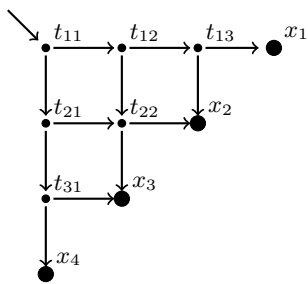

\section{Towards the Airy process}\label{sec:Airy}
In this section we formally check the validity of the following conjecture
\begin{conjecture}\label{thm:claim}
Let 
  \begin{align}\label{pointN2/3}
  	 (m_1,n_1) =(N-t_1N^{2/3},N+t_1N^{2/3}) \quad\text{and}\quad
  	 (m_2,n_2)=(N+t_2N^{2/3},N-t_2N^{2/3}),
  \end{align}
with $t_1,t_2>0$, and 
$Z_{(m_1,n_1)}, Z_{(m_2,n_2)}$ be point-to-point partition functions of a $(0,\gamma)-$log-gamma polymer 
(i.e. an $(\ga,\hat\ga)$-log-gamma polymer with $\ga_i=0$ and $\hat\ga_j=\gamma$ for $i,j\geq 1$).
If $\Ai(\cdot)$ denotes the Airy process (which will be formally defined in Section \ref{auxres}) and the functions
$F,G$ are defined as
\begin{align}\label{FG}
  	G(z) := \log \Gamma(z)-\log \Gamma(\gamma-z)+f_{\gamma}z \quad \text{and} \quad
  	F(z) := \log \Gamma(z)+\log \Gamma(\gamma-z),
\end{align}
 with $f_{\gamma}=-2\Psi(\gamma/2)$ and $\Psi(x) =[\log \Gamma]'(x)$ the digamma function, then
\begin{align*}
\left(\frac{\log Z_{(m_1,n_1)} -Nf_\gamma}{(c_1^{\gamma})^{-1}\,N^{1/3}}, \frac{\log Z_{(m_2,n_2)}-Nf_\gamma}{(c_{1}^{\gamma})^{-1}\,N^{1/3}}\right)
\xrightarrow[N\to\infty]{(d)} \big(\Ai(-c_3^\gamma t_1) -c_2^\gamma \, t_1^2, \, \Ai(c_3^\gamma t_2) -c_2^\gamma \,  t_2^2\big)
\end{align*}
where 
 $c_1^{\gamma}=(-\frac{G'''(\gamma/2)}{2})^{-1/3} $,  $c_2^{\gamma}=-\frac{c_1^{\gamma} (F''(\gamma/2))^2}{2 G'''(\gamma/2)}$ and $c_3^\gamma=-\frac{F''(\gamma/2)}{G'''(\gamma/2)}$ .
 \end{conjecture}
Even if formal, our check of the validity of the conjecture
makes some non trivial steps. 
In particular, we re-express formula \eqref{2p_formula} as finite series related to the Fredholm expansions of the Laplace transform of a single partition function
 and then we reveal a structure of block determinants that relate directly to the Fredholm expansion
  of the two-point distribution for the Airy process. We can then pass to the limit on each individual term in this expansion and establish that it
  converges to the corresponding term in the Fredholm expansion for the Airy process, cf. Proposition \ref{prop:dom_conv}.
   However, what makes our computation formal is that we were not able to justify that this limiting procedure can be done uniformly for all terms in the expansion.
   The difficulty arises from the presence of a {\em cross term} \eqref{CT} involving gamma functions. A naive estimate of this cross term
   with a uniform, constant bound on its modulus, when the values of the arguments are in a compact set, produces a bound of the form $C^{nm}$, which cannot 
   be compensated by the factorial $1/n! m!$ in the expansion to truncate the series.
    Besides such an extension of Proposition \ref{prop:dom_conv} to hold uniformly to all orders in the series and (for the same reason) a formal interchange of integration 
   in relation \eqref{eq:formal_Fub} (which, however, can be made rigorous in the case of the O'Connell-Yor polymer), the rest of our computations are rigorous.
  
\subsection{Auxiliary results}\label{auxres}
The Airy process $\Ai(\cdot)$ is a stationary, continuous process with finite dimensional  distributions given by
\begin{align}\label{def:airy}
\P\big(\Ai(t_1)\leq \xi_1,...,\Ai(t_k)\leq \xi_k\big) =\det(I-\sff\,\sfA\sfi \,\sff)_{L^2(\{t_1,...,t_k\}\times \bbR)},
\end{align}
where for $\ell=1,...,k$ we define $\sff(t_\ell,x):=\ind_{\{x\geq \xi_\ell\}}$ and $\sfA\sfi$ is the extended Airy kernel defined by
\begin{align*}
\sfA\sfi(t,\xi;t', \xi'):=
\begin{cases}
\,\,\int_0^\infty e^{-\gl (t-t')} Ai(\xi+\gl) Ai(\xi'+\gl)  \dd \gl, &\qquad \text{if} \,\,t\geq t',\\
&\\
-\int_{-\infty}^0 e^{-\gl (t-t')} Ai(\xi+\gl) Ai(\xi'+\gl) \dd \gl, &\qquad \text{if} \,\,t< t',
\end{cases}
\end{align*}
In the above integral $Ai(\cdot)$ is the Airy function, which has the integral representation 
\begin{align}\label{def:Airy}
Ai(x)=\frac{1}{2\pi \iota} \int_{C_{\frac{\pi}{3}}} e^{\frac{z^3}{3}-xz} \dd z
\end{align}
The contour of integration consists of two straight rays, one starting at infinity and ending at the origin in an angle $-\pi/3$ and the second departing from the origin at an angle $\pi/3$ and going to infinity. Let us mention that the Airy process can be viewed as the top line of an ensemble of non intersecting paths, which is known as the {\it Airy line ensemble}, whose statistics are also governed by kernel $\sfA\sfi$.

The determinant in \eqref{def:airy} is a Fredholm determinant on $L^2(\{t_1,...,t_k\}\times \bbR)$. Let us recall that for a trace-class operator $K$ acting on 
$L^2(\mathcal{X},\mu)$ as $Kf(x)=\int_\mathcal{X} K(x,y) f(y) \,\dd\mu(y)$, for some measure space $(\mathcal{X},\mu)$, the Fredholm determinant is defined as the
series 
\begin{align*}
\det(I+K)_{L^2(\mathcal{X})}:=1+\sum_{n=1}^\infty \frac{1}{n!} \int_{\mathcal{X}^n} \det\big( K(x_i,x_j)\big)_{n\times n} \,\dd \mu(x_1)\cdots \dd \mu(x_n) .
\end{align*}
In particular, the Fredholm determinant in \eqref{def:airy} can be written as
\begin{align}\label{eq:airy_exp}
&\det(I-\sff\,\sfA\sfi \,\sff)_{L^2(\{t_1,...,t_k\}\times \bbR)} \notag\\
&= 1+ \sum_{n=1}^\infty \frac{(-1)^n}{n!} \sum_{s_1,...,s_n\in\{t_1,...,t_k\}} \int_{-\infty}^\infty\cdots \int_{-\infty}^\infty \det\big(  ( \sff \,\sfA\sfi \,\sff) (s_i,y_i;s_j,y_j)  \big)_{n\times n}
 \dd y_1\cdots\dd y_n
\end{align}

It will be useful to rewrite the above formula in terms of a block determinant. We do this as follows: 
We divide the variables $s_1,...,s_n$ into $k$ groups according to which value out of $\{t_1,...,t_k\}$ they take. The number of ways we can do this 
assignment so that for $j=1,...,k$ we have 
$n_j$ variables taking the value $t_j$ equals 
\[
{n \choose n_1 \cdots n_k}=\frac{n!}{n_1!\cdots n_k!}.
\]
Taking this into account and the fact that $n=n_1+\cdots+n_k$, we can rewrite \eqref{eq:airy_exp} as
\begin{align}\label{blockFred}
&\det(I-\sff\,\sfA\sfi \,\sff)_{L^2(\{t_1,...,t_k\}\times \bbR)} 
= 1+ \sum_{n_1,...,n_k=1}^\infty \frac{(-1)^{n_1+\cdots+n_k}}{n_1!\cdots n_k!} \\
&\qquad  \int_{-\infty}^\infty\cdots \int_{-\infty}^\infty 
\det\Bigg(  \Big\{(\sff\,\sfA\sfi \,\sff)(t_r,y_{i_r};t_{s},y_{j_s})\Big\}_{\substack{i_r=\mathfrak{n}_{r-1}+1,...,\mathfrak{n}_r \\ j_s=\mathfrak{n}_{s-1}+1,...,\mathfrak{n}_s}}  \Bigg)_{r,s=1,...,k}
 \prod_{j=1}^{n_1+\cdots+n_k}\dd y_j  \notag,
\end{align}
where $\mathfrak{n}_r=\sum_{j=1}^{r-1}n_{j}$ and the determinant that appears above is a block determinant.
\vskip 2mm
  Let us record the following Lemma, which is a generalisation of Lemma 4.1.39 in \cite{BC14}.
  The latter has formalised an idea, which has appeared in the physics literature earlier, e.g. \cite{CLeDR10}, 
  and reduces the study of the distributional limit of random variables to the asymptotics of the corresponding Laplace transform. 
  The proof of Lemma \ref{lem:borodin} is a straightforward generalisation of that in \cite{BC14} and so we omit the proof.
  \begin{lemma}\label{lem:borodin}
  	Consider a sequence of functions $\{f_n\}_{n\geq 1}$  mapping $\mathbb{R}\to [0,1]$ such that
  	for each $n$, $f_n(x)$ is strictly decreasing in $x$, tends to one as $x$ tends to $-\infty$ and zero as $x$ tends to $\infty$
  	and for each $\delta>0$ converges uniformly to $1_{\{x\leq 0\}}$ in $\mathbb{R}\backslash[-\delta,\delta]$. 
	Define also $f_n^r(x)=f_n(x-r)$. Consider two sequences of random variable $(X_n)$ and $(Y_n)$ such that for each $r_1,r_2\in \R$, 
  	\begin{equation*}
  		\E[f_n^{r_1}(X_n)f_n^{r_2}(Y_n)]\to p(r_1,r_2),
  	\end{equation*}
  	and assume that $p(r_1,r_2)$ is a continuous joint distribution function of two random variables. Then $(X_n,Y_n)$ converges weakly to a random vector $(X,Y)$ which is distributed according to $ \P(X<r_1,Y<r_2)=p(r_1,r_2)$. 
  \end{lemma}   
  We will apply this lemma by considering the functions $f_N(x)= e^{-e^{N^{1/3}x}}$, which satisfies the criteria,
  and choosing the parameters of the joint Laplace transform as
  \begin{align}\label{u1u2}
  u_1= e^{-Nf_{\gamma}-r_1N^{1/3}} \quad \text{and} \quad
  u_2= e^{-Nf_{\gamma}-r_2N^{1/3}}
  \end{align}
  where $f_{\gamma}=-2\Psi(\gamma/2)$ and $\Psi(z)=[\log\Gamma]'(z)$ is the digamma function.
  For $i=1,2$, we have
  \begin{equation*}
  	e^{-u_i Z_{(m_i,n_i)}}=f_N^{r_i}\Bigg(\frac{\log Z_{(m_i,n_i)}-Nf_{\gamma}}{N^{1/3}}\Bigg)
  \end{equation*}
  By Lemma \ref{lem:borodin}, for each pair $(r_1,r_2)\in \R^2$ it suffices to compute the limit
  \begin{equation*}
  	\lim_{N\to \infty}\bbE\Bigg[f_N^{r_1}\Bigg(\frac{\log Z_{(m_1,n_1)}-Nf_{\gamma}}{N^{1/3}}\Bigg)f_N^{r_2}\Bigg(\frac{\log Z_{(m_2,n_2)}-Nf_{\gamma}}{N^{1/3}}\Bigg)\Bigg]=p_{\gamma}(r_1,r_2)
  \end{equation*}
  and show that this coincides with the joint distribution of the Airy process at the corresponding points. In other words we will need to compute the limit of the Laplace transform $\bbE\Big[e^{-u_1 Z_{(m_1,n_1)}-u_2Z_{(m_2,n_2)}}\Big]$ with the parameters set as above.
\vskip 2mm
We will make use of the following theorem due to Borodin-Corwin-Remenik \cite{BCR13}. We will also use the notation:
$\pm\ell_\delta$ is a vertical line in the complex plane with real part equal to $\pm\delta$, $\ell_{\delta_1,\delta_2,M}$ is a horizontal line segment from $-\delta_1+\iota M$
to $\delta_2+\iota M$ and $C_\delta$ is a circle centred at zero with radius $\delta$. 
\begin{theorem}[\cite{BCR13}]\label{thm:BCR}
Fix $0<\delta_2<1, 0<\delta_1<\min\{\delta_2,1-\delta_2\}$ and $a_1,..,a_N$ such that $|a_i|<\delta_1$  and let $F$ be a meromorphic function, which is non-zero along and inside $C_{\delta_1}$ and all of its poles have real part strictly larger than $\delta_2$. Moreover, assume that for all $\kappa>0$
\begin{align*}
\int_{\pm\ell_{\delta_2}} \dd w \,e^{\pi(\frac{N}{2}-1) | \Im (w)|}  | \Im (w)|^{\kappa} |F(w)| <\infty, \quad \int_{\ell'_{\delta_1,\delta_2,M}} \dd w \,e^{\pi(\frac{N}{2}-1) | \Im (w)|}  | \Im (w)|^{\kappa} |F(w)| \xrightarrow[|M|\to\infty]{} 0.
\end{align*}
Then
\begin{align*}
 \int_{(-\ell_{\delta_1})^N} \dd w \, s_N(w) \prod_{j,j'=1}^N \Gamma(a_{j'}-w_j) \prod_{j=1}^N \frac{F(w_j)}{F(a_j)}
=\det (I+K)_{L^2(C_{\delta_1})}.
\end{align*}
The kernel $K:L^2(C_{\delta_1})\to L^2(C_{\delta_1})$ in the above Fredholm determinant is given by
\begin{align*}
K(v,v')=\frac{1}{2\pi \iota} \int_{\ell_{\delta_2}} \frac{dw}{w-v'}\frac{\pi}{\sin(\pi(v-w))} \frac{F(w)}{F(v)} 
\prod_{j=1}^N \frac{\Gamma(v-a_j)}{\Gamma(w-a_j)}.
\end{align*}
In particular, recalling Theorem \ref{thm:one_point_Laplace}, we have that the Laplace transform of an $(\ga,\hat\ga)-$log-gamma polymer partition function is expressed as
 \begin{align}\label{BCR_Fredholm}
  \bbE\Big[e^{-u Z_{(m,n)} }\Big]=\det(I+K_u)_{L^2(C_{\delta_1})},
 \end{align}
   where the kernel $K_u:L^2(C_{\delta_1})\to L^2(C_{\delta_1})$ equals 
 \begin{align*}
K_u(v,v')=\frac{1}{2\pi \iota} \int_{\ell_{\delta_2}} \frac{dw}{w-v'}\frac{\pi}{\sin(\pi(v-w))} 
\frac{ u^{w}\, \sfF^{\,\ga}_{m}(-w)}{u^{v}\,\sfF^{\,\ga}_{m}(-v)} 
\prod_{j=1}^n \frac{\Gamma(v+\hat\ga_j)}{\Gamma(w+\hat\ga_j)},
\end{align*}
where notation $\sfF_m^\ga(\cdot)$ was introduced in Definition \ref{def:F}.
\end{theorem}
\begin{remark}\label{remark:BCR}\rm
Even though, in general, a Fredholm determinant is given by an infinite series, an inspection of the proof of the above theorem \cite{BCR13} shows that 
  \eqref{BCR_Fredholm} is finite rank and the terms beyond order $N$ in the series are zero. 
\end{remark}
\subsection{Road map}\label{sec:road}
Let us describe the steps we make towards Claim 1. First, in Proposition \eqref{prop:joint} we use Theorem \ref{thm:BCR} in an iterative manner,
 in order to rewrite formula \eqref{2p_formula} as a double, finite series of terms related to expansions of Fredholm determinants like \eqref{BCR_Fredholm}
 (denote by $I_{m,n}^{(N)}$ the generic term in these double series ).
 We note that the proof of this proposition is formal, as it contains a formal interchange of integrals in \eqref{eq:formal_Fub}. On the other hand we also note that
 this formal interchange of integrals can become rigorous, in the case of the O'Connell-Yor polymer, cf. Remark \ref{remark:OY_2}.
  However, for the sake of exposition and uniformity of the text, we have
 preferred to present the calculations in the case of the log-gamma polymer, even if these contain a formal step.
 
  Then in Lemma \ref{lem:blockCauchy} we make an observation that certain block determinants arise in the above mentioned expansion, which point towards the expansion \eqref{blockFred} of the Airy process.
  
  In Proposition \ref{lem:joint_conv} we further rewrite our formula, so that it is in a form more suitable for asymptotics. The asymptotics are then carried out with the
   use of steepest descent method. For this we need to make suitable contour deformations and also estimate our integrands along the deformed contours, cf Lemma 
   \ref{lem:BCR} and \eqref{eq:Fub_blue}. 
   
   The last step we perform is to show that using the steepest descent method, any fixed term in our double expansion, denoted by $I_{m,n}^{(N)}$, converges to  
   \begin{align}\label{2point_road}
   & \frac{(-1)^{n+m}}{n! \, m! (2\pi\iota)^{m+n}}
   \int_{c_1^{\gamma}r_2+c_2^{\gamma}t_2^2}^\infty\cdots \int_{c_1^{\gamma}r_2+c_2^{\gamma}t_2^2}^\infty  \,\, \prod_{j=1}^{n}\dd y_j 
    \,\,  \int_{c_1^{\gamma}r_1+c_2^{\gamma}t_1^2}^\infty  \cdots \int_{c_1^{\gamma}r_1+c_2^{\gamma}t_1^2}^\infty  \,\, \prod_{j=n+1}^{n+m}\dd y_j  \\
    &\notag \\
 &\hskip 1cm\det \left(
  	\begin{array}{cc}
  		\Big(  \sfA\sfi \big( c_3^\gamma t_2,y_{i} \,; \, c_3^\gamma t_{2},y_{j}\big)  \Big)_{n\times n} &
		 \Big( \sfA\sfi \big( c_3^\gamma t_2,y_{i} \,; \,- c_3^\gamma t_{1},y_{j+n}\big) \Big)_{n\times m}\\
		&\\
  		\Big(\sfA\sfi \big(- c_3^\gamma t_1,y_{i+n} \,; \, c_3^\gamma t_{2},y_{j}\big) \Big)_{m\times n} &
		 \Big( \sfA\sfi \big(- c_3^\gamma t_1,y_{i+n} \,;\, - c_3^\gamma t_{1},y_{j+n}\big) \Big)_{m\times m} 
  	\end{array}
  	\right) \notag
   \end{align}
   and when this is summed over all possible $n,m\geq 0$ it coincides with the probability
 \begin{align*}
 \P\Big( \,\Ai(-c_3^\gamma t_1)\leq  c_1^{\gamma}r_1+c_2^{\gamma}t_1^2 \, \, ,\Ai(c_3^\gamma t_2)\leq c_1^{\gamma}r_2+c_2^{\gamma}t_2^2 \, \Big).
 \end{align*}
 Then the use of Lemma \ref{lem:borodin} would confirm Claim 1, assuming one could obtain a good control on the tails of the series involving the terms $I_{m,n}^{(N)}$.
 \subsection{Main computations}
We now follow the roadmap and start by rewriting the contour integral formula provided by Theorem \ref{thm:2-point} for the two point function. 
\begin{proposition}{\rm [}modulo justification of formula \eqref{eq:formal_Fub}{\rm]} \label{prop:joint}
Let $Z_{(m_1,n_1)}, Z_{(m_2,n_2)}$ be the partition functions from $(1,1)$ to $(m_1,n_1)$ and $(m_2,n_2)$, respectively, of a $(0,\gamma)$-log-gamma polymer. 
Fix $0<\gd<\gamma/2$ and $0<\gd_1<\min\{\gd,1-\gd\}$.
We assume that $m_1\leq n_1, m_2\geq n_2, m_2>m_1$ and $n_2<n_1$. Then the joint Laplace transform is written in the form
\begin{align*}
 \bbE\Big[e^{-u_1 Z_{(m_1,n_1)}-u_2Z_{(m_2,n_2)}}\Big] =
 &\sum_{n=0}^{n_2}\sum_{m=0}^{m_1}  \frac{1}{n! m!(2\pi\iota)^{m+n}} 
  \int \dd {\bf v} \dd {\bf w} 
  \int \dd {\sfv} \dd {\sfw}
      \,\,\det\Big(\frac{1}{w_k-v_\ell}\Big)_{n\times n} \det\Big(\frac{1}{\sfw_k-\sfv_\ell}\Big)_{m\times m} \\
  &\quad   \times\,\,    \prod_{k=1}^n \frac{\pi}{\sin(\pi(v_k-w_k))} 
  \frac{ u_2^{w_k}}{u_2^{v_k}} \left(\frac{\Gamma(\gamma-w_k)}{\Gamma(\gamma-v_k)}\right)^{m_2}
 \left(\frac{\Gamma(v_k)}{\Gamma(w_k)}\right)^{n_2}\\
 &\quad   \times\,\,    \prod_{\ell=1}^m \frac{\pi}{\sin(\pi(\sfv_\ell-\sfw_\ell))} 
  \frac{ u_1^{\sfw_\ell}}{u_1^{\sfv_\ell}} \left(\frac{\Gamma(\gamma-\sfw_\ell)}{\Gamma(\gamma-\sfv_\ell)}\right)^{n_1}
 \left(\frac{\Gamma(\sfv_\ell)}{\Gamma(\sfw_\ell)}\right)^{m_1} \\
    &\quad   \times\,\,   \prod_{k=1}^n \prod_{\ell=1}^m 
    \frac{\Gamma(\gamma-\sfw_\ell-w_k)\Gamma(\gamma-\sfv_\ell-v_k)}{\Gamma(\gamma-\sfw_\ell-v_k\,) \,\Gamma(\gamma-\sfv_\ell-w_k)}
\end{align*}
where the first set of contour integrals above are over $(C_{\gd_1}\times \ell_{\gd})^{n}$ and the second over $(C_{\gd_1}\times \ell_{\gd})^{m}$.
\end{proposition}
\begin{remark}\label{remark:OY_2}\rm	
        The two-point function of the O'Connell-Yor polymer in Remark \ref{rem:OY} can also be written in terms a Fredholm-like expansion:
	\begin{align*}
	\bbE\Big[e^{-u_1Z^{OY}(m_1,t_1)-u_2Z^{OY}(m_2,t_2)}\Big] \notag =
	&\sum_{n=0}^{m_1}\sum_{m=0}^{m_2}  \frac{1}{n! m!(2\pi\iota)^{m+n}} 
	\int \dd {\bf v} \dd {\bf w} 
	\int \dd {\sfv} \dd {\sfw}
	\,\,\det\Big(\frac{1}{w_k-v_\ell}\Big)_{n\times n} \det\Big(\frac{1}{\sfw_k-\sfv_\ell}\Big)_{m\times m} \\
	&\quad   \times\,\,    \prod_{k=1}^n \frac{\pi}{\sin(\pi(v_k-w_k))} 
	\frac{ (u_1/u_2)^{w_k}}{(u_1/u_2)^{v_k}} \left(\frac{\Gamma(w_k)}{\Gamma(v_k)}\right)^{m_2-m_1} \exp\Big(\frac{t_1-t_2}{2}(w_k^2-v_k^2)\Big)
	\\
	&\quad   \times\,\,    \prod_{\ell=1}^m \frac{\pi}{\sin(\pi(\sfv_\ell-\sfw_\ell))} 
	\frac{ u_2^{\sfw_\ell}}{u_2^{\sfv_\ell}} 
	\left(\frac{\Gamma(\sfv_\ell)}{\Gamma(\sfw_\ell)}\right)^{m_2} \exp\Big(\frac{t_2}{2}(\sfw_\ell^2-\sfv_\ell^2)\Big)\\
	&\quad   \times\,\,   \prod_{k=1}^n \prod_{\ell=1}^m 
	\frac{\Gamma(w_k-\sfw_\ell)\Gamma(v_k-\sfv_\ell)}{\Gamma(v_k-\sfw_\ell\,) \,\Gamma(w_k-\sfv_\ell)}
	\end{align*}
	 where the first set of contour integrals above are over $(C_{\gd_1}\times \ell_{\gd_1'})^{n}$ and the second over $(C_{\gd_2}\times \ell_{\gd_2'})^{m}$ with $\gd_2'<\gd_1'<\gd_2<\gd_1$. The proof follows the same steps as in the upcoming for log gamma polymer. In this case, however, the interchange of integrals can be easily justified by Fubini, due to the super-exponential decay of the terms
	  $\exp\Big(\frac{t_1-t_2}{2}(w_k^2-v_k^2)\Big)$ and $\exp\Big(\frac{t_2}{2}(\sfw_\ell^2-\sfv_\ell^2)\Big)$ along vertical axes.
\end{remark}
\begin{proof}[Proof of Proposition \ref{prop:joint} .]
Let us start with a general $(\ga,\hat\ga)-$log-gamma polymer such that $\vert \ga_i\vert <\delta_1$ and $\vert \hat\ga_j-\gamma\vert <\delta_1$  and with formula \eqref{2p_formula}.
We multiply and divide inside the integrals of \eqref{2p_formula} by 
$\prod_{i=1}^{m_1}\prod_{j=1}^{n_2} \Gamma(\lambda_i+\hat\ga_j)
= \prod_{j=1}^{n_2} \prod_{i=1}^{m_1} \Gamma(\lambda_i+\hat\ga_j)
$ and rearrange the terms to write \eqref{2p_formula} as
\begin{align}\label{eq:manip1}
&\int_{(\ell_\delta)^{m_1} } \dd \lambda
    \,\,s_{m_1}(\lambda)  \prod_{1\leq i,i'\leq m_1}  \Gamma(-\ga_{i'}+\gl_{i}) 
    \prod_{i=1}^{m_1}\frac{u_1^{-\gl_i}\, \sfF_{n_2,n_1}^{\hat\ga}(\gl_i)  }{ u_1^{-\ga_i}\, \sfF_{n_2,n_1}^{\hat\ga}(\ga_i)  } 
 \,\,  \prod_{i=1}^{m_1}  \frac{\prod_{j=1}^{n_2} \Gamma(\gl_i+\hat\ga_j)}{\prod_{j=1}^{n_2}  \Gamma(\ga_i+\hat\ga_j)} 
   \notag\\
%\int_{(\iota \bbR)^{m_1} } \dd \lambda
%   \,\,s_{m_1}(\lambda)  \prod_{1\leq i,i'\leq m_1} \Gamma(-\ga_{i'}+\gl_{i}) \prod_{i=1}^{m_1}\frac{u_1^{-\gl_i}\, \sfF^{\,\hat\ga}_{n_2,n_1}(\gl_i)}{ u_1^{-\ga_i}\, \sfF^{\,\hat\ga}_{n_2,n_1} (\ga_i)} 
% \,\,  \prod_{i=1}^{m_1}  \frac{\prod_{j=1}^{n_2} \Gamma(\gl_i+\hat\ga_j)}{\prod_{j=1}^{n_2}  \Gamma(\ga_i+\hat\ga_j)} 
% \notag\\
&\quad\times   \int_{(\ell_{\delta+\gamma})^{n_2} } \dd \mu  
    \,\,s_{n_2}(\mu)  \prod_{1\leq j,j'\leq n_2} \Gamma(-\hat\ga_{j'}+\mu_{j}) \prod_{j=1}^{n_2}\frac{u_2^{-\mu_j} \,
   \sfF_{m_1,m_2}^{\ga}(\mu_j)  }{u_2^{-{\hat\ga_j}} \,\sfF_{m_1,m_2}^{\ga}(\hat\ga_j) } 
   \,\, \prod_{j=1}^{n_2}  \frac{\prod_{i=1}^{m_1} \Gamma(\gl_i+\mu_j)}{\prod_{i=1}^{m_1}  \Gamma(\gl_i+\hat\ga_j)}\notag\\
&=
\int_{(\ell_{\delta})^{m_1} } \dd \lambda
    \,\,s_{m_1}(\lambda)  \prod_{1\leq i,i'\leq m_1} \Gamma(-\ga_{i'}+\gl_{i}) \prod_{i=1}^{m_1}
    \frac{u_1^{-\gl_i}\, \sfF^{\,\hat\ga}_{n_1}(\gl_i)}{ u_1^{-\ga_i}\, \sfF^{\,\hat\ga}_{n_1} (\ga_i)} \notag\\
&\quad\times   \int_{(\ell_{\delta+\gamma})^{n_2} } \dd \mu  
  \,\,s_{n_2}(\mu)  \prod_{1\leq j,j'\leq n_2} \Gamma(-\hat\ga_{j'}+\mu_{j}) \prod_{j=1}^{n_2}
    \frac{u_2^{-\mu_j} \, \sfF_{m_1}^{\gl}(\mu_j) \sfF_{m_1,m_2}^{\ga}(\mu_j) }
    {u_2^{-{\hat\ga_j}} \, \sfF_{m_1}^{\gl}(\hat\ga_j) \sfF_{m_1,m_2}^{\ga}(\hat\ga_j)} .
\end{align}

We will now rewrite the integral over $\mu$ using Theorem \ref{thm:BCR} in terms of a Fredholm determinant. Observe that this integral 
represents the Laplace transform of the partition function of a $(\gl_1,...,\gl_{m_1},\ga_{m_1+1},...,\ga_{m_2}; \hat\ga_1,...,\hat\ga_{n_2})$-log-gamma polymer on $[1,m_2]\times[1,n_2]$. Therefore, the integral will be the same if the parameters are changed to
$(\gl_1+\gamma,...,\gl_{m_1}+\gamma,\ga_{m_1+1}+\gamma,...,\ga_{m_2}+\gamma; \hat\ga_1-\gamma,...,\hat\ga_{n_2}-\gamma)$. This will be technically convenient later on. We then have that
\begin{align*}
& \int_{(\ell_{\delta+\gamma})^{n_2} } \dd \mu  
  \,\,s_{n_2}(\mu)  \prod_{1\leq j,j'\leq n_2} \Gamma(-\hat\ga_{j'}+\mu_{j}) \prod_{j=1}^{n_2}
    \frac{u_2^{-\mu_j} \, \sfF_{m_1}^{\gl}(\mu_j) \sfF_{m_1,m_2}^{\ga}(\mu_j) }
    {u_2^{-{\hat\ga_j}} \, \sfF_{m_1}^{\gl}(\hat\ga_j) \sfF_{m_1,m_2}^{\ga}(\hat\ga_j)} \\
=&\int_{(\ell_{\delta})^{n_2} } \dd \mu  
  \,\,s_{n_2}(\mu)  \prod_{1\leq j,j'\leq n_2} \Gamma(-\hat\ga_{j'}+\gamma+\mu_{j}) \prod_{j=1}^{n_2}
    \frac{u_2^{-\mu_j} \, \sfF_{m_1}^{\gl+\gamma}(\mu_j) \sfF_{m_1,m_2}^{\ga+\gamma}(\mu_j)}
    {u_2^{-{\hat\ga_j}+\gamma} \,\,\sfF_{m_1}^{\gl}(\hat\ga_j) \sfF_{m_1,m_2}^{\ga}(\hat\ga_j) }\\     
=&\int_{(-\ell_{\delta})^{n_2} } \dd \mu  
 \,\,s_{n_2}(\mu)  \prod_{1\leq j,j'\leq n_2} \Gamma(-\hat\ga_{j'}+\gamma-\mu_{j}) \prod_{j=1}^{n_2}
 \frac{u_2^{\mu_j} \, \sfF_{m_1}^{\gl+\gamma}(-\mu_j) \sfF_{m_1,m_2}^{\ga+\gamma}(-\mu_j)}
 {u_2^{-{\hat\ga_j}+\gamma} \,\,\sfF_{m_1}^{\gl}(\hat\ga_j) \sfF_{m_1,m_2}^{\ga}(\hat\ga_j) }\\        
=&\det(I+K_{u_2}^\gl)_{L^2(C_{\delta_1})}, 
\end{align*}
 where  the kernel $K_{u_2}:L^2(C_{\delta_1})\to L^2(C_{\delta_1})$ equals 
 \begin{align*}
K_{u_2}^\gl(v,v')=\frac{1}{2\pi \iota} \int_{\ell_{\delta}} \frac{\dd w}{w-v'}\frac{\pi}{\sin(\pi(v-w))} 
\frac{ u_2^{w}\, \sfF^{\,\gl+\gamma}_{m_1}(-w) \sfF^{\,\ga+\gamma}_{m_1,m_2}(-w)  }{u_2^{v}\,\sfF^{\,\gl+\gamma}_{m_1}(-v)\sfF^{\,\ga+\gamma}_{m_1,m_2}(-v)} 
\prod_{j=1}^{n_2} \frac{\Gamma(v+\hat\ga_j-\gamma)}{\Gamma(w+\hat\ga_j-\gamma)}.
\end{align*}
Notice that all the poles in the product $\sfF_{m_1}^{\gl+\gamma}(-\mu_j) \sfF_{m_1,m_2}^{\ga+\gamma}(-\mu_j)$ have real part strictly larger than $\gd$. 
This is a consequence of the assumption $\gd_1<\gd<\gamma/2$ and $|\ga_i|<\gd_1$.
Moreover, the decay conditions in Theorem \ref{thm:BCR} is verified by using the asymptotics of gamma function \ref{gamma_asympto}.
 Inserting this Fredholm determinant into \eqref{eq:manip1} we obtain that the latter equals 
\begin{align}\label{lambda-fred}
\int_{(\ell_{\delta})^{m_1} } \dd \lambda
    \,\,s_{m_1}(\lambda)  \prod_{1\leq i,i'\leq m_1} \Gamma(-\ga_{i'}+\gl_{i}) \prod_{i=1}^{m_1}\frac{u_1^{-\gl_i}\, \sfF^{\,\hat\ga}_{n_1}(\gl_i)}{ u_1^{-\ga_i}\, \sfF^{\,\hat\ga}_{n_1} (\ga_i)} \,\, \det(I+K_{u_2}^\gl)_{L^2(C_{\delta_1})},
\end{align}
 The next step is to expand the Fredholm determinant in \eqref{lambda-fred} and then expand the determinants involved in this expansion. Recalling Remark \ref{remark:BCR}, which implies that the series expansion of the Fredholm determinant contain only $n_2$ number of terms, we have
 \begin{align*}
 \det(I+K_{u_2}^\gl)_{L^2(C_{\delta_1})} &=\sum_{n=0}^{n_2}\frac{1}{n!}  
 \int_{(C_{\delta_1})^n} \dd {\bf v} \,\det(K_{u_2}^\gl (v_k,v_\ell))_{n\times n} \\
 &= \sum_{n=0}^{n_2} \sum_{\sigma\in S_n} \frac{(-1)^\sigma}{n!}  \int_{(C_{\delta_1})^n} \dd {\bf v}\,
     \prod_{k=1}^n  K_{u_2}^\lambda(v_k,v_{\sigma(k)})\\
 &= \sum_{n=0}^{n_2} \sum_{\sigma\in S_n} \frac{(-1)^\sigma}{n!(2\pi\iota)^n} 
       \int_{(C_{\delta_1})^n} \dd {\bf v}\, \int_{(\ell_{\delta})^{n}} \dd{\bf w}
    \,\prod_{k=1}^n \frac{1}{w_k-v_{\sigma(k)}}\frac{\pi}{\sin(\pi(v_k-w_k))} \\
 &\qquad\qquad\qquad\qquad\qquad\times    
\frac{ u_2^{w_k}\, \sfF^{\,\gl+\gamma}_{m_1}(-w_k) \sfF^{\,\ga+\gamma}_{m_1,m_2}(-w_k)  }
{u_2^{v_k}\,\sfF^{\,\gl+\gamma}_{m_1}(-v_k) \sfF^{\,\ga+\gamma}_{m_1,m_2}(-v_k)} \frac{\sfF^{\hat\ga-\gamma}_{n_2}(v_k)}{\sfF^{\hat\ga-\gamma}_{n_2}(w_k)}\\
 &= \sum_{n=0}^{n_2}  \frac{1}{n!(2\pi\iota)^n} 
       \int_{(C_{\delta_1})^n} \dd {\bf v}\, \int_{(\ell_{\delta})^n} \dd{\bf w}
    \, \det\Big(\frac{1}{w_k-v_\ell}\Big)_{n\times n} \,\prod_{k=1}^n \frac{\pi}{\sin(\pi(v_k-w_k))} \\
 &\qquad\qquad\qquad\qquad\qquad\times    
\frac{ u_2^{w_k}\, \sfF^{\,\gl+\gamma}_{m_1}(-w_k) \sfF^{\,\ga+\gamma}_{m_1,m_2}(-w_k)  }{u_2^{v_k}\,\sfF^{\,\gl+\gamma}_{m_1}(-v_k)
\sfF^{\,\ga+\gamma}_{m_1,m_2}(-v_k)} \frac{\sfF^{\hat\ga-\gamma}_{n_2}(v_k)}{\sfF^{\hat\ga-\gamma}_{n_2}(w_k)} ,
 \end{align*}
 where in the last step we used the fact that
  \[
 \sum_{\sigma\in S_n} (-1)^\sigma \prod_{k=1}^n \frac{1}{w_k-v_{\sigma(k)}} =\det\Big(\frac{1}{w_k-v_\ell}\Big)_{n\times n}.
 \]
Inserting the Fredholm expansion into \eqref{lambda-fred} we write the latter as
\begin{align*}
&\sum_{n=0}^{n_2}\frac{1}{n!(2\pi \iota)^n}\int_{(\ell_\delta)^{m_1} } \dd \lambda
    \,\,s_{m_1}(\lambda)  \prod_{1\leq i,i'\leq m_1} \Gamma(-\ga_{i'}+\gl_{i}) \prod_{i=1}^{m_1}\frac{u_1^{-\gl_i}\, \sfF^{\,\hat\ga}_{n_1}(\gl_i)}{ u_1^{-\ga_i}\, \sfF^{\,\hat\ga}_{n_1} (\ga_i)} \\
 &\qquad   \int_{(C_{\delta_1})^n} \dd {\bf v} \int_{(\ell_{\delta})^n} \dd{\bf w}
    \, \det\Big(\frac{1}{w_k-v_\ell}\Big)_{n\times n} \,\prod_{k=1}^n \frac{\pi}{\sin(\pi(v_k-w_k))}    
\frac{ u_2^{w_k}\, \sfF^{\,\gl+\gamma}_{m_1}(-w_k) \sfF^{\,\ga+\gamma}_{m_1,m_2}(-w_k)  }{u_2^{v_k}\,\sfF^{\,\gl+\gamma}_{m_1}(-v_k)
\sfF^{\,\ga+\gamma}_{m_1,m_2}(-v_k)} \frac{\sfF^{\hat\ga-\gamma}_{n_2}(v_k)}{\sfF^{\hat\ga-\gamma}_{n_2}(w_k)}
\end{align*}
rearranging {\em formally} the integrals (however this rearrangement can be justified when $2n\leq n_1-m_1$) and 
using that $\prod_{k=1}^n \sfF^{\,\gl+\gamma}_{m_1}(-w_k) =\prod_{k=1}^n \prod_{i=1}^{m_1}  \Gamma(\lambda_i+\gamma-w_k)= \prod_{i=1}^{m_1}  \sfF^{\gamma-w}_{n}(\gl_i) $ and similarly for $\prod_{k=1}^n \sfF^{\,\gl+\gamma}_{m_1}(-v_k)$
we write the above as
\begin{align}\label{eq:formal_Fub}
&\sum_{n=0}^{n_2}  \frac{1}{n!(2\pi\iota)^n}  \int_{(C_{\delta_1})^n} \dd {\bf v} \int_{(\ell_{\delta})^n} \dd {\bf w}
      \det\Big(\frac{1}{w_k-v_\ell}\Big)_{n\times n} \prod_{k=1}^n \frac{\pi}{\sin(\pi(v_k-w_k))} 
     \frac{ u_2^{w_k}\,  \sfF^{\,\ga+\gamma}_{m_1,m_2}(-w_k)  }{u_2^{v_k}\,\sfF^{\,\ga+\gamma}_{m_1,m_2}(-v_k)} 
     \frac{\sfF^{\hat\ga-\gamma}_{n_2}(v_k)}{\sfF^{\hat\ga-\gamma}_{n_2}(w_k)}\notag\\
    &\qquad\qquad\qquad\qquad \times   \int_{(\ell_\delta)^{m_1} } \dd \lambda
    \,\,s_{m_1}(\lambda)  \prod_{1\leq i,i'\leq m_1} \Gamma(-\ga_{i'}+\gl_{i}) \prod_{i=1}^{m_1}
    \frac{u_1^{-\gl_i}\, \sfF^{\,\hat\ga}_{n_1}(\gl_i) \sfF^{\gamma-w}_{n}(\gl_i) }{ u_1^{-\ga_i}\, \sfF^{\,\hat\ga}_{n_1} (\ga_i) \sfF^{\gamma-v}_{n}(\gl_i)} .
\end{align}
 For $\gl\in \mathbb{C}$, set 
\[
\sff_{n,w,v}(\gl) := \frac{\sfF^{\,\hat\ga}_{n_1}(\gl) \sfF^{\gamma-w}_{n}(\gl) }{ \sfF^{\gamma-v}_{n}(\gl)} .
\]
Multiply and divide the last integrand by $\prod_{i=1}^{m_1} \sfF^{\gamma-v}_{n}(\ga_i) / \sfF^{\gamma-w}_{n}(\ga_i) = \prod_{k=1}^n  \sfF^{\ga+\gamma}_{m_1}(-v_k)/\sfF^{\ga+\gamma}_{m_1}(-w_k)  $ and rewrite the above as 
\begin{align}\label{eq:manip2}
&\sum_{n=0}^{n_2}  \frac{1}{n!(2\pi\iota)^n}  \int_{(C_{\delta_1})^n} \dd {\bf v} \int_{(\ell_{\delta})^n} \dd {\bf w}
      \det\Big(\frac{1}{w_k-v_\ell}\Big)_{n\times n} \prod_{k=1}^n \frac{\pi}{\sin(\pi(v_k-w_k))} 
     \frac{ u_2^{w_k}\,  \sfF^{\,\ga+\gamma}_{m_2}(-w_k)  }{u_2^{v_k}\,\sfF^{\,\ga+\gamma}_{m_2}(-v_k)} 
     \frac{\sfF^{\hat\ga-\gamma}_{n_2}(v_k)}{\sfF^{\hat\ga-\gamma}_{n_2}(w_k)}\notag\\
    &\qquad\qquad\qquad\qquad \times   \int_{(\ell_\gd)^{m_1} } \dd \lambda
    \,\,s_{m_1}(\lambda)  \prod_{1\leq i,i'\leq m_1} \Gamma(-\ga_{i'}+\gl_{i}) \prod_{i=1}^{m_1}
    \frac{u_1^{-\gl_i}\, \sff_{n,w,v}(\gl_i)  }{ u_1^{-\ga_i}\, \sff_{n,w,v}(\ga_i) } .
\end{align}

We will now use Theorem \ref{thm:BCR} to compute the integral with respect to the $\lambda$ variables as a Fredholm determinant, i.e. 
\[
 \int_{(\ell_\delta)^{m_1} } \dd \lambda
    \,\,s_{m_1}(\lambda)  \prod_{1\leq i,i'\leq m_1} \Gamma(-\ga_{i'}+\gl_{i}) \prod_{i=1}^{m_1}
    \frac{u_1^{-\gl_i}\, \sff_{n,w,v}(\gl_i)  }{ u_1^{-\ga_i}\, \sff_{n,w,v}(\ga_i) }
    =\det\big(I+K^{n,w,v}_{u_1}\big)_{L^2(C_{\delta_1})},
\]
where the kernel $K^{n,w,v}_{u_1}:L^2(C_{\delta_1} )\to L^2(C_{\delta_1})$ is given by
\begin{align}\label{KG}
K^{n,w,v}_{u_1}(\sfv,\sfv')=\frac{1}{2\pi \iota} \int_{\ell_{\delta}} \frac{\dd \sfw}{\sfw-\sfv'}\frac{\pi}{\sin(\pi(\sfv-\sfw))}
 \frac{u_1^{\sfw} \,\sff_{n,w,v}(-\sfw) }{ u_1^{\sfv} \,\sff_{n,w,v}(-\sfv)} \,\prod_{i=1}^{m_1}\frac{\Gamma(\sfv+\ga_i)}{\Gamma(\sfw+\ga_i)}.
\end{align}
All the conditions of Theorem \ref{thm:BCR} are verified: The poles of the function $\sff_{n,w,v}(-\lambda)$ are $\hat{\ga}_j$ and $\gamma-w$ which have real part strictly larger than $\delta$. The decay conditions are verified by usual estimation of gamma functions.
Formula \eqref{eq:manip2} is then written as
\begin{align*}
&\sum_{n=0}^{n_2} \frac{1}{n!(2\pi\iota)^n}  \int_{(C_{\delta_1})^n} \dd {\bf v} \int_{(\ell_{\delta})^n} \dd {\bf w}\,\,
      \det\Big(\frac{1}{w_k-v_\ell}\Big)_{n\times n} \\
 &\qquad\qquad \times \prod_{k=1}^n \frac{\pi}{\sin(\pi(v_k-w_k))} \frac{ u_2^{w_k}\,  \sfF^{\,\ga+\gamma}_{m_2}(-w_k)  }{u_2^{v_k}\,\sfF^{\,\ga+\gamma}_{m_2}(-v_k)} 
     \frac{\sfF^{\hat\ga-\gamma}_{n_2}(v_k)}{\sfF^{\hat\ga-\gamma}_{n_2}(w_k)}  \,\,
  \det\big(I+K^{n,w,v}_{u_1}\big)_{L^2(C_{\delta_1})}    
 \end{align*}
We now specialise and restrict attention to a $(0,\gamma)$-log-gamma polymer, i.e. when $\ga_i=0$ for $ i=1,2,...$ and $\hat\ga_j=\gamma>0$ for $ j=1,2,...$. Moreover, we expand the Fredholm determinant $ \det\big(I+K^{n,w,v}_{u_1}\big)_{L^2(C_{\delta_1})} $
and perform the same calculation as above to arrive to the result.
\end{proof}
The next lemma is crucial as it reveals the block determinant structure, which relates to the Fredholm expansion of the two-point distribution of the Airy process.
\begin{lemma}\label{lem:blockCauchy}
	For any sequence of complex numbers $(w_i)_{i\geq 1}, (v_i)_{i\geq 1}, (\sfw_i)_{i\geq 1}, (\sfv_i)_{i\geq 1}$, such that for
	some $\gamma>0$ it holds that $\Re(w_i), \Re(\sfw_i)<\gamma/2$ and  $\Re(v_i),\Re(\sfv_i)<\gamma/2$ for $i=1,2,...$ and furthermore $\Re(w_i-v_j)>0$ and $\Re(\sfw_i-\sfv_j)>0$, 
	the following identity is true
	\begin{align}\label{blockCauchy}
	&\det\Big(\frac{1}{w_k-v_\ell}\Big)_{n\times n} \det\Big(\frac{1}{\sfw_k-\sfv_\ell}\Big)_{m\times m}
	\prod_{k=1}^n \prod_{\ell=1}^m 
	\frac{(\gamma-\sfw_\ell-v_k\,) \,(\gamma-\sfv_\ell-w_k)}{(\gamma-\sfw_\ell-w_k)(\gamma-\sfv_\ell-v_k)}\\
	&=\int_{(0,\infty)^{n+m}} \dd x_1\cdots \dd x_{n+m} 
	\det\left(\begin{array}{cc}
	\left(e^{-x_j(w_i-v_j)}\right)_{n\times n} &   \left(-e^{-x_{j+n}(\gamma-w_i-\sfw_j)}\right)_{n\times m}  \\
	\left(e^{-x_j(\gamma-\sfv_i-v_j)}\right)_{m\times n} & \left(e^{-x_{j+n}(\sfw_j-w_i)}\right)_{m\times m} 
	\end{array}
	\right)\notag
	\end{align}
\end{lemma}
\begin{proof}
	Expanding the Cauchy determinants in the left hand side as 
	\begin{align*}
	\det\Big(\frac{1}{w_k-v_\ell}\Big)_{n\times n} = \frac{\prod_{1\leq k<\ell\leq n}(w_k-w_\ell)(v_\ell-v_k)}{\prod_{1\leq k,\ell\leq n}(w_k-v_\ell)}
	\end{align*}
	and recombining with the product term therein we can write
	the left hand side of \eqref{blockCauchy}  as a block Cauchy determinant 
	\begin{align*}
	&\det\left(\begin{array}{cc}
	\left(\frac{1}{w_i-v_j}\right)_{n\times n} & \left(\frac{1}{w_i+\sfw_i-\gamma}\right)_{n\times m}\\
	\left(\frac{1}{\gamma-\sfv_i-v_j}\right)_{m\times n}  & \left(\frac{1}{\sfw_i-\sfv_j}\right)_{m\times m}
	\end{array}
	\right)
	\\
	&\\  
	&\qquad\qquad=
	\det\left(\begin{array}{cc}
	\left(\int_0^\infty e^{-x(w_i-v_j)} \dd x \right)_{n\times n} & \left( -\int_0^{\infty} e^{-x(\gamma-w_i-\sfw_j)} \dd x \right)_{n\times m}\\
	\left(\int_0^\infty e^{-x(\gamma-\sfv_i-v_j)} \dd x \right)_{m\times n}  & \left( \int_0^\infty e^{-x(\sfw_i-\sfv_j)} \dd x \right)_{m\times m}
	\end{array}
	\right).
	\end{align*}
	Notice that the condition $\Re(w_i), \Re(\sfw_i)<\gamma/2$ and  $\Re(v_i),\Re(\sfv_i)<\gamma/2$ for $i=1,2,...$, ensures that the above integrals are well defined.
	The final step is to use the multilinearity of the determinant, in order to pull the integrals out of the block determinant and hence arrive to the right hand side of \eqref{blockCauchy}.
\end{proof}
Using Lemma \ref{lem:blockCauchy} we will rewrite the joint Laplace transform of $Z_{(n_1,m_1)},Z_{(n_2,m_2)}$ in a form with structure closer to that of the Airy process
\eqref{blockFred}.  
\begin{proposition}\label{lem:joint_conv}
\begin{align*}
& \bbE\Big[e^{-u_1 Z_{(m_1,n_1)}-u_2Z_{(m_2,n_2)}}\Big] \\&=
 \sum_{n=0}^{n_2}\sum_{m=0}^{m_1}  \frac{(-1)^{m+n}}{n! m!(2\pi\iota)^{m+n}} 
 \int \dd {\bf v} \dd {\bf w} \int \dd \sfv \dd\sfw
 \int_{\bbR_{>0}^{n+m}} \dd \bx \int_{\bbR_{>0}^{n+m}} \dd \btau
\,\, \det\left(
        \begin{array}{cc}
        A(\bx,\btau; {\bf w},{\bf v}) & B(\bx,\btau; {\bf w}, \sfw)\\
        C(\bx,\btau; {\bf v}, \sfv) & D(\bx,\btau; \sfw, \sfv) 
        \end{array}
          \right)\\
  & \qquad\quad\times   \prod_{k=1}^n\frac{\pi(v_k-w_k)}{\sin(\pi(v_k-w_k))}  \prod_{\ell=1}^m\frac{\pi(\sfv_\ell-\sfw_\ell)}{\sin(\pi(\sfv_\ell-\sfw_\ell))}
   \prod_{k=1}^n\prod_{\ell=1}^m 
        \frac{\Gamma(1+\gamma-\sfw_\ell-w_k)\Gamma(1+\gamma-\sfv_\ell-v_k)}{\Gamma(1+\gamma-\sfw_\ell-v_k\,) \,\Gamma(1+\gamma-\sfv_\ell-w_k)}\\
   &=: \sum_{n=0}^{n_2}\sum_{m=0}^{m_1} I_{m,n}^{(N)}    
 \end{align*}
  where in the above contour integrals the $v$ and $\sfv$ variables run over $C_{\gd_1}$ while the $w$ and $\sfw$ variables 
  run over $\ell_{\delta}$ and the block matrices in the $(n+m)\times(n+m)$ determinant are given by
 \begin{align*}
 A(\bx,\btau; {\bf w},{\bf v})_{ij}    &=e^{(x_j+\tau_i)\left(\frac{\gamma}{2} -w_i\right) -(x_j+\tau_j)\left(\frac{\gamma}{2}-v_j\right)} \\
 			    &\hskip 2.5 cm \times u_2^{w_i-v_j}\, 
                            \left(\frac{\Gamma(\gamma-w_i)}{\Gamma(\gamma-v_j)}\right)^{m_2}
                            \left(\frac{\Gamma(v_j)}{\Gamma(w_i)}\right)^{n_2}, \hskip 1.6cm 1\leq i,j\leq n,\\
  B(\bx,\btau; {\bf w}, \sfw)_{ij} &= -e^{(-x_{j+n}+\tau_i)\left(\frac{\gamma}{2}-w_i\right)+
  	(-x_{j+n}+\tau_{j+n})\left(\frac{\gamma}{2}-\sfw_j\right)}\\
  &\hskip 2.5 cm  \times u_2^{w_i-\frac{\gamma}{2}} u_1^{\sfw_j-\frac{\gamma}{2}} \,
  \frac{\Gamma(\gamma-w_i)^{m_2}}{\Gamma(w_i)^{n_2}}
  \frac{\Gamma(\gamma-\sfw_j)^{n_1}}{\Gamma(\sfw_j)^{m_1}},\\
& \hskip 9.3cm  1\leq i\leq n, 1\leq j \leq m,\\
 C(\bx,\btau; {\bf v}, \sfv)_{ij}   &=
 			     e^{-(x_j+\tau_j)\left(\frac{\gamma}{2}-v_j\right) -(x_j+\tau_{i+n})\left(\frac{\gamma}{2}-\sfv_i\right)}\\
			     &\hskip 2.5 cm  \times u_1^{-(\sfv_i-\frac{\gamma}{2})} u_2^{-(v_j-\frac{\gamma}{2})}\,
                              \frac{\Gamma(\sfv_i)^{m_1}}{\Gamma(\gamma-\sfv_i)^{n_1}}
                              \frac{\Gamma(v_j)^{n_2}}{\Gamma(\gamma-v_j)^{m_2}},  \\
                              &\,\,\hskip 9.1cm 1\leq i\leq m, 1\leq j \leq n,\\
 D(\bx,\btau; \sfw, \sfv)_{ij} &=
 			   e^{(x_{j+n}+\tau_{j+n})\left(\frac{\gamma}{2} -\sfw_j\right) -(x_{j+n}+\tau_{i+n})\left(\frac{\gamma}{2}-\sfv_i\right)} \\
 			    & \hskip 2.5 cm \times u_1^{\sfw_j-\sfv_i}\, 
                            \left(\frac{\Gamma(\gamma-\sfw_j)}{\Gamma(\gamma-\sfv_i)}\right)^{n_1}
                            \left(\frac{\Gamma(\sfv_i)}{\Gamma(\sfw_j)}\right)^{m_1}, \hskip 1.8cm 1\leq i,j\leq m.
 \end{align*}
 \end{proposition}
 \begin{proof}
 We begin with the formula obtained in Proposition \ref{prop:joint} and use the identity $z\Gamma(z)=\Gamma(1+z)$ to write the product 
 \begin{align*}
 &\prod_{k=1}^n\prod_{\ell=1}^m 
  \frac{\Gamma(\gamma-\sfw_\ell-w_k)\Gamma(\gamma-\sfv_\ell-v_k)}{\Gamma(\gamma-\sfw_\ell-v_k\,) \,\Gamma(\gamma-\sfv_\ell-w_k)}\\
=& \prod_{k=1}^n \prod_{\ell=1}^m 
    \frac{(\gamma-\sfw_\ell-v_k\,) \,(\gamma-\sfv_\ell-w_k)}{(\gamma-\sfw_\ell-w_k)(\gamma-\sfv_\ell-v_k)}
\prod_{k=1}^n\prod_{\ell=1}^m 
        \frac{\Gamma(1+\gamma-\sfw_\ell-w_k)\Gamma(1+\gamma-\sfv_\ell-v_k)}{\Gamma(1+\gamma-\sfw_\ell-v_k\,) \,\Gamma(1+\gamma-\sfv_\ell-w_k)}.
 \end{align*}
 Moreover, we write the term
 \begin{align*}
   \prod_{k=1}^n  \frac{\pi }{\sin(\pi(v_k-w_k))} 
 &=\prod_{k=1}^n  \frac{1}{ (v_k-w_k)}  \prod_{k=1}^n  \frac{\pi (v_k-w_k) }{\sin(\pi(v_k-w_k))} \\
 &=(-1)^n \prod_{k=1}^n  \int_{0}^\infty  e^{ \tau_k (v_k-w_k)} \dd\tau_k \,\, \prod_{k=1}^n  \frac{\pi (v_k-w_k) }{\sin(\pi(v_k-w_k))} ,
 \end{align*}
 and we also write the analogous formula for $\prod_{\ell=1}^m \pi/\sin(\pi(\sfv_\ell-\sfw_\ell))$, by replacing in the above formula $v_k,w_k$
 by $\sfv_\ell,\sfw_\ell$ and $\tau_k$ by $\tau'_\ell$.
 Then we use Lemma \ref{lem:blockCauchy} to obtain a block Cauchy determinant, to which we distribute into its first $n$ rows the
 terms 
 \begin{align*}
 u_2^{w_k-\frac{\gamma}{2}} \,e^{-\tau_k \left(w_k-\frac{\gamma}{2}\right)} \,
 \frac{\Gamma(\gamma-w_k)^{m_2}}{\Gamma(w_k)^{n_2}},  \hskip 1cm \text{for} \hskip 0.3cm k=1,...,n,
 \end{align*} 
 into the first $n$ columns the terms 
 \[u_2^{-(v_k-\frac{\gamma}{2})}e^{\tau_k(v_k-\frac{\gamma}{2})}\frac{\Gamma(v_k)^{n_2}}{\Gamma(\gamma-v_k)^{m_2}},
 \hskip 1.3cm \text{for} \hskip 0.3cm k=1,...,n,
 \]
  into rows $n+1$ to $n+m$ the terms 
 \[u_1^{-(\sfv_\ell-\frac{\gamma}{2})}e^{\tau_{\ell+n}(\sfv_\ell-\gamma/2)} \frac{\Gamma(\sfv_\ell)^{m_1}}{\Gamma(\gamma-\sfv_\ell)^{n_1}},
 \hskip 1.2cm \text{for} \hskip 0.3cm \ell=1,...,m,
 \]
 and into columns $n+1$ to $n+m$ the terms
 \begin{align*}
  u_1^{\sfw_\ell-\frac{\gamma}{2}} e^{-\tau_{\ell+n} \left(\sfw_\ell-\frac{\gamma}{2}\right)}
 \frac{\Gamma(\gamma-\sfw_\ell)^{n_1}}{\Gamma(\sfw_\ell)^{m_1}}, \hskip 1.5cm \text{for} \hskip 0.3cm \ell=1,...,m.
  \end{align*}  
   \end{proof}
    We will further manipulate the formula obtained in Proposition 
   \ref{lem:joint_conv} in order to bring it in a form suitable for asymptotics. Recall the definition of the functions $G,F$ from \eqref{FG}
   and denote by 
     \begin{align}
  	\tilde G(z) = G(z)-G(\gamma/2) \quad \text{and} \quad
  	\tilde F(z) = F(z)-F(\gamma/2). \notag
  \end{align}
  In the block determinant that appears in Proposition \ref{lem:joint_conv} we multiply each column between $1$  and $n$  by the number  $\exp\big(t_2F(\gamma/2)\big)$ and each row between $n+1$ and $m+n$ by the number $\exp\big(t_1F(\gamma/2)\big)$. Moreover, we multiply each column between $n+1$ and $m+n$  by the number  $\exp\big(-t_1F(\gamma/2)\big)$ and each row between $1$ and $n$ by the number  $\exp\big(-t_2F(\gamma/2)\big)$.
  This operation does not change the value of the determinant because if these number are pulled outside the determinant by mulitlinearity, they produce a pre factor
   $\exp\big((nt_2+mt_1)F(\gamma/2)-(nt_2+mt_1)F(\gamma/2)\big)$ , which is obviously equal to one.
  On the other hand it provides a useful centering of the entries. Using the particular choice for values of $u_1,u_2$ as in \eqref{u1u2} and for the lattice points  
  $(m_1,n_1),(m_2,n_2)$ as in \eqref{pointN2/3} and making a straightforward calculation employing Fubini, we have for fixed $m,n$ that
  we can arrive at formula 
   \begin{align}\label{eq:mod_block}
  	&  I_{m,n}^{(N)}= \\
	&=  \frac{ (-1)^{m+n}}{n! m!(2\pi\iota)^{m+n}} 
	 \int_{\bbR_{>0}^{n+m}} \dd \btau
	 \int_{\bbR_{>0}^{n+m}} \dd \bx
  	\int \dd {\bf v} \dd {\bf w} \int \dd \sfv \dd\sfw
  	\,\, \det\left(
  	\begin{array}{cc}
  		\tilde{A}(\bx,\btau; {\bf w},{\bf v}) & \tilde{B}(\bx,\btau; {\bf w}, \sfw)\\
  		\tilde{C}(\bx,\btau; {\bf v}, \sfv) & \tilde{D}(\bx,\btau; \sfw, \sfv) 
  	\end{array}
  	\right)\notag\\
	  	& \qquad\quad\times   \prod_{k=1}^n\frac{\pi(v_k-w_k)}{\sin(\pi(v_k-w_k))}  \prod_{\ell=1}^m\frac{\pi(\sfv_\ell-\sfw_\ell)}{\sin(\pi(\sfv_k-\sfw_k))}
  	\prod_{k=1}^n\prod_{\ell=1}^m 
  	\frac{\Gamma(1+\gamma-\sfw_\ell-w_k)\Gamma(1+\gamma-\sfv_\ell-v_k)}{\Gamma(1+\gamma-\sfw_\ell-v_k\,) \,\Gamma(1+\gamma-\sfv_\ell-w_k)} \notag
  \end{align}
  where  
  \begin{align*}
  	\tilde A(\bx,\btau; {\bf w},{\bf v})_{ij}    &=\exp\Bigg(N\big(\tilde G(v_j)-\tilde G(w_i)\big)+t_2N^{2/3}\big(\tilde F(w_i)-\tilde F(v_j)\big)+r_2N^{1/3}\big(v_j-w_i\big)+\\
  	& \qquad\qquad\qquad  + \big(x_j+\tau_i\big)\left(\frac{\gamma}{2} -w_i\right) -(x_j+\tau_j)\left(\frac{\gamma}{2}-v_j\right)\Bigg),\\
	& \hskip 8.6cm \text{for} \hskip 0.5cm 1\leq i,j\leq n,\\
  	\tilde B(\bx,\btau; {\bf w}, \sfw)_{ij} &=- \exp\Bigg(- N\big( \tilde G(\sfw_j)+\tilde G(w_i) \big)
	 +t_2N^{2/3} \tilde F(w_i)+t_1N^{2/3}\tilde F(\sfw_j) \\
	&\qquad\quad -r_2N^{1/3} \Big(w_i-\frac{\gamma}{2}\Big)-r_1N^{1/3} \Big(\sfw_j-\frac{\gamma}{2}\Big) \\
  	&\qquad\qquad + \big(-x_{j+n}+\tau_i\big)\left(\frac{\gamma}{2}-w_i\right)
	 + \big(-x_{j+n}+\tau_{j+n}\big)\left(\frac{\gamma}{2}-\sfw_j\right)\Bigg),\\
	 & \hskip 8.5cm \text{for}\hskip 0.5cm 1\leq i\leq n, 1\leq j \leq m,
  \end{align*}
  \begin{align*}
  	\tilde C(\bx,\btau; {\bf v}, \sfv)_{ij}   &=\,\,\exp\Bigg(N\big(\tilde G(v_j) + \tilde G(\sfv_i) \big)  
	  -t_2N^{2/3} \tilde F(v_j)-t_1N^{2/3} \tilde F(\sfv_i) \\
	&\qquad\quad +r_2N^{1/3} \Big(v_j-\frac{\gamma}{2}\Big)+r_1N^{1/3} \Big(\sfv_i-\frac{\gamma}{2} \Big)  \\
	&\qquad\qquad -\big(x_j+\tau_j\big)\left(\frac{\gamma}{2} -v_j\right)
  	  -\big(x_j+\tau_{i+n}\big)\left(\frac{\gamma}{2}-\sfv_i\right)\Bigg), \\
	  & \hskip 8.5cm   \text{for} \hskip 0.5cm 	1\leq i\leq m, 1\leq j \leq n,\\
  	\tilde D(\bx,\btau; \sfw, \sfv)_{ij} &=\exp\Bigg(N\big(\tilde G(\sfv_i)-\tilde G(\sfw_j)\big)+t_1N^{2/3}\big(\tilde F(\sfw_j)-\tilde F(\sfv_i)\big)
	+r_1N^{1/3}\big(\sfv_i-\sfw_j\big)+
  	\\
  	& \qquad \qquad \qquad +\big(x_{j+n}+\tau_{j+n}\big)\left(\frac{\gamma}{2} -\sfw_j\right) -\big(x_{j+n}+\tau_{i+n}\big)\left(\frac{\gamma}{2}-\sfv_i\right)\Bigg), \\
	& \hskip 8.5cm \text{for} \hskip 0.5cm 1\leq i,j\leq m.
  \end{align*}
  The use of Fubini is useful, in order to avoid crossing the poles at $\sfw_\ell+w_k=\gamma$ in the contour deformation that follows.
 Using the asymptotics \eqref{gamma_asympto}, we can justify the Fubini for fixed $m,n$ (or for values that are a sufficiently small multiple of $N^{2/3}$), since the integrand in \eqref{eq:mod_block}
 is bounded by
 \begin{align}\label{eq:Fub_blue}
 \exp\Big(
 &-\frac{\pi}{2}\sum_{\substack{ 1\leq k \leq n \\ 1\leq \ell \leq m}} |\sfw_\ell+w_k | +\frac{\pi n}{2}  \sum_{1\leq \ell \leq m } |\sfw_\ell| +\frac{\pi m}{2} \sum_{1\leq k \leq n} |w_k| \\
 &\hskip 3cm
 - \frac{\pi (n_1-m_1)}{2} \sum_{1\leq \ell \leq m} |\sfw_\ell|
 -\frac{\pi (m_2-n_2)}{2} \sum_{1\leq k \leq n} |w_k| 
 \Big),\notag
 \end{align}
   which is integrable since $m_2-n_2\,, \,n_1-m_1 \approx N^{2/3}$.
   \vskip 2mm
The main asymptotic behaviour of each fixed summand in \eqref{eq:mod_block} is captured by the determinant. The study of the asymptotic behaviour of the
terms $\tilde A, \tilde B, \tilde C, \tilde D$ in there is done via the steepest decent method.  
  The basic idea of this method is to find the critical points for the integrand and then deform the contour of integration so that it passes through or
   close to the critical point and away from the critical point the integrands decay fast.
    The contours over which the  ${\bf v}, {\bf w}, \sfv, \sfw$ variables originally integrate are depicted in Figure \ref{fig:initial_cont}.
     These contours need to be deformed 
  to the steepest descent contours, which are depicted in Figure \ref{fig:steep_cont}. 
 The contour $C_w$ is symmetric over the real axis. Its upper part consists of a ray from $\gamma/2+\eta N^{-1/3}$, with fixed, small $\eta>0$, leaving at angle $\pi/3$ until the point $\gamma/2+\eta N^{-1/3}+\gamma e^{\frac{\pi \iota}{3}}$ and from this point it goes vertically to infinity. The $C_v$ contour is also chosen to be symmetric over the real axis. 
 It departs from $\gamma/2$ at an angle $2\pi /3$, as a straight line until the point $\gamma e^{2\pi \iota/3}$, where it follows the circular arc $\gamma e^{\iota t}, 2\pi/3\leq t \leq 4\pi/3$ and, finally, from there it returns to $\gamma/2$ via a straight line.
  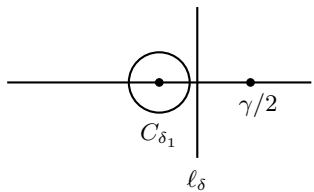
\begin{figure}[t]
  		\begin{center}
  		\begin{tikzpicture}
  		\draw[thick] (0,0) circle (0.4);
  		\draw[thick] (-2,0)--(2,0);
  		\draw[thick] (0.5,-1)--(0.5,1);
  		%\draw[thick] (0.9,-1)--(0.9,1);
  		\node at (1.3,-0.3){\small{$\gamma/2$}};
  		\node at (0, -0.7){\small{$C_{\delta_1}$}};
  		\node at (0.5,-1.3){\small{$\ell_{\delta}$}};
  		%\node at (1.0,-1.3){\small{$\ell_{\delta_3}$}};
  		\draw [fill] (1.2,0) circle [radius=0.05];
  		\draw [fill] (0,0) circle [radius=0.05];
  		\end{tikzpicture}
  	\end{center}
  	\caption{Initial contours.   }\label{fig:initial_cont}
  	\end{figure}
  	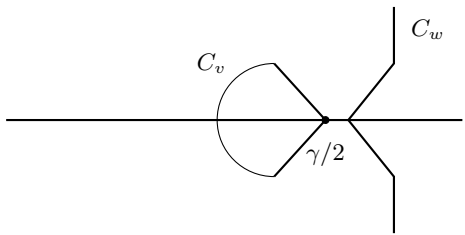
\begin{figure}[t]
  		\begin{center}
  		\begin{tikzpicture}[scale=1.5]
  		\draw[thick] (-2,0)--(2,0);
  		\draw[thick] (1.4,-1)--(1.4,-0.5)--(1,0)--(1.4,0.5)--(1.4,1);
  		\draw[thick] (0.35,0.5)--(0.8,0)--(0.35,-0.5);
  		  \draw (0.35,0.5) arc (90:270:0.5);
  		  \draw [fill] (0.8,0) circle [radius=0.03];
  		  \node at (0.8, -0.3){\small{$\gamma/2$}};
  		  \node at (-0.2,0.5){\small{$C_v$}};
  		  \node at (1.7,0.8){\small{$C_w$}};
  		\end{tikzpicture}
  	\end{center}
  	\caption{Steepest decent contours } \label{fig:steep_cont}
  	\end{figure}
In order to motivate the choice of the contours $C_v,C_w$, we need to identify the behaviour of the functions $\tilde G, \tilde F$
 around the (critical) point $\gamma/2$ and away from it. We do this by employing a Taylor expansion around $\gamma/2$.
  To this end, we need to record	
  \begin{align*}
  	G'(z)&=\Psi(z)+\Psi(\gamma-z)+f_{\gamma}\ ,\\
  	G''(z)&=\Psi'(z)-\Psi'(\gamma-z)\ ,\\
  	G'''(z)&=\Psi''(z)+\Psi''(\gamma-z)\ .\\
  \end{align*} 
  By the choice of $f_{\gamma}=-2\Psi(\gamma/2)$ we have that $G'(\gamma/2)=G''(\gamma/2)=0$, while  $G'''(\gamma/2)\neq 0$.
  Consequently, around the critical point $\gamma/2$ the function $\tilde G(z)$ behaves as 
  \begin{align}\label{Gtaylor}
  	\tilde G(z) =G(z)-G(\gamma/2)= \frac{G'''(\gamma/2)}{6} (z-\gamma/2)^3 + o\big((z-\gamma/2)^3\big)
  \end{align}
   Moreover, $\gamma/2$ is also the critical point of the function $F(z)$ because
  \begin{align*}
  	F'(\gamma/2)&=\Psi(\gamma/2)-\Psi(\gamma/2)=0,
  	% F''(z)&=\Psi'(z)+\Psi'(\gamma-z)\\
  \end{align*} 
  and around it we have
  \begin{align}\label{Ftaylor}
  	\tilde F(z)=F(z)-F(\gamma/2)= \frac{F''(\gamma/2)}{2} (z-\gamma/2)^2 +o\big((z-\gamma/2)^2\big) ,
  \end{align}
  The cubic behaviour of the function $\tilde G$ as in \eqref{Gtaylor} shows that along a straight line departing at an angle $\pi/3$ from $\gamma/2$,
  $\tilde G$ is strictly negative with the value zero at $\gamma/2$. Similar estimates hold along straight lines departing at an angle $2\pi/3$.
  \vskip 2mm
  The behaviour away from the critical point $\gamma/2$ is described by the following lemmas. 
 The first one is due to \cite{BCR13} Lemmas 2.4, 2.5, which we recall :
  	\begin{lemma}\label{lem:BCR}
  		There exists $\gamma^*>0$ such that for all $\gamma<\gamma^*$ the following two facts hold:
		\begin{itemize}
		\item[(1)] There exists a constant $c_1>0$ such that for all $v$ along the straight line segments of $C_v$
		\begin{equation*}
			\Re[G(v)-G(\gamma/2)]\leq \Re[-c_1\gamma^3(v-\gamma/2)^3]\ .
		\end{equation*}
		
		\item[(2)]There exists a constant $c_2>0$ such that for all $w$ along the contour $C_w$ at the distance less than $\gamma$ from $\gamma/2$
		\begin{equation*}
			\Re[G(w)-G(\gamma/2)]\geq \Re[-c_2\gamma^3(w-\gamma/2  )^3]\ .
		\end{equation*}
		
		\item[(3)] There exists $c>0$, such that for all $v$ along the circular part of the contour $C_v$
		\begin{equation*}
		\Re[G(v)-G(\gamma/2)]\leq -c\ .
		\end{equation*}
		\end{itemize}
  	\end{lemma}
  The cubic behaviour in \eqref{Gtaylor} suggests rescaling around $\gamma/2$ by change of variables
  \begin{align}\label{variable_change}
  	\tilde v_i&= N^{1/3}(v_i-\gamma/2)\ ,\qquad\qquad \tilde w_i= N^{1/3}(w_i-\gamma/2), \notag\\
  	\tilde \sfv_i&= N^{1/3}(\sfv_i-\gamma/2)\ ,\qquad\qquad \tilde \sfw_i= N^{1/3}(\sfw_i-\gamma/2), \\
  	\tilde x_i&= N^{-1/3}x_i\ ,\qquad\qquad\qquad\quad \tilde \tau_i=N^{-1/3}\tau_i. \notag
  \end{align}
\begin{proposition}\label{prop:dom_conv}
Recall the term $I^{(N)}_{m,n}$ in \eqref{eq:mod_block}. Setting the parameters $(m_1,n_1), (m_2,n_2)$ as in \eqref{pointN2/3} and the 
parameters $u_1,u_2$ as in \eqref{u1u2} (depending all on $N$), then as $N$ tends to infinity $I^{(N)}_{m,n}$ converges for any fixed $m,n$ to
  \begin{align}
  	\label{params: gen_term}
  	I_{m,n}= \frac{(-1)^{n+m}}{n! m!(2\pi\iota)^{m+n}} \int_{\bbR_{>0}^{n+m}}  \dd \tilde\btau
  	\int \dd \tilde{\bf v} \dd \tilde{\bf w} \int \dd \tilde\sfv \dd\tilde\sfw
  	\int_{\bbR_{>0}^{n+m}} \dd \tilde\bx 
  	\,\, \det\left(
  	\begin{array}{cc}
  		A^*(\tilde\bx,\tilde\btau; \tilde{\bf w},\tilde{\bf v}) & B^*(\tilde\bx,\tilde\btau; \tilde{\bf w}, \tilde\sfw)\\
  		C^*(\tilde\bx,\tilde\btau; \tilde{\bf v}, \tilde\sfv) & D^*(\tilde\bx,\tilde\btau; \tilde\sfw, \tilde\sfv) 
  	\end{array}
  	\right),
  \end{align}
  where the variables $\tilde v$ and $\tilde \sfv$ integrate along the contour $e^{-2\pi \iota/3}\mathbb{R}_{>0}\cup e^{+2\pi \iota/3}\mathbb{R}_{>0}$ and $\tilde w$ and $\tilde \sfw$ integrate along the contour $\{e^{-\pi \iota/3}\mathbb{R}_{>0}+\eta\}\cup \{e^{+\pi \iota/3}\mathbb{R}_{>0}+\eta\}$ for any horizontal shift $\eta>0$. The  block matrix in (\ref{params: gen_term}) is given by 
    \begin{align*}
  	A^*(\tilde \bx,\tilde\btau; \tilde{\bf w},\tilde{\bf v})_{ij}    &=\exp\Big(\frac{G'''(\gamma/2)}{6}(\tilde v_j^3- \tilde w_i^3)+t_2\frac{F''(\gamma/2)}{2}(\tilde w_i^2-\tilde v_j^2)\\
	&\hskip 4cm  +r_2(\tilde v_j-\tilde w_i)-(\tilde x_j+\tilde\tau_i)\tilde w_i +(\tilde x_j+\tilde\tau_j)\tilde v_j\Big),
  	\\
  	B^*(\tilde\bx,\tilde\btau; \tilde{\bf w},\tilde \sfw)_{ij} &=- \exp\Big(-\frac{G'''(\gamma/2)}{6}(\tilde\sfw_j^3+\tilde w_i^3)+t_2\frac{F''(\gamma/2)}{2}\tilde w_i^2+t_1\frac{F''(\gamma/2)}{2}  \tilde\sfw_j^2 \\
	&\hskip 4cm-r_2\tilde w_i-r_1\tilde\sfw_j-(-\tilde x_{j+n}+\tilde\tau_i)\tilde w_i-
  	(-\tilde x_{j+n}+\tilde\tau_{j+n})\tilde\sfw_j\Big) ,
  	\\
  	C^*(\tilde\bx,\tilde\btau; \tilde{\bf v}, \tilde\sfv)_{ij}   &=\exp\Big(\frac{G'''(\gamma/2)}{6}(\tilde v_j^3+\tilde\sfv_i^3)-t_2\frac{F''(\gamma)}{2}\tilde v_j^2-t_1
	\frac{F''(\gamma/2)}{2}\tilde\sfv_i^2  \\
	&\hskip 4cm +r_2\tilde v_j+r_1\tilde\sfv_i+(\tilde x_j+\tilde\tau_j)v_j +(\tilde x_j+\tilde\tau_{i+n})\tilde\sfv_i\Big) ,
  	\\
  	D^*(\tilde\bx,\tilde\btau; \tilde\sfw, \tilde\sfv)_{ij} &=\exp \Big(\frac{G'''(\gamma/2)}{6}(\tilde\sfv_i^3-\tilde\sfw_j^3)+t_1\frac{F''(\gamma/2)}{2}(\tilde\sfw_j^2-\tilde\sfv_i^2)\\
	&\hskip 4cm +r_1(\tilde\sfv_i-\tilde\sfw_j)-(\tilde x_{j+n}+\tilde\tau_{j+n})\tilde\sfw_j +(\tilde x_{j+n}+\tilde\tau_{i+n})\tilde\sfv_i\Big).
  	\\
  \end{align*}  
  \begin{proof}
  Let us start with the formula for $I_{m,n}^{(N)}$ provided by \eqref{eq:mod_block} and denote
  \begin{align}\label{CT}
  \sfC \sfT_{m,n}:=
  \prod_{k=1}^n\prod_{\ell=1}^m 
  	\frac{\Gamma(1+\gamma-\sfw_\ell-w_k)\Gamma(1+\gamma-\sfv_\ell-v_k)}{\Gamma(1+\gamma-\sfw_\ell-v_k\,) \,\Gamma(1+\gamma-\sfv_\ell-w_k)}
  \end{align}
  {\bf Step 1. (Contour deformation)}  We deform the contour $C_{\delta_1}$, on which $v$ and $\sfv$ are integrated, to contour $C_v$, as well as the contour on  which $w$ and $\sfw$ are integrated to contour $C_w$. Some care is required when deforming the $w,\sfw$ integrals. Let us look at it. For fixed values of $\btau$ and $\bx$, the only poles of the integrand (as a function of $v,w,\sfv,\sfw$) 
  are those due to the cross term $\sfC \sfT_{m,n}$ coming from the poles of the gamma functions $\Gamma(1+\gamma-\sfv_\ell-v_k)$ and 
  $\Gamma(1+\gamma-\sfw_\ell-w_k)$ when 
  \begin{align*}
  \gamma-\sfw_\ell-w_k = \big(\frac{\gamma}{2}-\sfw_\ell \big)+\big(\frac{\gamma}{2}-w_k \big) = -r-1, \qquad \text{for} \quad r=0,-1,-2,...
  \end{align*}
  and
  \begin{align*}
  \gamma-\sfv_\ell-v_k = \big(\frac{\gamma}{2}-\sfv_\ell \big)+\big(\frac{\gamma}{2}-v_k \big) = -r-1, \qquad \hskip 0.4cm \text{for} \quad r=0,-1,-2,...
  \end{align*}
  The deformation from $C_{\delta_1}$ to $C_v$ crosses no poles, as on both contours $\Re(\gamma-\sfv_\ell-v_k)\geq0$.
  
  Similarly, the deformation from  $\ell_{\delta}$ to $C_w$ 
  does not cross any poles, since on or on the left of the $C_w$ contour it holds that
  $\Re( \gamma-\sfw_\ell-w_k)  \geq -\gamma >-1$, for $\gamma$ sufficiently small.
   Finally, the deformation is justified by suitable decay bounds near infinity between the initial and final contours. 
    \vskip 2mm
  {\bf Step 2. (Localization around the critical point)} 
	Thanks to Lemma \ref{lem:BCR} and estimate \eqref{eq:Fub_blue}, the contribution to the contour integral away (e.g. outside a ball of radius $\gamma$ from the critical point $\gamma/2$) is exponentially negligible as $N$ tends to infinity. 
	 Then, again thanks to Lemma \ref{lem:BCR}, we can further localize the contour integrals in a window of size $N^{-1/3}$ around the critical point with a small cost $\epsilon$, uniformly in $N$.
	  We now do a  change of variables as in \eqref{variable_change} to rescale this window by $N^{1/3}$. By using the Taylor expansion around critical point \eqref{Gtaylor} and \eqref{Ftaylor}, 
	  we can easily  show that the contribution inside this ball is uniformly bounded by a constant, which will allow to pass in the limit in the quantities below.	 
  As $N\to\infty$, we have (for fixed $m,n$) that
  \begin{align*}
  	\prod_{k=1}^n\prod_{\ell=1}^m 
  	\frac{\Gamma\big (1-N^{-1/3}(\tilde\sfw_\ell+\tilde w_k)\big)\Gamma\big(1-N^{-1/3}(\tilde\sfv_\ell+\tilde v_k)\big)}{\Gamma\big(1-N^{-1/3}(\tilde\sfw_\ell+\tilde v_k\,)\big) \,\Gamma\big(1-N^{-1/3}(\tilde\sfv_\ell+\tilde w_k)\big)}\to 1\ ,
  \end{align*}
  and moreover,
  \begin{equation*} 
\frac{\pi N^{-1/3}(\tilde v_k-\tilde w_k)}{\sin\big( \pi N^{-1/3}(\tilde v_k-\tilde w_k)\big)}\to 1
\quad \text{and}\quad
\frac{\pi N^{-1/3}(\tilde \sfv_k-\tilde \sfw_k)}{\sin\big( \pi N^{-1/3}(\tilde \sfv_k-\tilde \sfw_k)\big)}\to 1.
  \end{equation*}
   Using the Taylor expansions \eqref{Gtaylor} and \eqref{Ftaylor} for $\tilde G$ and $\tilde F$, we have
   that in this limit and with the change of variables  \eqref{variable_change}, $\tilde A, \tilde B, \tilde C, \tilde D$ converge to $A^*, B^*, C^*, D^*$, respectively, while
    the diagonal parts of $C_v, C_w$ are replaced by contours
    $e^{-2\pi \iota/3}\mathbb{R}_{>0}\cup e^{+2\pi \iota/3}\mathbb{R}_{>0}$ and $\{e^{-\pi \iota/3}\mathbb{R}_{>0}+\delta\}\cup \{e^{+\pi \iota/3}\mathbb{R}_{>0}+\delta\}$, respectively.
  \end{proof}

\end{proposition}

By the multilinearity of the determinant, it follows that $I_{m,n}$ can be written as 
  \begin{equation}\label{for:expansion}
  I_{m,n} =
  	 \frac{(-1)^{n+m}}{n! m!(2\pi\iota)^{m+n}} 
  	\int_{\bbR_{>0}^{n+m}} \dd \btau \det \left(
  	\begin{array}{cc}
  		\big(A'(\tau_i,\tau_j)\big)_{n\times n} & \big(B'(\tau_i, \tau_{j+n})\big)_{n\times m}\\
		&\\
  		\big(C'(\tau_{i+n}, \tau_j)\big)_{m\times n} & \big(D'(\tau_{i+n}, \tau_{j+n})\big)_{m\times m} 
  	\end{array}
  	\right)\ ,
  \end{equation} 
  where
  \begin{align*}
  	A'(\tau_i,\tau_j)    &=\int \dd \tilde v \dd \tilde  w 
  	\int_{\bbR_{>0}} \dd\tilde x \  A^*(\tilde x,\tilde\tau; \tilde{ w},\tilde v)_{ij},\\
  	B'(\tau_i, \tau_{j+n}) &=\int \dd \tilde{ w} \dd\tilde\sfw
  	\int_{\bbR_{>0}} \dd\tilde x  \ B^*(\tilde x,\tilde\tau; \tilde{ w},\tilde \sfw)_{ij},\\
  	C'(\tau_{i+n}, \tau_j)  &=\int \dd \tilde{ v}  \dd \tilde\sfv 
  	\int_{\bbR_{>0}} \dd \tilde x  \ C^*(\tilde x,\tilde \tau; \tilde{v}, \tilde\sfv)_{ij},\\
  	D'(\tau_{i+n}, \tau_{j+n})  &= \int \dd \tilde\sfv \dd\tilde\sfw
  	\int_{\bbR_{>0}} \dd \tilde x \ D^*(\tilde x,\tilde\tau; \tilde\sfw, \tilde\sfv)_{ij} .
  \end{align*}
   Let us, now, simplify the block determinant by doing the following change of variables
  \begin{align*}
  	\tilde{\bf v}&\mapsto \Big(-\frac{2}{G'''(\gamma/2)}\Big)^{1/3} \tilde{\bf v} -t_2 \frac{F''(\gamma/2)}{G'''(\gamma/2)}, 
	\quad\quad \tilde{\bf w}\mapsto \Big(-\frac{2}{G'''(\gamma/2)}\Big)^{1/3} \tilde{\bf w} -t_2 \frac{ F''(\gamma/2)}{G'''(\gamma/2)},\\
  	\tilde{\sfv}&\mapsto \Big(-\frac{2}{G'''(\gamma/2)}\Big)^{1/3}\tilde{\sfv} -t_1\frac{F''(\gamma/2)}{G'''(\gamma/2)},
	\quad\quad \tilde{\sfw}\mapsto \Big(-\frac{2}{G'''(\gamma/2)}\Big)^{1/3}\tilde{\sfw} -t_1\frac{F''(\gamma/2)}{G(\gamma/2)},\\
  	\tilde \bx&\mapsto \Big(-\frac{G'''(\gamma/2)}{2}\Big)^{1/3} \tilde \bx,
	 \qquad\qquad\quad\quad \qquad\tilde \btau\mapsto \Big(-\frac{G'''(\gamma/2)}{2}\Big)^{1/3} \tilde \btau \ .
  \end{align*}
  Define  $c_1^{\gamma}=(-\frac{G'''(\gamma/2)}{2})^{-1/3} $,  $c_2^{\gamma}=-\frac{c_1^{\gamma} (F''(\gamma/2))^2}{2 G'''(\gamma/2)}$ and $c_3^\gamma=-\frac{F''(\gamma/2)}{G'''(\gamma/2)}$.
   By a long but fairly straightforward computation using the above change of variables and the integral representation of the Airy function \eqref{def:Airy}, we obtain that
  \begin{align*}
	A'(\tau_i,\tau_j)    &= \int_{0}^{\infty} \dd x \,  Ai(c_1^{\gamma}r_2+c_2^{\gamma}t_2^2+x+\tau_i)
   									    \,Ai(c_1^{\gamma}r_2+c_2^{\gamma}t_2^2+ x+\tau_j)\ ,\\
  	B'(\tau_i, \tau_{j+n}) &=-\int_{0}^{\infty} \dd x \, e^{- c_3^{\gamma} x (t_2+t_1)}Ai(c_1^{\gamma}r_1+c_2^{\gamma}t_1^2- x+\tau_i)Ai(c_1^{\gamma}r_2+c_2^{\gamma}t_2^2- x+\tau_{j+n})\ ,\\
  	C'(\tau_{i+n}, \tau_j)  &=\int_{0}^{\infty} \dd x  \,e^{-c_3^{\gamma} x (t_2+t_1)} \, Ai(c_1^{\gamma}r_1+c_2^{\gamma}t_1^2+ x+\tau_{i+n})
															      \, Ai(c_1^{\gamma}r_2+c_2^{\gamma}t_2^2+ x+\tau_{j})\ ,\\
  	D'(\tau_{i+n}, \tau_{j+n})  &= \int_{0}^{\infty} \dd x \, Ai(c_1^{\gamma}r_1+c_2^{\gamma}t_1^2+x+\tau_{i+n})
										   \,  Ai(c_1^{\gamma}r_1+c_2^{\gamma}t_1^2+ x+\tau_{j+n} )\ .
  \end{align*}
The only tricky part of the computation that leads to the above expression is to take care of some
 multiplicative factors, which are independent of the variables $\tilde{\bf w},\tilde{\bf v},\tilde \sfw, \tilde \sfv, \tilde{\bx} $ and appear in the entries of the blocks 
$A'(\tau_i,\tau_j), B'(\tau_i, \tau_{j+n}), C'(\tau_{i+n}, \tau_j), D'(\tau_{i+n}, \tau_{j+n}) $. Denoting these terms by $\exp (\mathcal{C}_{ij})$, it turns out that $\mathcal{C}_{ij}$
 are equal to
\begin{align*}
&\ind_{\{1\leq i \leq n\}}
\left\{-\left( r_2-t_2^2\frac{(F''(\gamma/2))^2}{3G'''(\gamma/2)} \right) \frac{F''(\gamma/2)}{G'''(\gamma/2)} t_2 -\tau_i \frac{F''(\gamma/2)}{G'''(\gamma/2)} t_2
\right\}\\
&\\
-&\,\ind_{\{n+1\leq i \leq n+m\}}
\left\{-\left( r_1-t_1^2\frac{(F''(\gamma/2))^2}{3G'''(\gamma/2)} \right) \frac{F''(\gamma/2)}{G'''(\gamma/2)} t_1 -\tau_i \frac{F''(\gamma/2)}{G'''(\gamma/2)} t_1
\right\}\\
&\\
-&\,\ind_{\{1\leq j \leq n\}}
\left\{-\left( r_2-t_2^2\frac{(F''(\gamma/2))^2}{3G'''(\gamma/2)} \right) \frac{F''(\gamma/2)}{G'''(\gamma/2)} t_2 -\tau_j \frac{F''(\gamma/2)}{G'''(\gamma/2)} t_2
\right\}\\
&\\
+&\,\ind_{\{n+1\leq j \leq n+m\}}
\left\{-\left( r_1-t_1^2\frac{(F''(\gamma/2))^2}{3G'''(\gamma/2)} \right) \frac{F''(\gamma/2)}{G'''(\gamma/2)} t_1 -\tau_j \frac{F''(\gamma/2)}{G'''(\gamma/2)} t_1
\right\}.
\end{align*}
These terms can be pulled out of the block determinant, by multi linearity, and they then cancel out.
    	  	  
		  An immediate change of variables brings us to expression \eqref{2point_road}, thus bringing us to the end of the road map.
		  \vskip 2mm
		  {\bf Acknowledgement.} 
		  We thank Guillaume Barraquand for useful discussions and Ivan Corwin for comments on the draft. 
		  We are also grateful to the anonymous referee for a very careful reading and comments which improved the exposition of the article.
		  NZ acknowledges support from EPRSC through grant EP/L012154/1.

\end{document}